\documentclass[a4paper,twopage,reqno,12pt]{amsart}
\usepackage[top=30mm,right=30mm,bottom=30mm,left=30mm]{geometry}

\usepackage{tabularx}
\usepackage{graphicx}
\usepackage{amsmath}
\usepackage{amssymb}
\usepackage{amsfonts}
\usepackage{amsthm}
\usepackage{amstext}
\usepackage{amsbsy}
\usepackage{amsopn}
\usepackage{amscd}
\usepackage{enumerate}
\usepackage{color}
\usepackage[colorlinks]{hyperref}
\usepackage[hyperpageref]{backref}
\usepackage{multirow}
\usepackage{longtable}
\usepackage{float}
\usepackage{lscape}
\usepackage{pdflscape}
\usepackage{url}

\newtheorem{theorem}{Theorem}[section]
\newtheorem{corollary}[theorem]{Corollary}
\newtheorem{lemma}[theorem]{Lemma}
\newtheorem{proposition}[theorem]{Proposition}

\theoremstyle{definition}

\newtheorem{example}[theorem]{Example}

\numberwithin{equation}{section}

\renewcommand{\leq}{\leqslant}
\renewcommand{\geq}{\geqslant}

\begin{document}
\title[flag-transitive $2$-designs of affine type]{Classification of the non-trivial $2$-$(k^{2},k,\lambda )$ designs,
with $\lambda \mid k$, admitting a flag-transitive automorphism group of affine type}

\author[]{ Alessandro Montinaro}

\address{Alessandro Montinaro, Dipartimento di Matematica e Fisica “E. De Giorgi”, University of Salento, Lecce, Italy}
\email{alessandro.montinaro@unisalento.it}

\subjclass[MSC 2020:]{05B05; 05B25; 20B25}%
\date{\today}%

\begin{abstract}
The pairs $(\mathcal{D},G)$, where $\mathcal{D}$ is a non-trivial $2$-$(k^{2},k,\lambda )$ design, with $\lambda \mid k$, and $G$ is a flag-transitive automorphism group of $\mathcal{D}$ of affine type such that $G \nleq A \Gamma L_{1}(k^{2})$, are classified.
\end{abstract}

\maketitle

\section{Introduction, Main Result and Examples}

A $2$-$(v,k,\lambda )$ \emph{design} $\mathcal{D}$ is a pair $(\mathcal{P},%
\mathcal{B})$ with a set $\mathcal{P}$ of $v$ points and a set $\mathcal{B}$
of blocks such that each block is a $k$-subset of $\mathcal{P}$ and each two
distinct points are contained in $\lambda $ blocks. We say $\mathcal{D}$ is 
\emph{non-trivial} if $2<k<v$. All $2$-$(v,k,\lambda )$ designs in this paper
are assumed to be non-trivial.

An automorphism of $\mathcal{D}$ is a
permutation of the point set which preserves the block set. The set of all
automorphisms of $\mathcal{D}$ with the composition of permutations forms a
group, denoted by $\mathrm{Aut(\mathcal{D})}$. For a subgroup $G$ of $%
\mathrm{Aut(\mathcal{D})}$, $G$ is said to be \emph{point-primitive} if $G$
acts primitively on $\mathcal{P}$, and said to be \emph{point-imprimitive}
otherwise. A \emph{flag} of $\mathcal{D}$ is a pair $(x,B)$ where $x$ is a point and $B$
is a block containing $x$. If $G\leq \mathrm{Aut(\mathcal{D})}$ acts
transitively on the set of flags of $\mathcal{D}$, then we say that $G$ is 
\emph{flag-transitive} and that $\mathcal{D}$ is a \emph{flag-transitive
design}.\ 

The $2$-$(v,k,\lambda )$ designs $\mathcal{D}$ admitting a flag-transitive
automorphism group $G$ have been widely studied by several authors. In 1990,
a classification of those with $\lambda =1$ and $G\nleq A\Gamma L_{1}(v)$
was announced by Buekenhout, Delandtsheer, Doyen, Kleidman, Liebeck and Saxl
in \cite{BDDKLS} and proven in \cite{BDD}, \cite{Da}, \cite{De0}, \cite{De}, 
\cite{Kle}, \cite{LiebF} and \cite{Saxl}. Since then a special attention was
given to the case $\lambda >1$. A classification of the flag-transitive $2$%
-designs with $\gcd (r,\lambda )=1$, $\lambda >1$ and $G\nleq A\Gamma
L_{1}(v)$, where $r$ is the replication number of $\mathcal{D}$, has been
announced by Alavi, Biliotti, Daneshkakh, Montinaro, Zhou and their
collaborators in \cite{glob} and proven in \cite{A}, \cite{A1}, \cite{ABD0}, 
\cite{ABD1}, \cite{ABD2},\cite{ABD}, \cite{BM}, \cite{BMR}, \cite{MBF}, \cite%
{TZ}, \cite{Zie}, \cite{ZD}, \cite{ZZ0}, \cite{ZZ1}, \cite{ZZ2}, \cite{ZW}
and \cite{ZGZ}. Moreover, recently the flag-transitive $2$-designs with $%
\lambda =2$ have been investigated by Devillers, Liang, Praeger and Xia in 
\cite{DLPX}, where it is shown that apart from the two known symmetric $2$-$%
(16,6,2)$ designs, $G$ is primitive of affine or almost simple type.
A classification is also provided when the socle of $G$ is isomorphic
to $PSL_{n}(q)$ with $n\geq 3$.

The investigation of the flag-transitive $2$-$(k^{2},k,\lambda )$ designs,
with $\lambda \mid k$, has been recently started in \cite{MF}. The reason of studying such $2$-designs is that they represent a natural generalization of the affine planes in
terms of parameters, and also because, it is shown in \cite{Monty} that, the blocks of imprimitivity of a family of flag-transitive, point-imprimitive symmetric $2$-designs investigated in \cite{PZ} have the structure of the $2$-designs analyzed here.
In \cite{MF} it is shown that, apart from the smallest Ree group, a
flag-transitive automorphism group $G$ of a $2$-$(k^{2},k,\lambda )$ design $%
\mathcal{D}$, with $\lambda \mid k$, is either an affine group or an almost
simple classical group. Moreover, a complete classification of $(\mathcal{D},G)$ is achieved in \cite{Mo1} when $G$ is an almost simple
classical group. The result contained in the present paper is a complete classification of $(\mathcal{D},G)$ when $G$ is
affine type and $G \nleq A \Gamma L_{1}(k^{2})$. More precisely, the following result is obtained: 

\medskip 

\begin{theorem}\label{main}
Let $\mathcal{D}$ be a $2$-$(k^{2},k,\lambda )$ design, with $\lambda \mid k$,
admitting a flag-transitive automorphism group $G=T:G_{0}$. Then one
of the following holds:

\begin{enumerate}
\item $\mathcal{D}$ is a $2$-$(q^{2},q,\lambda )$ design, the blocks are
subspaces of $AG_{2m}(p)$ and $G_{0}\leq \Gamma L_{1}(q)$.

\item $\mathcal{D}$ is the Desarguesian plane of order $q$ and $SL_{2}(q)\trianglelefteq
G_{0}\leq \Gamma L_{2}(q)$.

\item $\mathcal{D}$ is the L\"{u}neburg plane of order $2^{2h}$, $h$ odd,
and $Sz(2^{h})\trianglelefteq G_{0}\leq Sz(2^{h}):Z_{h}$.

\item $\mathcal{D}$ is the Hall plane of order $3^{2}$ and $%
SL_{2}(5)\leq G_{0}\leq (D_{8}\circ Q_{8}).S_{5}$.

\item $\mathcal{D}$ is the Hering plane of order $3^{3}$ and $%
G_{0}\cong SL_{2}(13)$.

\item $\mathcal{D}$ is a $2$-$(q^{2},q,q^{1/2})$ design, $q=p^{h}$ and $h$ even, and $SL_{2}(q)\trianglelefteq G_{0}\leq (Z_{q^{1/2}-1}\circ SL_{2}(q)).Z_{h}$.

\item $\mathcal{D}=(V,B^{G})$ is a $2$-$(q^{3},q^{3/2},q)$ design, $q=p^{h}$, $p$
odd and $h$ even, and $SU_{3}(q^{1/2})\trianglelefteq G_{0}\leq (Z_{q^{1/2}-1}\times
SU_{3}(q^{1/2})).Z_{h}$.

\item $\mathcal{D}$ is a $2$-$(q^{4},q^{2},q^{2})$ design, $q=p^{h}$, and $Sp_{4}(q)%
\trianglelefteq G_{0}\leq \Gamma Sp_{4}(q)$.

\item $\mathcal{D}$ is one of the three $2$-$(q^{4},q^{2},\lambda)$ designs, with $\lambda=q,q^{2}/2,q^{2}$ respectively, where $q=2^{h}$ and $h$ odd, $Sz(q)\trianglelefteq G_{0}\leq \left( Z_{q-1}\times
Sz(q)\right) .Z_{h}$.

\item $\mathcal{D}$ is a $2$-$(q^{6},q^{3},q^{3})$ design, $q=2^{h}$, and $G_{2}(q)\trianglelefteq G_{0}\leq (Z_{q-1}\times
G_{2}(q)):Z_{h}$.

\item $\mathcal{D}$ is a $2$-$(3^{4},3^{2},3)$-design and $SL_{2}(5)\trianglelefteq G_{0}\leq \left(\left\langle -1\right\rangle .S_{5}^{-}\right) :Z_{2}$.

\item $\mathcal{D}$ is one of the two $2$-$(2^{6},2^{3},2^{2})$-designs and $G_{0}$ is either one of the groups $3^{1+2}:Q_{8}$, $%
3^{1+2}:Z_{8}$ or $3^{1+2}:SD_{16}$, or
 $3^{1+2}:Z_{8} \leq G_{0} \leq PSU_{3}(3)$.
\item $\mathcal{D}$ is a $2$-$(2^{6},2^{3},2^{3})$-design and $G_{0}$ is one of the groups $3^{1+2}:Q_{8}$, $%
3^{1+2}:C_{8}$, $3^{1+2}:SD_{16}$, $\left( 3^{1+2}:Q_{8}\right) :3:2$.
\end{enumerate}

Moreover, (6)--(13) do occur and they are isomorphic to the $2$-designs constructed in Examples %
\ref{Ex4}--\ref{Ex6}.
\end{theorem}

\bigskip

The proof of Theorem \ref{main} is outlined as follows. The semilinear $1$-dimensional case is analyzed in Lemma \ref{jedan}, whereas in almost all the remaining cases it shown that the order of $G_{0}$ is divisible by the primitive part of $p^{2m}-1$ and hence $G_{0}$ is determined in \cite[Theorem 3.1]{BP}. Each group is then analyzed singularly, and the structure of the stabilizer of a flag is mostly determined by combining the Aschbacher's theorem \cite {KL} together with the constraints on the replication number of $\mathcal{D}$ and on the blocks of this one provided in Section \ref{S2}. 

\bigskip

An immediate consequence of Theorem \ref{main} is the following corollary which determines the flag-transitive $2$-$(2^{2m},2^{m},2)$ designs of affine type and hence it provides a small contribution to the result contained in \cite{DLPX}. 

\begin{corollary}
Let $\mathcal{D}$ be a $2$-$(2^{2m},2^{m},2)$ design
admitting a flag-transitive automorphism group $G$ of affine type. If $G\nleq A\Gamma L_{1}(2^{m})$, then $\mathcal{D}$ is a $2$-$(16,4,2)$ design isomorphic to that constructed in Example \ref{Ex4} for $q=4$.
\end{corollary}

\bigskip

The remainder of this section is dedicated to construction of the following examples.

\bigskip

\begin{example}
\label{Ex4}Let $q=p^{h}$, with $h$ even and $h>0$, let $B=\left\langle
e_{1},e_{2}\right\rangle _{GF(q^{1/2})}$, where $\left\{ e_{1},e_{2}\right\} 
$ is a basis of $V_{2}(q)$, and let $G=TG_{0}$, where $SL_{2}(q)%
\trianglelefteq G_{0}\leq (Z_{q^{1/2}-1}\circ SL_{2}(q)).Z_{h}$. \ Then $%
\mathcal{D}=(V_{2}(q),B^{G})$ is a $2$-$(q^{2},q,q^{1/2})$ design admitting $%
G$ as a point-$2$-transitive, flag-transitive automorphism group.
\end{example}

\begin{proof}
Let $q$, $B$ and $G$ as in the statement. Then $G_{B}=T_{B}:G_{0,B}$, where $%
SL_{2}(q^{1/2})\trianglelefteq G_{0,B} = G_{0} \cap \Gamma L_{2}(q)_{B}$ and $%
\Gamma L_{2}(q)_{B}\cong (Z_{q^{1/2}-1}\circ SL_{2}(q^{1/2})).Z_{h}$ by \cite[Table 8.1]{BHRD}. Hence $%
\mathcal{D}=(V,B^{G})$ is a $2$-$(q^{2},q,\lambda )$ design, since $G$ acts $%
2$-transitively on $V$. Also, $G$ acts flag-transitively on $\mathcal{D}$,
since $T_{B}$ acts regularly on $B$ and $G$ is block-transitive on $\mathcal{%
D}$. Moreover $r=\left[ G_{0}:G_{0,B}\right]=[SL_{2}(q):SL_{2}(q^{1/2})]=q^{1/2}(q+1)$, since $SL_{2}(q)$ contains $(2,q-1)$
conjugacy classes of subgroups isomorphic to $SL_{2}(q^{1/2})$, and when $q$ is odd these are fused in $\Gamma L_{2}(q)$ but not in $(Z_{q^{1/2}-1}\circ SL_{2}(q)).Z_{h}$ by \cite[Table 8.1]{BHRD}. Thus $\lambda =q^{1/2}$.
\end{proof}

\bigskip

\begin{example}
\label{Ex5}Let $V$ be a $3$-dimensional $GF(q)$-space, with $q=p^{h}$, $p$
odd and $h$ even, let $\mathcal{H}$ be a
Hermitian unital of order $q$ of $PG_{2}(V)$ and let $\pi $ be a Baer
subplane of $PG_{2}(q)$ such that $\pi \cap \mathcal{H}$ is a non-degenerate
conic of $\pi $. If $B$ is any $3$-dimensional $GF(q^{1/2})$-subspace of $V$
inducing $\pi $, and $G=TG_{0}$ where $%
SU_{3}(q^{1/2})\trianglelefteq G_{0}\leq (Z_{q^{1/2}-1}\times
SU_{3}(q^{1/2})).Z_{h}$ induces a group on $PG_{2}(V)$ that preserves $%
\mathcal{H}$, then $\mathcal{D}=(V,B^{G})$ is a $2$-$(q^{3},q^{3/2},q)$
design admitting $G$ as a flag-transitive automorphism group.
\end{example}

\begin{proof}
Let $V$, $G$, $B$ be as in the statement, and let $x_{1}^{G_{0}}$ and $%
x_{2}^{G_{0}}$ be the two orbits consisting of the non-zero, isotropic and
non-isotropic vectors of $V$ respectively. Then $\left\vert x_{1}^{G_{0}}\right\vert = (q^{3/2}+1)(q-1)$ and $\left\vert x_{2}^{G_{0}}\right\vert = q(q^{3/2}+1)(q^{1/2}-1)$. Moreover, $(x_{i}^{G_{0}},B^{G_{0}})$ is a tactical configuration for each $i=1,2$ by \cite[1.2.6]{Demb}.

Set $X_{0}$ the (normal) copy of $SU_{3}(q^{1/2})$ inside $G_{0}$. It results from \cite[Table 8.5]{BHRD} that $X_{0,B}%
\trianglelefteq G_{0,B}=G_{0} \cap \Gamma U_{3}(q^{1/2})_{B}$, where $X_{0,B} \cong SO_{3}(q^{1/2})$ and $\Gamma U_{3}(q^{1/2})_{B} \cong  \Gamma O_{3}(q^{1/2})$. Moreover, $X_{0}$ contains $(3,q^{1/2}+1)$ conjugacy classes of subgroups isomorphic to $X_{0,B}$ and these are fused in $\Gamma U_{3}(q^{1/2})$ but not in $(Z_{q^{1/2}-1}\times
SU_{3}(q^{1/2})).Z_{h}$ and hence neither in $G_{0}$. Thus $\left\vert B^{G_{0}}\right\vert =\left[ G_{0}:G_{0,B}\right] =\left[X_{0}:X_{0,B}\right] =q(q^{3/2}+1)$. 

Since $G_{0,B}$ acts $2$-transitively on the non-degenerate conic $\pi \cap \mathcal{H}$ of $\pi$ and since the stabilizer in $G_{0,B}$ of a point $\left\langle c \right\rangle_{GF(q)}$ of  $\pi \cap \mathcal{H}$ splits the set of the non-zero vectors of $\left\langle c \right\rangle_{GF(q)}$ into $q^{1/2}+1$ orbits each of length $q^{1/2}-1$, it follows that $\left\vert B\cap x_{1}^{G_{0}}\right\vert
=q-1$. Consequently, we have that $\left\vert B\cap x_{2}^{G_{0}}\right\vert
=\allowbreak q\left( q^{1/2}-1\right)$ as $\left\vert B\backslash \left\{ (0,0,0)\right\} \right\vert =q^{3/2}-1$. Therefore the parameters of $(x_{i}^{G_{0}},B^{G_{0}})$ are 
\[
(v_{i},k_{i},b_{i},r_{i})=\left\{ 
\begin{array}{ll}
((q^{3/2}+1)(q-1),q-1,q(q^{3/2}+1),q)& \text{ for }i=1 \\ 
(q(q^{3/2}+1)(q^{1/2}-1),q\left( q^{1/2}-1\right),q(q^{3/2}+1),q)& \text{ for }i=2 \text{.}\\
\end{array}%
\right. 
\]%
Thus $%
r_{1}=r_{2}=q$ and hence the number of elements in $B^{G_{0}}$ incident with 
$0$ and with any non-zero vector of $V$ is $q$. Therefore the
flag-transitive incidence structure $\mathcal{D}=(V,B^{G})$ is a $2$-$%
(q^{3},q^{3/2},q)$ design.
\end{proof}

\begin{example}
\label{Ex0}Let $q=p^{h}$, $h\geq 1$, and let $B$ a non-degenerate $2$%
-dimensional subspace of $V_{4}(q)$ with respect to a symplectic form. Then $%
\mathcal{D}=(V_{4}(q),B^{G})$, where $G=TG_{0}$ and $Sp_{4}(q)%
\trianglelefteq G_{0}\leq \Gamma Sp_{4}(q)$, is a point-$2$%
-transitive, flag-transitive $2$-$(q^{4},q^{2},q^{2})$ design.
\end{example}

\begin{proof}
Let $B$ a non-degenerate $2$-dimensional subspace of $V_{4}(q)$ with respect
to a symplectic form, then $V_{4}(q)=B\oplus B^{\perp }$. Hence $\mathcal{D}%
=(V_{4}(q),B^{G})$ is a $2$-$(q^{4},q^{2},\lambda )$ design, since $G$ acts
point-$2$-transitively on $V_{4}(q)$. Also, $G$ acts flag-transitively on $%
\mathcal{D}$ and $Sp_{2}(q) \times Sp_{2}(q)\trianglelefteq G_{0,B}$ by \cite[Table 8.12]{BHRD}. Moreover $$r=\left[ G_{0}:G_{0,B}\right] =\left[
Sp_{4}(q):Sp_{2}(q)^{2}\right] =q^{2}(q^{2}+1),$$
since the stabilizers of non-degenerate $2$-dimensional subspaces of $V$ lies in a unique conjugacy $Sp_{4}(q)$-class, hence $\lambda =q^{2}$.
\end{proof}

\begin{example}
\label{Ex2} Let $(x_{0},y_{0},z_{0},t_{0})$ be a non-zero vector of $V=V_{4}(q)$, $q=2^{h}$, $%
h=2e+1$ and $e \geq 1$, let $\sigma \in \mathrm{Aut}(GF(q))$ defined by $\sigma: w \rightarrow w^{2^{e}}$, and let
\begin{equation}\label{BgigB}
B=\left\lbrace (m_{1}^{\sigma+2}x_{0},m_{1}^{\sigma}y_{0},m_{2}^{-\sigma}z_{0},m_{2}^{-\sigma -2}t_{0})\right\rbrace
\end{equation}
be a $GF(2)$-subspace of $V$ (e.g. see \cite[Lemma IV.2.1]{Lu}). Then $\mathcal{D}=(V,B^{G})$, where $%
G=TG_{0}$ and $Sz(q)\trianglelefteq G_{0}$, is a flag-transitive $2$-$(q^{4},q^{2},\lambda)$ design provided that one of the following conditions is fulfilled:
\begin{description}
\item [Family 1] $x_{0}=z_{0}=1$ and $y_{0}=t_{0}=0$. In this case
\begin{itemize}
\item $B$ is a $2$-dimensional $GF(q)$-subspace of $V$ inducing a line of $PG_{3}(q)$ tangent to the Tits Ovoid and not lying the L\"{u}neburg spread.
\item $\lambda=q$.
\item $G_{0}\leq \left( Z\times Sz(q) \right).Z_{h}$, where $Z$ denotes the center of $GL_{4}(q)$.
\item $F_{q(q-1)} \trianglelefteq G_{0,B}=G_{0}\cap (\left( Z\times Sz(q) \right).Z_{h})_{B}$, where $(\left( Z\times Sz(q) \right).Z_{h})_{B} \cong (Z \times F_{q(q-1)}).Z_{h}$.
\end{itemize} 
\item [Family 2] $x_{0},y_{0},z_{0},t_{0} \neq 0$, $t_{0}y_{0}^{\sigma +1} = z_{0}^{\sigma +1}x_{0}$ with  $(z_{0},t_{0})=(y_{0},x_{0})$. Here
\begin{itemize}
\item $\lambda =q^{2}/2$. 
\item $G_{0}\leq \mathrm{Aut}(Sz(q)) \cong Sz(q).Z_{h}$.
\item $D_{2(q-1)} \trianglelefteq G_{0,B}=G_{0}\cap \mathrm{Aut}(Sz(q))_{B}$, where  $\mathrm{Aut}(Sz(q))_{B}\cong D_{2(q-1)}.Z_{h}$.
\end{itemize}
\item[Family 3] $x_{0},y_{0},z_{0},t_{0} \neq 0$, $t_{0}y_{0}^{\sigma +1} = z_{0}^{\sigma +1}x_{0}$ with  $(z_{0},t_{0})\neq (y_{0},x_{0})$. In this case
\begin{itemize}
\item $\lambda =q^{2}$.
\item $G_{0}\leq \mathrm{Aut}(Sz(q))$.
\item $Z_{q-1} \trianglelefteq G_{0,B}= G_{0} \cap \mathrm{Aut}(Sz(q))_{B}$, where $\mathrm{Aut}(Sz(q))_{B}\cong  Z_{q-1}.Z_{h}$.
\end{itemize}
\item [Family 4] $x_{0},y_{0},z_{0},t_{0} \neq 0$ and $t_{0}y_{0}^{\sigma +1} \neq z_{0}^{\sigma +1}x_{0}$ and the map $$\zeta
:X\rightarrow \frac{(y_{0}/t_{0})^{\sigma }X^{\sigma }+(z_{0}/t_{0})^{\sigma
+2}}{(x_{0}/t_{0})X^{\sigma }+(z_{0}/t_{0})(y_{0}/t_{0})}$$ of $P\Gamma
L_{2}(q)$ fixes exactly one point in $PG_{1}(q) \setminus \{0, \infty\}$. Here
\begin{itemize}
\item $\lambda =q^{2}$.
\item $G_{0}\leq \mathrm{Aut}(Sz(q))$.
\item $Z_{q-1} \trianglelefteq G_{0,B}= G_{0} \cap \mathrm{Aut}(Sz(q))_{B}$, where $\mathrm{Aut}(Sz(q))_{B}\cong  Z_{q-1}.Z_{h}$.
\end{itemize}  
\end{description}
\end{example}

\begin{proof}
It is well known that $G_{0}$ partitions the set of non-zero vectors of $%
V_{4}(q)$ in two orbits, say $x_{1}^{G_{0}}$ and $x_{2}^{G_{0}}$, of length $%
(q^{2}+1)(q-1)$ and $(q^{2}+1)q(q-1)$ respectively, and these induce the Tits Ovoid and its complementary set in $PG_{3}(q)$ correspondingly. Let $B$ be as in (\ref{BgigB}).  Then $(x_{i}^{G_{0}},B^{G_{0}})$, $%
i=1,2 $, are tactical configurations by \cite[1.2.6]{Demb}. Set $H=T:H_{0}$, where $H_{0}$ is the subgroup of $G_{0}$ isomorphic to $Sz(q)$. One can easily see that $\left\vert B^{H_{0}}\right\vert= \lambda(q^{2}+1)$, where $\lambda$ is as in Families (1)--(4) by using (\ref{Sugr}). Moreover $B^{H_{0}}=B^{G_{0}}$, since $H_{0}$ has a unique conjugacy class of subgroups isomorphic to $F_{q(q-1)},Z_{q-1}$ or $D_{2(q-1)}$ respectively, and it also is a conjugacy $G_{0}$-class (see \cite[Table 8.16]{BHRD}). Therefore, the parameters $(v_{i},k_{i},b_{i},r_{i})$ of $(x_{i}^{G_{0}},B^{G_{0}})$, $i=1,2$, are $((q^{2}+1)(q-1),k_{1},\lambda(q^{2}+1),r_{1})$ and $((q^{2}+1)q(q-1),k_{2},\lambda(q^{2}+1),r_{2})$ respectively. Hence $r_{1}(q-1)=\lambda k_{1}$ and $r_{2}q(q-1)=\lambda k_{2}$. Then the flag-transitive incidence structure $\mathcal{D}=(V,B^{G})$ is a $2$-$(q^{4},q^{2},\lambda)$ design, where $\lambda$ is as in Families (1)--(4) respectively if, and only if, $r_{1}=r_{2}=\lambda$. This is equivalent to show that $k_{1}=q-1$ in (1)--(4), since $k_{1}+k_{2}=q^{2}-1$.     

It is immediate to see that $k_{1}=q-1$ in Family (1) by using (\ref{Ovoid}), hence $\mathcal{D}$ is a $2$-$(q^{4},q^{2},q)$ design in this case. Also $B$ induces a line of $PG_{3}(q)$ tangent to $\mathcal{O}$ and not lying the L\"{u}neburg spread by \cite[Theorem IV.23.2]{Lu}.
The fact that $k_{1}=q-1$ actually corresponds to the conditions on $x_{0},y_{0},z_{0},t_{0}$ for the  Families (2)-(4) is shown in Proposition \ref{IntOrb} and Corollary \ref{DiheD}. Thus in these cases $\mathcal{D}$ is a to $2$-$(q^{4},q^{2},\lambda)$ designs with $\lambda=q^{2}/2,q^{2},q^{2}$ respectively. Although it is immediate to see that there are examples belonging to any of the Families (2)-(3) for each $q$, it is difficult to prove the existence of examples belonging to Family (4) for any $q$. However, some computations made with the aid of \textsf{GAP} \cite{GAP} show that examples occur for $q=2^{3}$ or $2^{5}$.    
\end{proof}

\begin{example}
\label{Ex1}Let $q=2^{h}$, $h\geq 1$, and let $B$ be any totally isotropic $3$%
-dimensional subspace of $V_{6}(q)$ with respect to the $G_{2}(q)$-invariant
symplectic form. Then $\mathcal{D}=(V_{6}(q),B^{G})$, with $%
G=TG_{0}$ and $G_{2}(q)\trianglelefteq G_{0}\leq (Z_{q-1}\times
G_{2}(q)):Z_{h}$, where $Z_{h}=\mathrm{Out}(G_{2}(q))$, is a point $2$%
-transitive, flag-transitive $2$-$(q^{6},q^{3},q^{3})$ design.
\end{example}

\begin{proof}
Since $G_{2}(q)<Sp_{6}(q)$, $q=2^{h}$, $h\geq 1$, it follows that $G_{2}(q)$
preserves a symplectic form $F$ on $V$, where $V=V_{6}(q)$. The group $%
G_{2}(q)$ has one orbit of length $q^{3}(q^{3}+1)$ of totally isotropic $3$%
-dimensional subspaces of $V$ by \cite[Lemma 5.3]{Coop}. Moreover, if $B$ is
a representative of such a orbit, $G_{2}(q)_{B}\cong SL_{3}(q)$ and $\mathrm{%
Aut}(G_{2}(q))_{B}\cong SL_{3}(q):Z_{h}$.

Let $G=TG_{0}$, where $G_{2}(q)\leq G_{0}\leq (Z_{q-1}\times G_{2}(q):Z_{h}$%
, where $Z_{q-1}$ is the center of $GL_{6}(q)$. Then $\left\vert B^{G_{0}}\right\vert =q^{6}(q^{3}+1)$ by \cite[Table 8.30]{BHRD},
hence $\mathcal{D}=(V,B^{G})$ is a $2$-$(q^{6},q^{3},q^{3})$ design as $G$
acts $2$-transitively on $V$. Also $\mathcal{D}$ is clearly flag-transitive.
\end{proof}

\begin{example}
\label{Ex3}Let $B=\left\langle (1,0,0,0),(0,0,0,1)\right\rangle _{GF(3)}$
and let $G=TG_{0}$ where $SL_{2}(5)\trianglelefteq G_{0}\leq \left(
\left\langle -1\right\rangle .S_{5}^{-}\right) :Z_{2}$. Then $\mathcal{D}%
=\left( V_{4}(3),B^{G_{0}}\right) $ is a $2$-$(3^{4},3^{2},3)$-design
admitting $G=TG_{0}$ as flag-transitive automorphism group. Moreover, if $%
G_{0}$ is not isomorphic $SL_{2}(5)$ then $G_{0}$ acts point-$2$%
-transitively on $\mathcal{D}$.
\end{example}

\begin{proof}
Let $H_{0}=\left\langle \alpha ,\beta \right\rangle $, where 
\[
\alpha =\left( 
\begin{array}{rrrr}
0 & 0 & 0 & -1 \\ 
-1 & 0 & 1 & 1 \\ 
-1 & -1 & 0 & -1 \\ 
1 & 0 & 0 & 0%
\end{array}%
\right)\text{ and  } \; \beta =\left( 
\begin{array}{cccc}
1 & 1 & 0 & 0 \\ 
0 & 1 & 0 & 0 \\ 
0 & 0 & 1 & 1 \\ 
0 & 0 & 0 & 1%
\end{array}%
\right) \text{.} 
\]%
It is straightforward to check that $\alpha ^{4}=\left[ \alpha ^{2},\beta %
\right] =\beta ^{3}=\left( \alpha \beta \right) ^{5}=1$. Thus $H_{0}\cong
SL_{2}(5)$. Moreover $H_{0}$ preserves the symplectic form $$%
F(X,Y)=-X_{1}Y_{2}+X_{2}Y_{1}-X_{3}Y_{4}+X_{4}Y_{3},$$ where $X=(X_{1},X_{2},X_{3},X_{4})$ and $Y=(Y_{1},Y_{2},Y_{3},Y_{4})$, hence $%
H_{0}<Sp_{4}(3)$.

The $2$-dimensional subspace $B=\left\langle
(1,0,0,0),(0,0,0,1)\right\rangle $ of $V_{4}(3)$ is\ totally isotropic (with
respect to $F$), and $\left\langle \alpha \right\rangle \leq H_{0,B}$.
Assume that the order of $H_{0,B}$ is divisible of $8$. Then $H_{0,B}$
contains the Sylow $2$-subgroup $S$ of $H_{0}$ containing $\left\langle
\alpha \right\rangle $ and hence $S=\left\langle
\alpha ,\gamma \right\rangle \cong Q_{8}$, where $\gamma =\alpha ^{3}\beta
\alpha ^{3}\beta ^{2}\alpha ^{3}\beta \alpha ^{3}\beta ^{2}\alpha \beta $.
However,%
\[
\gamma =\left( 
\begin{array}{rrrr}
-1 & 0 & -1 & -1 \\ 
0 & 1 & 0 & 1 \\ 
-1 & -1 & 1 & -1 \\ 
0 & 1 & 0 & -1%
\end{array}%
\right) 
\]%
does not preserve $B$.

Assume that the order of $H_{0,B}$ is divisible of $3$. Since each Sylow $2$%
-subgroup of $SL_{2}(5)$ is isomorphic to $Q_{8}$ and its unique involution
is $-1$, and since each $5$-subgroup of $H_{0}$ acts irreducibly on $%
V_{4}(3) $, it follows that the forty $1$-dimensional subspaces of $V_{4}(3)$
are is partitioned into two $H_{0}$-orbits each of length $20$. Hence $\mathrm{%
Fix(}K\mathrm{)}$ is a $2$-dimensional subspace of $V_{4}(3)$, where $K \cong Z_{3}$, and $\mathcal{%
S}=\mathrm{Fix(}K\mathrm{)}^{H_{0}}$ is a transitive $2$-spread of $V_{4}(3)$%
. Also $\mathcal{S}$ is Desarguesian by \cite[Proposition 5.3 and Corollary 5.5]{Fou2}.

If $B\in \mathcal{S}$ then $\beta $ fixes $B$, since $\beta $ fixes $%
(0,0,0,1)$ in $B$, and we reach a contradiction. Then $H_{0,B}$ preserves the $1$%
-dimensional subspace $\mathrm{Fix}(K)\cap B$. However, this is
impossible since the square of each element of order $4$ contained in $%
H_{0,B}$ is $-1$. Thus $H_{0,B}=\left\langle \alpha \right\rangle $ and
hence $\left\vert B^{H_{0}}\right\vert =30$.

Let 
\[
\delta =\left( 
\begin{array}{rrrr}
0 & 0 & 1 & 1 \\ 
0 & 0 & 0 & 1 \\ 
-1 & 1 & 0 & 0 \\ 
0 & -1 & 0 & 0%
\end{array}%
\right) \text{ and  } \; \psi =\left( 
\begin{array}{rrrr}
1 & 1 & 0 & 0 \\ 
0 & -1 & 0 & 0 \\ 
0 & 0 & 1 & 0 \\ 
0 & 0 & 0 & -1%
\end{array}%
\right) \text{.} 
\]

It is not difficult to see that $\delta \in N_{Sp_{4}(3)}(H_{0})$. On the
other hand, $H_{0}\left\langle \delta \right\rangle \cong \left\langle
-1\right\rangle .S_{5}^{-}$ by \cite{At}, since the Sylow $2$-subgroups of $%
H_{0}\left\langle \delta \right\rangle $ are isomorphic to $Q_{16}$. Then $%
H_{0}\left\langle \delta \right\rangle $ acts transitively on the set of
non-zero vectors of $V_{4}(3)$, since its unique involution is $-1$ and
since each $5$-subgroup of $H_{0}\left\langle \delta \right\rangle $ acts
irreducibly on $V_{4}(3)$. Then $T:H_{0}\left\langle \delta \right\rangle $
acts $2$-transitively on $V_{4}(3)$. Moreover, $B^{H_{0}}=B^{H_{0}\left\langle \delta \right\rangle}$ since $%
H_{0}\vartriangleleft H_{0}\left\langle \delta \right\rangle <Sp_{4}(3)$ and
since the number of totally isotropic $2$-dimensional subspaces is $40$.
Thus $\mathcal{D}=\left( V_{4}(3),B^{G_{0}}\right) $ is a $2$-$%
(3^{4},3^{2},3)$-design admitting both $T:H_{0}$ and $T:H_{0}\left\langle
\delta \right\rangle $ as a flag-transitive automorphism groups.

Since $%
H_{0}\trianglelefteq H_{0}\left\langle \delta \right\rangle :\left\langle
\psi \right\rangle \leq GSp_{4}(3)$ and $o(\psi )=2$, since $%
GSp_{4}(3) $ preserves the $Sp_{4}(3)$-orbit consisting of the $40$ totally
isotropic $2$-dimensional subspaces, it follows that $B^{H_{0}}=B^{H_{0}\left\langle \delta
\right\rangle }=B^{H_{0}\left\langle \delta ,\psi \right\rangle }$. Hence 
$\mathcal{D}$ admits also $T:(H_{0}\left\langle \delta
,\psi \right\rangle )$ as flag-transitive automorphism groups. Now, if we set $G=T:G_{0}$, where $H_{0} \trianglelefteq G_{0} \leq H_{0}\left\langle \delta,\psi \right\rangle$, the assertion follows.
\end{proof}

\bigskip

It is worth noting that in Examples \ref{Ex4}--\ref{Ex3} the full
flag-transitive automorphism group $\Gamma $ is determined. Indeed, it
follows from \cite{MF} and \cite{Mo1} that $\Gamma $ is of affine type since
the number of points of $\mathrm{\mathcal{D}}$ is a power of a prime. Now,
we may apply Theorem \ref{main}, with $\Gamma $ in the role of $G$, thus
obtaining that $\Gamma $ is a group as in Examples \ref{Ex4}--\ref{Ex3}.

\bigskip 
\begin{example}
\label{Ex6} The following further examples are found with the aid of \textsf{GAP}\cite{GAP}. In
these cases $\mathcal{D}=(V_{6}(2),B^{G})$ is a $2$-$(2^{6},2^{3},2^{i} )$
design, with $1 \leq i \leq 3$, admitting a flag-transitive automorphism group $G=TG_{0}$, where $%
(B,G_{0})$ are as in Table \ref{t0}.

\begin{table}[h]
\caption{flag-transitive $2-(2^{6},2^{3},2^{i} )$ designs with $1 \leq i \leq 3$}
\label{t0} 
\centering
\begin{tabular}{|l|l|l|}
\hline
$i $ & Base block $B$ & $G_{0}$ \\ \hline
$1$ & $\left\{ 1,6,10,16,20,22,28,30\right\} $ & $D_{18}$, $D_{18}.Z_{3}$ \\ 
\hline
$2$ & $\left\{ 1,2,7,8,18,24,30,32\right\} $ & $3^{1+2}:Q_{8}$, $%
3^{1+2}:Z_{8}$, $3^{1+2}:SD_{16}$ \\ \hline
& $\left\{ 1,2,7,8,26,28,31,32\right\} $ & $3^{1+2}:Z_{8}$, $PSU_{3}(3)$ \\ 
\hline
$3$ & $\left\{ 1,2,3,4,18,19,21,24\right\} $ & $3^{1+2}:Q_{8}$, $%
3^{1+2}:Z_{8}$, $3^{1+2}:SD_{16}$, $\left( 3^{1+2}:Q_{8}\right) :3:2$ \\ 
\hline
& $\left\{ 1,2,7,8,26,28,31,32\right\} $ & $3^{1+2}:SD_{16}$, $G_{2}(2)$ \\ 
\hline
\end{tabular}
\end{table}

Note that the $2$-design in Line 1 is an example of (1) in Theorem \ref{main}. Indeed, $%
G_{0}<\Gamma L_{1}(2^{6})$ and it can be shown that the set of the blocks incident with 
$0$ is the union of two Desarguesian spreads of $V_{6}(2)$ each of these
permuted transitively by a cyclic group of order $9$ of $G_{0}$. The two
spreads are switched by an involution of $G_{0}$.

The example in Line 6 actually is a $2$-design constructed in Example \ref{Ex1} for $q=2$. Finally, the
example in Line 3 is a subdesign of that in Line 6 as the base block is the
same in both cases and since $G_{2}(2)\cong PSU_{3}(3):Z_{2}$.
\end{example}

\section{Preliminary Reductions}\label{S2}

Let $\mathcal{D}=(\mathcal{P},\mathcal{B})$ be a $2$-$(k^{2},k,\lambda )$
design, where $\lambda \mid k$, admitting a flag-transitive automorphism
group $G$ such that $T$, the socle of $G$, is an an elementary abelian $p$%
-group for some prime $p$. Then $G$ acts point-primitively on $\mathcal{D}$
by \cite[2.3.7.c]{Demb}, since $r=(k+1)\lambda >(k-3)\lambda $. Thus $%
\mathcal{P}$ can be identified with a $2m$-dimensional $GF(p)$-vector space $%
V$ in a way that $T$ is the translation group of $V$ and hence $G=TG_{0}\leq
AGL(V)$. Therefore, $\mathcal{D}$ is a $2$-$(p^{2m},p^{m},p^{f})$ design,
with $0\leq f\leq m$, and so $r=p^{f}(p^{m}+1)$ and $b=p^{m+f}(p^{m}+1)$.

\bigskip

The following lemmas and corollary will frequently be used throughout the
paper.

\begin{lemma}
\label{PP}If $\mathcal{D}$ is a $2$-$(k^{2},k,\lambda )$ design, with $%
\lambda \mid k$, admitting a flag-transitive automorphism group $G$, then $%
\left\vert y^{G_{x}}\right\vert =(p^{m}+1)\left\vert B\cap
y^{G_{x}}\right\vert $ for any point $y$ of $\mathcal{D}$, with $y\neq x$,
and for any block $B$ of $\mathcal{D}$ incident with $x$. In particular, $%
p^{m}+1$ divides the length of each point-$G_{x}$-orbit on $\mathcal{D}$
distinct from $\left\{ x\right\} $.
\end{lemma}

\begin{proof}
Let $y$ be any point of $\mathcal{D}$, $y\neq x$, and $B$ be any block of $%
\mathcal{D}$ incident with $x$. Since $(y^{G_{x}},B^{G_{x}})$ is a tactical
configuration by \cite[1.2.6]{Demb}, and $r=p^{f}(p^{m}+1)$, it follows that $%
\left\vert y^{G_{x}}\right\vert p^{f}=(p^{m}+1)p^{f}\left\vert B\cap
y^{G_{x}}\right\vert $. Hence $\left\vert y^{G_{x}}\right\vert
=(p^{m}+1)\left\vert B\cap y^{G_{x}}\right\vert $, which is the assertion.
\end{proof}

\begin{lemma}
\label{inv}If $p$ is odd, $T$ does not act block-semiregularly on $\mathcal{D%
}$ and $-1\in G_{0}$, then the blocks of $\mathcal{D}$ are subspaces of $%
AG_{2m}(p)$.
\end{lemma}

\begin{proof}
Let $x$ be any non-zero vector of $V$. Then $-1$ preserves the $p^{f}$
blocks incident with $\pm x$, and hence $-1$ preserves one of them, say $B$,
as $p$ is odd. Then $-1$ fixes a point on $B$ as $p^{m}$ is odd and hence $%
0\in B$. Also $T_{B}\neq 1$, since $G$ acts block-transitively, $T$ does
not act block-semiregularly on $\mathcal{D}$ and $T \triangleleft G$. Hence $0^{T_{B}}=0^{T_{B}\left%
\langle -1\right\rangle }$. Then each $T_{B}$-orbit in $B$ is also a $%
T_{B}\left\langle -1\right\rangle $-orbit as both $T_{B}$ and $%
T_{B}\left\langle -1\right\rangle $ are normal subgroups of $G_{B}$.
Therefore $-1$ fixes a point on each $T_{B}$-orbit contained in $B$, as $%
G_{B}$ acts transitively on $B$. Then $B=0^{T_{B}}$ since $-1$ does not fix
non-zero vectors of $V$. Thus $B$ is a subspace of $V$ and hence the
assertion follows from the block-transitivity of $G$ on $\mathcal{D}$, since 
$V=V_{2m}(p)$.
\end{proof}

\begin{lemma}
\label{cici}The following hold:

\begin{enumerate}
\item $m-t\leq f\leq m$, where $p^{t}=\left\vert T_{B}\right\vert $ and $%
B $ is any block of $\mathcal{D}$.

\item $G_{B}/T_{B}$ is isomorphic to a subgroup of $G_{0}$ of index $%
p^{f-(m-t)}(p^{m}+1)$ and containing an isomorphic copy of $G_{0,B}$.
\end{enumerate}
\end{lemma}

\begin{proof}
Let $B$ be any block of $\mathcal{D}$ incident with $0$. Then $B$ is split into $T_{B}$%
-orbits of equal length $p^{t}$, where $t\geq 0$, permuted transitively by $%
G_{B}$, since $T_{B}\trianglelefteq G_{B}$. Thus $\left[ G_{B}:G_{0,B}T_{B}%
\right] =p^{m-t}$ and hence $G_{B}/T_{B}$ is isomorphic to a subgroup $J$ of 
$G_{0}$ such that $\left\vert J\right\vert =p^{m-t}\left\vert
G_{0,B}\right\vert $ and it contains an isomorphic copy of $G_{0,B}$. Then $r=\left[ G_{0}:J\right] p^{m-t}$ implies $\left[
G_{0}:J\right] p^{m-t}=p^{f}(p^{m}+1)$ and hence $m-t\leq f\leq m$. This
proves (1) and (2).
\end{proof}

\begin{corollary}
\label{p2}If $p=2$ and $t\geq m-1$, then the blocks of $\mathcal{D}$ are
subspaces of $AG_{2m}(2)$.
\end{corollary}

\begin{proof}
Assume that $t=m-1$. Then $B=0^{T_{B}}\cup x^{T_{B}}=0^{T_{B}}\cup
(0^{T_{B}}+x)$ for some $x\notin 0^{T_{B}}$, with $\left\vert
0^{T_{B}}\right\vert =\left\vert T_{B}\right\vert =2^{m-1}$. So, hence $%
B=0^{T_{B}}\oplus \left\langle x\right\rangle $, which is a contradiction as 
$t=m-1$. Thus $t=m$ and hence the assertion follows.
\end{proof}

\bigskip

For each divisor $h$ of $2m$ the group $\Gamma L_{2m/h}(p^{h})$ has a
natural irreducible action on $V$. The point-primitivity of $G$ on $\mathcal{%
D}$ implies that $G_{0}$ acts irreducibly on $V_{2m}(p)$. Choose $h$ to be
minimal such that $G_{0}\leq \Gamma L_{2m/h}(p^{h})$ in this action, and
write $q=p^{h}$ and $n=2m/h$. Thus $G_{0}\leq \Gamma L_{n}(q)$ and $%
v=q^{n}=p^{2m}$.

\bigskip

\begin{theorem}
\label{lambada}If $\lambda =1$, then $\mathcal{D}$ is a translation plane
and either $G_{0}\leq \Gamma L_{1}(q)$ or one of the following holds:

\begin{enumerate}
\item $\mathcal{D}$ is the Desarguesian plane of order $q$ and $SL_{2}(q)\trianglelefteq G_{0}\leq
\Gamma L_{2}(q)$.

\item $\mathcal{D}$ is the L\"{u}neburg plane of order $2^{2h}$, 
$h$ odd, and $Sz(2^{h})\trianglelefteq G_{0}\leq \left( Z_{2^{h}-1}\times
Sz(2^{h})\right) .Z_{h}$.

\item $\mathcal{D}$ is the Hall plane of order $3^{2}$ and $SL_{2}(5)\leq
G_{0}\leq (D_{8}\circ Q_{8}).S_{5}$.

\item $\mathcal{D}$ is the Hering plane of order $3^{3}$ and $G_{0}\cong
SL_{2}(13)$.
\end{enumerate}
\end{theorem}

\begin{proof}
$\mathcal{D}$ is a translation plane of order $q$ by \cite[Theorems 2
and 4]{Wa}, since $\lambda =1$. Hence, the assertion follows from \cite[Main Theorem]{LiebF}.
\end{proof}

\bigskip

\emph{In view of Theorem \ref{lambada} in the sequel we will assume that $\lambda \geq 2$.}

\bigskip

\begin{lemma}
\label{jedan}If $G_{0}\leq \Gamma L_{1}(p^{2m})$, the blocks of $\mathcal{D}$
are $m$-dimensional subspaces of $AG_{2m}(p)$.
\end{lemma}

\begin{proof}
Since $G_{0}\leq \Gamma L_{1}(p^{2m})$, where $\Gamma
L_{1}(p^{2m})=\left\langle \bar{\omega},\bar{\alpha}\right\rangle $, with $%
\bar{\omega}:x\rightarrow \omega x$, $\omega $ a primitive element of $%
GF(p^{2m})^{\ast }$ and $\bar{\alpha}:x\rightarrow x^{p}$, it follows that $%
G_{0}=\left\langle \bar{\omega}^{c},\bar{\omega}^{e}\bar{\alpha}%
^{s}\right\rangle $ where $c,e,s$ are integers such that $c\mid p^{2m}-1$, $%
s\mid 2m$ and $e\left( \frac{p^{2m}-1}{p^{s}-1}\right) \equiv 0\pmod{c}$
(e.g. see \cite[Lemma 2.1]{FK}). Then $\left\vert G_{0}\right\vert =\frac{%
p^{2m}-1}{c}\cdot \frac{2m}{s}$, and since $r=p^{f}(p^{m}+1)$ divides $%
\left\vert G_{0}\right\vert $, it follows that 
\begin{equation}
p^{f}\mid \frac{2m}{s}\text{ and }c\mid \left( p^{m}-1\right) \left(
p^{m}+1,2m/sp^{f}\right) \text{.}  \label{restr}
\end{equation}%
Let $B$ be any block incident with $0$ and let $W=0^{T_{B}}$. Then $%
\left\vert W\right\vert =p^{t}$, where $t\geq 0$, and $0\leq m-t\leq f\leq m$
by Lemma \ref{cici}(1).

\begin{enumerate}
\item \textbf{$T$ does not act block-semiregularly on $\mathcal{D}$.}
\end{enumerate}

Assume that $t=0$. Then $f=m$ and hence $p^{m}\leq 2m/s\leq 2m$ by (\ref{restr}). Thus $%
p^{m}=4$, as $p^{m}>2$, and hence $G_{0}\leq \Gamma L_{1}(2^{4})$ and $%
\mathcal{D}$ is a $2$-$(2^{4},2^{2},2^{2})$ design. Moreover, $%
G_{0}=\left\langle \bar{\omega}^{c},\bar{\omega}^{e}\bar{\alpha}%
\right\rangle $ with $c\mid 3$. If $c=1$ then $G_{0}=\Gamma L_{1}(2^{4})$
and hence $\left\langle \bar{\omega}^{5}\right\rangle $ preserves each block
incident with $0$, since $G_{0}$ \ acts transitively on the $20$ blocks
incident with $0$ and since $\left\langle \bar{\omega}^{5}\right\rangle
\vartriangleleft G_{0}$. However this is impossible since $V^{\ast }$ is
split into $5$ orbits under $\left\langle \bar{\omega}^{5}\right\rangle $.
Thus $c=3$ and hence $G_{0}\cong F_{20}$. Moreover $G_{B} \cong Z_{4}$ by Lemma \ref{cici}(2).

Since $G_{0}$ contains a Sylow $2$-subgroup of $\Gamma L_{1}(2^{4})$ and since $%
\left\langle \bar{\omega}^{3}\right\rangle $ is a normal subgroup of $\Gamma
L_{1}(2^{4})$, we may assume that $e=0$. Hence, $G_{0}=\left\langle \bar{\omega}%
^{3},\bar{\alpha}\right\rangle $. Since $G_{B}\cong Z_{4}$ by Lemma \ref{cici}(2), we may also assume that $B$ is such that $G_{B} \leq T:\left\langle \bar{\alpha}\right\rangle$. Let $\gamma $ be a generator of $G_{B}$, one can easily see that either that $\gamma :x\rightarrow x^{2}+\omega ^{i}$
with $T_{GF(2^{4})/GF(2)}(\omega ^{i})=0$ and $\left\langle \gamma \right\rangle$ is a $G$-conjugate of $\left\langle \bar{\alpha}\right\rangle$, or $\gamma :x\rightarrow
x^{4}+\omega ^{i}$ with $\omega
^{i}\notin GF(4)$. Therefore, in the former case we may actually assume that $G_{B}=\left\langle \bar{\alpha}\right\rangle$. Then $B=\{z,z^{2},z^{4},z^{8} \}$ for some $z \in GF(16) \setminus GF(4)$. If $T_{GF(16)/GF(2)}(z)=0$, then $B+z=\{0, z^{2}+z,z^{4}+z,z^{8}+z\}$ is a $GF(2)$-space, and we reach a contradiction. Hence $T_{GF(16)/GF(2)}(z)=1$. However this case is ruled out since the norm of $GF(16)$ over $GF(4)$ on two elements of $B+z$ is equal. Indeed, the three $\left\langle \bar{\omega}^{3}\right\rangle $-orbits form a partition of $GF(16)^{\ast }$ and they consist of the points with norm of $GF(16)$ over $GF(4)$ equals to $1,\omega ^{5},\omega ^{10}$ respectively. Also each block of $\mathcal{D}$ incident with $0$ must
intersect each of them in precisely one point, since $\left\langle \bar{\omega}^{3}\right\rangle \trianglelefteq G_{0}$ and since $G_{0}$ permutes these blocks transitively.  

Assume that $G_{B}=\left\langle \gamma \right\rangle$ with $\gamma :x\rightarrow
x^{4}+\omega ^{i}$ and $\omega^{i}\notin GF(4)$. If $y \in B$ then $B=\{y,y^{4}+\omega^{i},y+\omega^{4i}+\omega^{i},y^{4}+\omega^{4i} \}$ and hence $B+y=\{0,y^{4}+y+\omega^{i},\omega^{4i}+\omega^{i},y^{4}+y+\omega^{4i} \}$ is a $GF(2)$-space, a contradiction. 
\begin{enumerate}
\item[2.] \textbf{The blocks are $m$-dimensional subspaces of $AG_{2m}(p)$.}
\end{enumerate}

Suppose the contrary. Then $0<t<m$ by (1). Assume that $G_{B}\cap T\left\langle \bar{\omega}^{c}\right\rangle \neq T_{B} $. Then $G_{B}\cap T\left\langle \bar{\omega}^{c}\right\rangle = T_{B}:X $ with $1<X \leq \left\langle \bar{\omega}^{c}\right\rangle^{\tau}$ for some $\tau \in T$ by \cite[Propositions 8.2 and 17.3]{Pass}, since $T\left\langle \bar{\omega}^{c}\right\rangle$ is a Frobenius group. Since $X$ has order coprime to $p$ and since $k=p^{m}$, it follows that $X$ fixes at least a point $x_{0}$ in $B$. Thus each $T_{B}$-orbit
in $B$ is also a $T_{B}X$-orbit, since both $T_{B}$ and $T_{B}X$ are normal subgroups of $G_{B}$, $x_{0}^{T_{B}X}=x_{0}^{T_{B}}$ and $G_{B}$ acts transitively on $B$. Then $X$ fixes a point in $B \setminus x_{0}^{T_{B}}$ since $0<t<m$. However this is impossible as $X$ acts semiregularly on $GF(q)\setminus \{x_{0}\}$, since $1<X \leq \left\langle \bar{\omega}^{c}\right\rangle^{\tau}$.  

Assume that $G_{B}\cap T\left\langle \bar{\omega}^{c}\right\rangle =T_{B}$.
Then $G_{B}/\left( G_{B}\cap T\left\langle \bar{\omega}^{c}\right\rangle
\right) $ is isomorphic to a subgroup of $\left\langle \bar{\omega}^{e}\bar{%
\alpha}^{s}\right\rangle $ of order $p^{m-t}\left\vert G_{0,B}\right\vert $.
Then $p^{m-t}\frac{\left( p^{m}-1\right) \left( 2m/sp^{f}\right) }{c}\mid
2m/sp^{f}$ as $\left\vert G_{0,B}\right\vert =\frac{\left( p^{m}-1\right)
\left( 2m/sp^{f}\right) }{c}$, hence $p^{m-t}\left( p^{m}-1\right) \mid c$.
Thus $t=m$ as $c\mid (p^{2m}-1)$, but this contradicts $t<m$. Therefore the
blocks of $\mathcal{D}$ are $m$-dimensional subspaces of $AG_{2m}(p)$.
\end{proof}

\bigskip

Let $a,e$ be integers. A divisor $w$ of $a^{e}-1$ that is coprime to each $%
a^{i}-1$ for $1\leq i<e$ is said to be a \emph{primitive divisor}, and we
call the largest primitive divisor $\Phi _{e}^{\ast }(a)$ of $a^{e}-1$ the 
\emph{primitive part} of $a^{e}-1$. One should note that $\Phi _{e}^{\ast
}(a)$ is strongly related to cyclotomy in that it is equal to the quotient
of the cyclotomic number $\Phi _{e}(a)$ and $(n,\Phi _{e}(a))$ when $e>2$.
Also $\Phi _{e}^{\ast }(a)>1$ for $a\geq 2$, $e>2$ and $(q,e)\neq (2,e)$ by
Zsigmondy's Theorem (for instance, see \cite[Theorem II.6.2]{Lu}).

\bigskip

\emph{From now on we assume that $m>1$ and $(p,2m)\neq (2,6)$.} The cases $m=1$ or $(p,2m)= (2,6)$ are tackled in the final section of the
present paper.

\bigskip

\begin{lemma}
\label{alternative}Set $H_{0}=G_{0}\cap GL_{n}(q)$, then $\Phi _{2m}^{\ast
}(p)$ divides the order of $H_{0}$ and $\Phi _{2m}^{\ast }(p)>1$.
\end{lemma}

\begin{proof}
Since $r=p^{f}(p^{m}+1)$ divides $\left\vert G_{0}\right\vert $, then $\Phi
_{2m}^{\ast }(p)$ divides $\left\vert G_{0}\right\vert $. Moreover, $\Phi
_{2m}^{\ast }(p)>1$ by \cite[Theorem II.6.2]{Lu}, since $m>1$ and $(p,2m) \neq (2,6)$
by our assumption. As $\Phi _{2m}^{\ast }(p)\equiv 1\pmod{2m}$ by \cite[Proposition 5.2.15.(ii)]{KL}, it follows that $\Phi _{2m}^{\ast }(p)$ is
coprime to $h$, since $h\mid 2m$. Thus $\Phi _{2m}^{\ast }(p)$ divides the
order $H_{0}$, since $q=p^{h}$ and $\left[ \Gamma L_{n}(q):GL_{n}(q)\right]
=h$.
\end{proof}

\bigskip

On the basis of Lemma \ref{alternative}, $H_{0}$ is one of the following
groups by \cite[Theorem 3.1]{BP} and by the minimality of $n$:

\begin{enumerate}
\item $SL_{n}(q)\trianglelefteq G_{0}$;

\item $Sp_{n}(q)\trianglelefteq G_{0}$;

\item $SU_{n}(q^{1/2})\trianglelefteq G_{0}$ and $n$ odd;

\item $\Omega _{n}^{-}(q)\trianglelefteq G_{0}$;

\item $G_{0}\leq (D_{8}\circ Q_{8}).S_{5}$ and $(q,n)=(3,4)$;

\item $G_{0}$ is a \emph{nearly simple group}, that is, $S\trianglelefteq
G_{0}/\left( G_{0}\cap Z\right) \leq Aut(S)$ for some non-abelian simple
group $S$, where $Z$ is the center of $GL(V)$. Moreover, if $N$ is the preimage of $S$ in $G_{0}$, then $N$ is
absolutely irreducible on $V$ and $N$ is not a classical subgroups defined
over a subfield of $GF(q)$ in its natural representation.
\end{enumerate}

\bigskip

We are going to analyze (6) first, then (5) and finally (1)--(4).

\section{The case where $G_{0}$ is a nearly simple group}

The aim of this section is to prove the following result.

\begin{theorem}
\label{ClassS}If $G_{0}$ is a nearly simple group, then $q=2^{h}$ and one of the following holds:
\begin{enumerate}

\item $\mathcal{D}$ is a $2$-$(2^{4},2^{2},2^{2})$ design isomorphic to that constructed in Example \ref{Ex0} for $q=2$, and $A_{6}\trianglelefteq G_{0}\leq S_{6}$.

\item $\mathcal{D}$ is a $2$-$(q^{4},q^{2},\lambda)$ design, with $\lambda = q, q^{2}/2$ or $q^{2}$, isomorphic to one the $2$-designs constructed in Example \ref{Ex1}, and $Sz(q)\trianglelefteq G_{0}\leq \left( Z_{q-1}\times
Sz(q)\right) .Z_{h}$ ($h$ is odd).

\item $\mathcal{D}$ is a $2$-$(q^{6},q^{3},q^{3})$ design isomorphic to that constructed in Example \ref{Ex2}, and $G_{2}(q)\trianglelefteq G_{0}\leq (Z_{q-1}\times
G_{2}(q)):Z_{h}$.

\end{enumerate}
\end{theorem}

\bigskip
We analyze the cases where $S$, the socle of $G_{0}/\left( G_{0}\cap Z\right)$, is alternating, sporadic and a Lie type simple group in characteristic $p^{\prime}$ and $p$ separately.
\bigskip

\begin{lemma}
\label{NotFDPM}If $S\cong A_{\ell }$, $\ell \geq 5$, then $\ell =6$ and $%
\mathcal{D}$ is isomorphic to the $2$-design of constructed in Example \ref%
{Ex0} for $q=2$.
\end{lemma}

\begin{proof}
Assume that $V$ is the fully deleted permutation module for $S$, where $%
S\cong A_{\ell }$, $\ell \geq 5$. Then $q=p$, $n=2m$, $A_{\ell
}\trianglelefteq G_{0}\leq S_{\ell }\times Z$, where $Z$ is the center of $%
GL_{n}(q)$, and either $\ell =2m+1$ or $2m+2$ according to whether $p$ does
not divide or does divide $n$, respectively, by \cite[Lemma 5.3.4]{KL}.
Moreover $q=2$ and $m=2,5,6,9\,$, or $q=3$ and $m=2,3$, or $q=5$ and $m=3$
by \cite[Theorem 3.1]{BP}. The group $G_{0}$ has one orbit on the set of $1$%
-dimensional $GF(q)$-subspaces of $V$ of\ length $2m+1$ or $(2m+1)\left(
m+1\right) $ according to $\ell =2m+1$ or $2m+2$ respectively. By Lemma \ref%
{PP} $p^{m}+1$ divides $(p-1)\left( 2m+1\right) $ or $(p-1)(2m+1)\left(
m+1\right) $, respectively, and this rules out the cases $p=2$ and $m=6,9$, or $p=5$ and $m=3$.

Assume that $p=m=2$ and $A_{5}\trianglelefteq G_{0}\leq S_{5}$. Then $%
V_{4}(2)^{\ast }=x^{G_{0}}\cup y^{G_{0}}$ with $\left\vert
x^{G_{0}}\right\vert =5$ and $\left\vert y^{G_{0}}\right\vert =10$. Moreover 
$G_{0}$ acts primitively on $y^{G_{0}}$ and hence $G_{0,y}$ is a maximal
subgroup of $G_{0}$ isomorphic either to $D_{6}$ or to $D_{12}$ according as $%
G_{0}$ is isomorphic either to $A_{5}$ or to $S_{5}$ respectively. Thus any
cyclic subgroup of $G_{0}$ of order $3$ fixes exactly one point on $y^{G_{0}}
$. Since $r=2^{f}\cdot 5$, if $B$ is a block incident with $0$ then $G_{0,B}$
contains a cyclic subgroup $C$ of order $3$. Since $\left\vert B\cap
y^{G_{0}}\right\vert =2$ by Lemma \ref{PP}, it follows that $C$ fixes $B\cap
y^{G_{0}}$ pointwise, but this is in contrast with the fact that any cyclic
subgroup of $G_{0}$ of order $3$ fixes exactly one point in $y^{G_{0}}$.

Assume that $p=m=2$ and $A_{6}\trianglelefteq G_{0}\leq S_{6}$. Then $G_{0}$
acts transitively on $V^{\ast }$ and $G_{0}$ contains two conjugacy
classes of subgroups of order $3$: one consisting of the subgroups fixing a
plane of $V_{4}(2)$, the other consisting of subgroups acting semiregularly
on $V^{\ast }$. Since $r=2^{f}\cdot 5$, if $B$ is any block incident
with $0$, $G_{0,B}$ contains a Sylow $3$-subgroup $G_{0}$ and hence there is
a cyclic subgroup $C$ of $G_{0}$ of order $3$ fixing $B$ pointwise. Actually $B=%
\mathrm{Fix}(C)$ and hence $f=2$ and $\mathcal{D}$ is the $2$%
-design constructed in Example \ref{Ex0}, since $G$ acts $2$-transitively on $V$ and $A_{6} \cong Sp_{4}(2)^{\prime}$.

Assume that $p=2$ and $m=5$. Hence $r=2^{f}\cdot 33$ with $f\leq 5$, and  $G_{0}\cong A_{12}:Z_{2^{\varepsilon}}$ with $\varepsilon=0,1$. If $B$ is any block incident with $0$, and $M$ is a maximal subgroup of $G_{0}$ containing $G_{0,B}$, then either $M \cong A_{11}:Z_{2^{\varepsilon}}$ or $M \cong S_{10} \times Z_{2^{\varepsilon}}$ by \cite{At}. If $M \cong A_{11}:Z_{2^{\varepsilon}}$, then $[G_{0}:M]=12$ and $[M:G_{0,B}]=2^{f-2}\cdot 11$. Therefore either $f=2$ and $G_{0,B} \cong A_{10}:Z_{2^{\varepsilon}}$, or $f=3$, $\varepsilon=1$ and $G_{0,B} \cong A_{10}$ by \cite{At}.

If $M \cong S_{10} \times Z_{2^{\varepsilon}}$, then $[G_{0}:M]=66$ and hence $[M:G_{0,B}]=2^{f-1}$ with $ 1 \leq f \leq 2+ \varepsilon$, and again we have $A_{10} \trianglelefteq G_{0,B}$ and $f \leq 3$. Therefore, $G_{0,B}$ contains a normal copy of $A_{10}$, say $Q$, and $f \leq 3$ in either case. Then $Q$ preserves $0^{T_{B}}$, as $T_{B} \trianglelefteq G_{B}$, hence $Q$ fixes $0^{T_{B}}$ pointwise by \cite{AtMod}, since $2 \leq t=\dim 0^{T_{B}} \leq 5$ by Lemma \ref{cici}(1), as $m=5$ and $f \leq 3$. Thus $Q$ fixes $B$ pointwise if $t \geq 4$ by Corollary \ref{p2}. If $t=2$ or $3$, then $B$ is partitioned in $2^{3}$ orbits under $T_{B}$ each of length $2^{2}$ or in $2^{2}$ orbits under $T_{B}$ each of length $2^{3}$ respectively. In both cases $Q$ fixes $B$ pointwise, since $Q$ permutes the $T_{B}$-orbits contained in $B$ and since the minimal primitive permutation representation of $Q$ is $10$. Thus $B \subseteq \mathrm{Fix(}Q \mathrm{)}$ and hence $\dim \mathrm{Fix(}Q \mathrm{)} \geq 5$. Therefore $Q$ acts on $V/ \mathrm{Fix(}Q \mathrm{)}$, which has dimension at most $5$, hence $\mathrm{Fix(}Q \mathrm{)}=V$ by \cite{AtMod}, and we reach a contradiction.

Assume that $p=3$, $m=2$ and $A_{6}\trianglelefteq G_{0} \leq S_{6}\times Z_{2}$. Then $A_{6}\cong
\Omega _{4}^{-}(3)$ and $S_{6}\times Z_{2}\cong GO_{4}^{-}(3)$ by \cite{At}.
Let $Q$ be the quadratic form preserved by $\Omega _{4}^{-}(3)$, then the $%
\Omega _{4}^{-}(3)$-orbits on $V^{\ast }$ consist of the vectors on which $Q$
takes values $0$, $1$ and $-1$ respectively. These orbits have length $20$, $30$ and $%
30$ respectively. Let $\left\langle z\right\rangle $ be a non-singular $1$%
-dimensional subspace of $V$ and let $U\cong D_{8}$ be a Sylow $2$-subgroup
of $\Omega _{4}^{-}(3)_{\left\langle z\right\rangle }\cong S_{4}$. Then $U$
permutes $\pm z$ and hence $U$ fixes at least a block $B$ incident with $\pm z$,
since $\lambda =3^{f}$. Then $U$ fixes a point in $B$ since $k=3^{2}$, which
is necessarily $0$, since $U$ is also a Sylow $2$-subgroup of $\Omega
_{4}^{-}(3)$ and the length of each $\Omega _{4}^{-}(3)$-orbit on $V^{\ast }$
is even. Since $u^{\Omega _{4}^{-}(3)}$, where $u$ is a non-zero singular
vector of $V$ contained in $B$ is also a $G_{0}$-orbit, it follows from Lemma \ref%
{PP} that $\left\vert B\cap u^{\Omega _{4}^{-}(3)}\right\vert =2$. Thus $4$ divides $%
\left\vert \Omega _{4}^{-}(3)_{B,u}\right\vert$, as $U\leq \Omega
_{4}^{-}(3)_{B}$, but this contradicts $\Omega _{4}^{-}(3)_{u}\cong F_{18}$.

Assume that $p=3$, $m=2$ and $A_{5}\trianglelefteq G_{0}\leq S_{5}\times
Z_{2}$. Then $r=30$. If $-1\notin G_{0}$, then $G_{0}$ has one orbit of
length $5$ since $V$ is the fully deleted permutation module for $S_{5}$. Then $r/3=5$ by
Lemma \ref{PP}, whereas $r=30$. Thus $-1\in G_{0}$ . Also $t>0$ by Lemma \ref%
{cici}(1), since $m=2$ and $f=1$. Then the blocks incident with $0$ are $2$%
-dimensional subspaces of $V$ by Lemma \ref{inv}. Also there is a $%
G_{0}$-orbit on $V^{\ast}$, say $\mathcal{O}$, of length $10$. Then $\left\vert B^{\prime}\cap \mathcal{O}\right\vert
=1 $ by Lemma \ref{PP}, where $B^{\prime}$ is any block incident with $0$, since $%
r=30$ and $\lambda =3$. However, this is impossible since $B^{\prime}$ is a $2$%
-dimensional subspace of $V$, $\mathcal{O}$ is a $G_{0}$-orbit and $-1\in
G_{0}$.

Assume that $p=m=3$. Then $A_{8}\trianglelefteq G_{0}\leq
S_{8}\times Z_{2} $ and $r=3^{f}\cdot 28$ with $f=1,2$%
. Note that $T$ does not act block-semiregularly on $\mathcal{D}$, since $%
t \geq m-f\geq 3-2=1$ by Lemma \ref{cici}(1). Suppose that $-1\in G_{0}$. Then the blocks of 
$\mathcal{D}$ are $3$-dimensional subspaces of $AG_{6}(3)$ by Lemma \ref{inv}%
. Let $W$ be a Sylow $5$-subgroup, then $W$ preserves at least two distinct
blocks of $\mathcal{D}$ incident with $0$, say $B_{1},B_{2}$, since $r\equiv
4,2\pmod{5}$ according to whether $f=1$ or $2$ respectively. Then $W$
fixes $B_{i}$ pointwise, $i=1,2$, since the order of $%
GL_{3}(3)$ is not divisible by $5$. Thus $B_{1}+B_{2}\subseteq 
\mathrm{Fix(}W\mathrm{)}$, since the blocks of 
$\mathcal{D}$ are $3$-dimensional subspaces of $AG_{6}(3)$. Hence $W$ fixes $V$ pointwise, since $\dim 
\mathrm{Fix(}W\mathrm{)}\geq 4$ and since $W$ acts semiregularly on $%
V\backslash \mathrm{Fix(}W\mathrm{)}$. This is a contradiction. Thus 
$-1\notin G_{0}$ and hence $A_{8}\trianglelefteq G_{0}\leq S_{8}$. If $B$ is any block incident with $0$, then $%
G_{0,B}\leq S_{6}\times Z_{\epsilon }$ by \cite{At}, where $\epsilon =1,2$ according to
whether $G_{0}$ is $A_{8}$ or $S_{8}$ respectively. Then $%
\left[ S_{6}\times Z_{\epsilon }:G_{0,x}\right] =3^{f}$, since $\left[
G_{0}:S_{6}\times Z_{\epsilon }\right] =28$, and we reach a contradiction by 
\cite{At}.

Assume that $V$ is not the fully deleted permutation module for $S$. Then either $%
G_{0}\cong A_{\ell }$, $\ell =7,8$, and $V=V_{4}(2)$, or $G_{0}\cong (Z_{2^{i}} \times Z_{3}.A_{7}).Z_{2^{j}}$, $i \leq 3$ and $j \leq 1$, and $V=V_{3}(25)$ by \cite[Theorem 3.1]{BP} (note that $\Phi _{4}^{\ast
}(7)=5^{2}$ does not divide the order of $Z_{2}.S_{7}\times Z_{3}$). Then either $r=2^{f}\cdot 5$ with 
$f\leq 2$, or $r=2 \cdot 3^{3} \cdot 5 \cdot 7$ respectively. However both cases are ruled out, since $G_{0}$ has no transitive permutation representations of
degree $r$ by \cite{At}. This completes the proof.
\end{proof}

\begin{lemma}
\label{NotSpor}$S$ is not a sporadic group.
\end{lemma}

\begin{proof}
Assume that $S$ is a sporadic group. Table \ref{T1} contains the
admissible $S$, and the corresponding $G_{0}$, which are determined by using \cite%
[Theorem 3.1]{BP} and \cite{AtMod}.%
\begin{table}[h] 
\caption{ Admissible nearly simple groups with $S$ sporadic}
\label{T1}
\centering
\begin{tabular}{|l|l|l|}
\hline
$S$ & $G_{0}$ & $V$ \\ \hline
$M_{11}$ & $M_{11}$ & $V_{10}(2)$ \\ \hline
$M_{12}$ & $M_{12}:Z_{2^{i}}$, $i=0,1$ &  \\ \hline
$M_{22}$ & $M_{22}:Z_{2^{i}}$, $i=0,1$ &  \\ \hline
$J_{2}$ & $Z_{2^{1+j}}.J_{2}.Z_{2^{i}}$, $i \leq j$, $i,j=0,1$  & $V_{6}(5)$ \\ \hline
$J_{3}$ & $Z_{3}.J_{3}:Z_{2^{i}}$, $i=0,1$ & $V_{9}(4)$ \\ \hline
\end{tabular}%
\end{table}

The fully deleted permutation module for $A_{11}$ is an irreducible module
for $M_{11}$ since $M_{11}<A_{11}$ and by \cite[Theorem 3.5]{He1/2}, since $\Phi _{10}^{\ast }(2)=11$
divides the order of $M_{11}$. On the other hand, up to equivalence, $M_{11}$
has a unique irreducible linear representation on a $10$-dimensional $GF(2)$%
-space by \cite{AtMod}. Thus $V$ is isomorphic to a fully deleted
permutation module for $A_{11}$. Then each $A_{11}$-orbit is a union of $M_{11}
$-orbits, and there is a $A_{11}$-orbit of length $11$ on $V$ which is also
a $M_{11}$-orbit. Then $r/\lambda $ divides $11$ by Lemma \ref{PP},
but this is impossible since $r/\lambda =2^{5}+1$.

Assume that $G_{0} \cong M_{12}:Z_{2^{i}}$, $i=0,1$, and that $V=V_{10}(2)$. Then $r=2^{f}\cdot 33$ with $f \leq 5$. Arguing as in the $M_{11}$-case, we see that $V$ is the fully deleted permutation module for $A_{12}$. Then there is a $A_{12}$-orbit $\mathcal{O}$ of length $66$ which is also a $M_{12}$-orbit. The stabilizer in $M_{12}$ of $x$ in $\mathcal{O}$ is a maximal subgroup of $M_{12}$ isomorphic to $P\Gamma L_{2}(9)$ by \cite{At}. Moreover, $M_{12}$ contains two conjugacy classes $\mathcal{K}$ and $\mathcal{K}^{\prime}$ of subgroups isomorphic to $P\Gamma L_{2}(9)$ which are fused in $M_{12}:Z_{2}$. Also $P\Gamma L_{2}(9)$ is maximal in $M_{12}:Z_{2}$. Thus there is a further $M_{12}$-orbit $\mathcal{O}^{\prime}$ of length $66$ on $V$, and the stabilizer in $G_{0}$ of any point of $\mathcal{O}^{\prime}$ belongs to $\mathcal{K}^{\prime}$, whereas stabilizer in $G_{0}$ of any point in $\mathcal{O}$ belongs to $\mathcal{K}$, since $M_{12}:Z_{2}$ acts on $V$. In particular, $\mathcal{O} \cup \mathcal{O}^{\prime}$ is a $(M_{12}:Z_{2})$-orbit. We denote either $\mathcal{O}$ or $\mathcal{O} \cup \mathcal{O}^{\prime}$ simply by $\mathcal{X}$ according to whether $G_{0}$ is isomorphic to $M_{12}$ or to $M_{12}:Z_{2}$ respectively. Then $\left \vert B \cap \mathcal{X} \right \vert=2$ or $4$, respectively, by Lemma \ref{PP}.     

Let $K$ be a maximal subgroup of $G_{0}$ containing $G_{0,B}$, then $\left[
G_{0}:K\right] \left[ K:G_{0,B}\right] =$ $r=2^{f}\cdot 33$, with $%
f \leq 5$. One of the following holds by \cite{At}:
\begin{enumerate}
\item $K\cong M_{11}$, $\left[ K:G_{0,B}\right] =2^{f-i
-2}\cdot 11$ and $f-i -2=0,1$;
\item $K\cong P\Gamma L_{2}(9)$, $\left[
K:G_{0,B}\right] =2^{f-i -1}$ and $f-i -1=0,1,2$.
\end{enumerate}

Therefore $(G_{0,B})^{\prime} \cong PSL_{2}(9)$ in either case again by \cite{At}. Then $(G_{0,B})^{\prime}$ fixes $B \cap \mathcal{X}$ pointwise, since $\left \vert B \cap \mathcal{X} \right \vert =2, 4$ whereas the minimal primitive permutation representation of $PSL_{2}(9)$ is $6$. Hence  $(G_{0,B})^{\prime}$ fixes at least two points in $\mathcal{X}$. This is not the case, since $N_{G_{0}}((G_{0,B})^{\prime}) \cong P\Gamma L_{2}(9)$ is a maximal subgroup of $G_{0}$ fixing a unique point in $\mathcal{X}$. 

Assume that $G_{0}\cong M_{22}:Z_{2^{i}}$, with $i=0,1$, and $V=V_{10}(2)$. Let $W$ be a maximal subgroup of $G_{0}$ containing $G_{0,B}$. Then $[G_{0}:W][W:G_{0,B}]=r=2^{f}\cdot 33$ with $f \leq 5$. Thus either $W\cong A_{7}$ and $[W:G_{0,B}]=2^{f-i-4}\cdot 3$, with $4+i \leq f \leq 5$, or $W \cong PSL_{3}(4):Z_{2^{i}}$ and $[W:G_{0,B}]=2^{f-1}\cdot 3$, with $f \leq 5$, by \cite{At}. However, this is impossible since $W$ has no such transitive permutation
representations degrees by \cite{At}.

Assume that $G_{0} \cong Z_{2^{1+j}}.J_{2}.Z_{2^{i}}$, with $i \leq j$ and $i,j=0,1$, and $V=V_{6}(5)$. Then $r=5^{f}\cdot 2 \cdot 3^{2} \cdot 7$ with $f=1,2$, since $r$ divides the order of $G_{0}$. One $G_{0}$-orbit of $1$-dimensional subspaces of $V$ is of length $2016$ by \cite[Lemma 5.1]{Lieb2}, and if $\left\langle x \right\rangle$ is any representative of such a orbit, then $E:K \leq G_{0,x}$ where $E \cong Z_{5}\times Z_{5}$ and $K \cong Z_{3}$. If $f=1$ then $E:K \leq G_{0,B}$, where $B$ is any block incident with $0,x$, since $K$ acts irreducibly on $E$. So $r$ is coprime to $5$, since $r=[G_{0}:G_{0,B}]$ and $5^{2}$ is the maximum power of $5$ dividing the order of $G_{0}$, and we reach a contradiction. Therefore $f=2$ and hence $r=3150$. Also $T_{B} \neq 1$ by Lemma \ref{cici}(1) since $m=3$. Then $B$ is a $3$-dimensional $GF(5)$-subspace of $V$ by Lemma \ref{inv}. However $G_{0}$ has no orbits of length $3150$ on the set of $3$-dimensional subspaces of $V$ by \cite{GAP}, hence this case is ruled out. 

Assume that $G_{0} \cong Z_{3}.J_{3}:Z_{2^{i}}$, with $i=0,1$, and $V=V_{9}(4)$. Then $r=2^{f}\cdot 3^{3} \cdot 19$ with $f \leq 7+i$, since $r$ divides the order of $G_{0}$. So the order of $G_{0,B}$ is divisible by $3^{3} \cdot 5\cdot 17$, which is not the case by \cite{At}. This completes the proof.
\end{proof}

\begin{lemma}
\label{Cross}If $S$ is a Lie type simple group in characteristic $p^{\prime
}$, then $S\cong PSL_{2}(9)$ and $%
\mathcal{D}$ is isomorphic to the $2$-design of constructed in Example \ref%
{Ex0} for $q=2$.
\end{lemma}

\begin{proof}
Assume that $S$ is a Lie type simple group in characteristic $p^{\prime }$. Since $\Phi^{\ast}_{2m}(p)$ divides the order of $G_{0}$, with $\Phi^{\ast}_{2m}(p)>1$, then Table \ref{T2} holds by \cite[Theorem 3.1]{BP}, \cite{BHRD} and \cite{AtMod}. The groups in lines 2--3, 13--15 and 17--18 are ruled out, since they violate the fact that $r=p^{f}(p^{m}+1)$, with $f\geq 1$, must divide the order of $G_{0}$. The group in Line 5 yields the assertion by Lemma \ref{NotFDPM}.  
\begin{table}[h!] 
\caption{Admissible nearly simple groups with $S$ a Lie type simple group in characteristic $p^{\prime }$.}
\label{T2}
\centering
\begin{tabular}{|l|l|l|}
\hline
$S$ & $G_{0}$ & $V$ \\ \hline
$PSL_{2}(7)$ &  $PSL_{2}(7): Z_{2^{i}}$, $i \leq 2$ & $V_{6}(3)$ \\ \hline
 &        $SL_{2}(7): Z_{2^{i}}, (Z_{4} \circ SL_{2}(7)): Z_{2^{i}}, i \leq 1$                                        & $V_{6}(5)$ \\ \hline
 &        $(PSL_{2}(7) \times Z_{2^{i}}): Z_{2^{j}}, i \leq 3,j \leq 1$                                        & $V_{3}(9)$ \\ \hline
             & $(PSL_{2}(7) \times Z_{\theta}):Z_{2^{i}}$, $\theta \mid 24$ & $V_{3}(25)$ \\ \hline
$PSL_{2}(9)$ & $A_{6}, S_{6}$ & $V_{4}(2)$ \\ \hline
$PSL_{2}(11)$ & $PSL_{2}(11),PGL_{2}(11)$ & $V_{10}(2)$ \\ \hline
              & $(PSL_{2}(11) \times Z_{3^{i}}):Z_{2^{j}}$, $i,j \leq 1$ & $V_{5}(4)$ \\ \hline
$PSL_{2}(13)$ & $SL_{2}(13)$ & $V_{6}(3)$ \\ \hline
& $(PSL_{2}(13)\times Z_{3^{i}}):Z_{2^{j}}$, $i,j\leq 1$ & $V_{6}(4)$ \\ \hline
$PSL_{2}(17)$ & $PSL_{2}(17)$ & $V_{8}(2)$ \\ \hline
$PSL_{2}(19)$ & $(PSL_{2}(19) \times Z_{3^{i}}):Z_{2^{j}}$, $i,j\leq 1$ & $V_{9}(4)$ \\ \hline
$PSL_{2}(5^{2})$ & $PSL_{2}(5^{2})$, $P\Sigma L_{2}(5^{2})$ & $V_{12}(2)$ \\ 
\hline
$PSL_{2}(37)$ & $PSL_{2}(37),PGL_{2}(37)$ & $V_{18}(2)$ \\ \hline
$PSL_{2}(41)$ & $PSL_{2}(41),PGL_{2}(41)$ & $V_{20}(2)$ \\ \hline
$PSL_{3}(3)$ & $PSL_{3}(3),PGL_{3}(3)$ & $V_{12}(2)$ \\ \hline
$PSL_{3}(2^{2})$ & $Z_{2}.PSL_{3}(2^{2}).K$, where $K\leq Z_{2}\times Z_{2}$
& $V_{6}(3)$ \\ \hline
$PSU_{3}(3)$ & $PSU_{3}(3)\times Z_{2^{i}}$, $G_{2}(2)\times Z_{2^{i}}$, $%
i\leq 2$ & $V_{6}(5)$ \\ \hline
$PSp_{4}(5)$ & $(PSp_{4}(5) \times Z_{3^{i}} ):Z_{2^{j}}$, $i,j \leq 1$ & $V_{12}(4)$ \\ \hline
\end{tabular}%
\end{table}

Assume that $G_{0} \cong PSL_{2}(7): Z_{2^{i}}$, $i \leq 2$, and $V=V_{6}(3)$. Then $V$ is the deleted permutation module for $(G_{0})^{\prime}$ by \cite[Table 1]{Mor} and by \cite{AtMod}. Thus there is a $G_{0}$-orbit of length $7$ or $14$ for $i\neq 2$. However this is impossible by Lemma \ref{PP}, since $r/\lambda=28$. Therefore $G_{0} \cong PSL_{2}(7): Z_{4}$ and there is a non-zero vector $x$ of $V$ such that $\left\vert x^{G_{0}}\right\vert=28$. Then $\left\vert B \cap x^{G_{0}}\right\vert=1$, where $B$ is any block incident with $0$ by Lemma \ref{PP}. Moreover $r=3\cdot 28$, since $\lambda=3^{f} > 1$. Then $T$ does not act block-semiregularly on $\mathcal{D}$ by Lemma \ref{cici}(1), since $t \geq m-f=3-1>0$, hence each block incident with $0$ is a subspace of $V$ by Lemma \ref{inv}. So $-1$ preserves $B$ and hence it fixes $B \cap x^{G_{0}}$, as $\left\vert B \cap x^{G_{0}}\right\vert=1$, and we reach a contradiction.  

Assume that $G_{0} \cong (PSL_{2}(7) \times Z_{2^{i}}): Z_{2^{j}}$, where $i \leq 3$, $j \leq 1$,                           and $V=V_{3}(9)$. Before analyzing this case some information on the $G_{0}$-orbits on $PG_{2}(9)$ are needed.

Set $Q$ the (normal) copy of $PSL_{2}(7)$ inside $G_{0}$. Then $Q<SU_{3}(3)<SL_{3}(9)$ by \cite[Tables 8.5--8.6]{BHRD}, hence $Q$ preserves an Hermitian unital $\mathcal{H}$ of order $3$ in $PG_{2}(9)$. It is not difficult to see that the $3$-elements in $Q$ are not unitary transvections. Hence, any Sylow $3$-subgroup of $Q$ fixes a point of $\mathcal{H}$ and acts semiregularly elsewhere in $PG_{2}(9)$. Also the Sylow $7$-subgroups of $Q$ do not fix points or lines of $PG_{2}(9)$ by \cite[Theorem 3.5]{He1/2}, since $\Phi^{\ast}_{6}(3)=7$. Thus $Q$ acts transitively on $\mathcal{H}$, since the normalizer in $Q$ of any Sylow $3$-subgroup of $Q$ is $S_{3}$. Now, any Sylow $2$-subgroup of $Q$ fixes a point in $PG_{2}(9) \setminus \mathcal{H}$ and hence there is $Q$-orbit in $PG_{2}(9) \setminus \mathcal{H}$ of length $21$. The remaining $42$ points of $PG_{2}(9)$ form a $Q$-orbit because $Q$ contains elementary abelian groups of order $4$ consisting of unitary homologies in a triangular configuration and a Sylow $2$-subgroup of $Q$ containing the homology group fixes one vertex of the triangle e switches the remaining two ones.

Since $Q<SU_{3}(3)$ and since there is a unique conjugacy class of involutions of $SU_{3}(3)$, it follows that each involution in $Q$ fixes a unique non-isotropic $1$-dimensional subspace of $V$ and acts semiregularly elsewhere on $V$. Let $J$ be a Sylow $2$-subgroup of $Q$ and let $\left\langle z \right\rangle$ be its fixed point in $PG_{2}(9)$. Clearly $J \cong D_{8}$ and  $\left\langle z \right\rangle \notin \mathcal{H}$. Since $J$ consists of linear maps and $J^{\left\langle z \right\rangle} \leq GL_{1}(9) \cong Z_{8}$, it follows that $J(\left\langle z \right\rangle) \cong Z_{4}$. Therefore $\left\vert z^Q \right \vert = 42$.

Since $Q_{\left \langle x \right\rangle} \cong S_{3}$, where $\left \langle x \right\rangle$ is isotropic, and since each involution in $Q$ fixes a unique non-isotropic $1$-dimensional subspace of $V$ pointwise and acts semiregularly elsewhere on $V$, it follows that the set of non-zero isotropic vectors of $V$ are partitioned into four $Q$-orbits each of length $56$. Note that, $Q:Z_{2}$ has still orbits of length $56$ and $42$, since each involution in $(Q:Z_{2})/Q$ fixes a $3$-dimensional $GF(3)$-subspace of $V$ (that is the involution is a Baer one).    

Now, we are in position to analyze the case. Since $r$ is a divisor of the order of $G_{0}$, $r=3 \cdot 28$. Note that $G_{0} \cap Z(GL_{3}(9)) \cong Z_{2^{i}}$. If $-1 \notin G_{0}$, then $i=0$ and hence $G_{0} \cong PSL_{2}(7) : Z_{2^{j}}$. We have proven that there are $G_{0}$-orbits on $V^{\ast}$ of length $56$ and $42$. Then $r/\lambda$ must divide $14$ by Lemma \ref{PP}, whereas $r/\lambda=28$. Therefore $-1 \in G_{0}$, hence any block $B$ incident with $0$ is a $3$-dimensional $GF(3)$-subspace of $V$ by Lemma \ref{inv}, since $t=m-f=3-1=2$ by Lemma \ref{cici}(1). Furthermore, no blocks incident with $0$ are $1$-dimensional $GF(9)$-subspaces of $V$ since $\lambda \geq 2$. Thus $i=1$. If $B$ is any block incident with $0$ and $x$ is a non-zero isotropic vector of $V$, then $\left \vert B \cap x^{G_{0}} \right \vert =2,8,14$ by \cite[Theorem 2.9]{BarEbe}. Then the possibilities for $\left \vert x^{G_{0}} \right\vert $ are $2 \cdot 28,8\cdot 28,14\cdot 28$ respectively by Lemma \ref{PP}. On the other hand, since $Q \times \left \langle -1 \right\rangle \trianglelefteq G_{0} \leq (Q \times \left \langle -1 \right\rangle):Z_{2}$, the stabilizer in $Q \times \left \langle -1 \right\rangle$ of $\left \langle x \right\rangle$ contains an elementary abelian group of order $4$ consisting of linear maps, hence there is a cyclic group group of order $2$ fixing $\left \langle x \right\rangle$ pointwise. Therefore, $\left \vert x^{G_{0}} \right\vert \leq 4 \cdot 28$ and combining this fact with the previous information on  $\left \vert x^{G_{0}} \right\vert$ we derive that $\left \vert x^{G_{0}} \right\vert = 2 \cdot 28$. So $B$ induces a Baer subplane of $PG_{2}(9)$ tangent to the unital $\mathcal{H}$. Since the number of blocks incident with $0$ and with any non-zero vector of $\left \langle x \right\rangle$ is $\lambda$ times the number of $1$-dimensional $GF(3)$-subspaces contained in $\left \langle x \right\rangle$, which is $3 \cdot 4$, and since $G_{0}$ induces a transitive group on $\mathcal{H}$, it follows that $b=12 \cdot 28$. However this is impossible since $b=vr/k=3^{4} \cdot 28$. Thus this case is ruled out.

Assume that $PSL_{2}(11) \trianglelefteq G_{0} \leq PGL_{2}(11)$ and $V=V_{10}(2)$. Then $V$ is the deleted permutation module for $(G_{0})^{\prime}$ by \cite[Table 1]{Mor} and by \cite{AtMod}. Thus there is a $G_{0}$-orbit of length $11$ or $22$. However this is impossible by Lemma \ref{PP}, since $r/\lambda=33$.

Assume that $G_{0} \cong (PSL_{2}(11) \times Z_{3^{i}}):Z_{2^{j}}$, $i,j \leq 1$ and $V_{5}(4)$. Then $r=2^{f}\cdot 33$, with $1 \leq f \leq 2+j$, since $r$ divides the order of $G_{0}$. Hence $Z_{5} \triangleleft G_{0,B} \leq (D_{10}:\times Z_{3^{i}}):Z_{2^{j}}$, where $B$ is a block incident with $0$, by \cite{At}. If $c$ denotes the number of blocks incident with $0$ and preserved by any fixed cyclic subgroup $C$ of order $5$, it follows that $r=[G_{0}:N_{G_{0}}(C)]\cdot c$ and hence $c=2^{f-1}$. On the other hand, $C$ fixes at least a non-zero vector $y$ of $V$ and preserves the $2^{f}$ blocks incident with $0,y$ provided that $2^{f} \leq 5$. Consequently, we have $f=3$ and hence $j=1$. Also the number of blocks preserved by $C$ and incident with $0$ and with any non-zero vector of $V$ fixed by $C$ is necessarily $3$. Let us count the pairs $(x,B^{\prime})$, where $x$ is any non-zero vector of $V$ fixed by $C$, and $B^{\prime}$ is any block incident with $0,x$ and preserved by $C$. Note that the number of non-zero vectors fixed by $C$ in any block incident with $0$ and preserved by $C$ is a constant $\mu$, as $C$ is a normal subgroup of the stabilizer of the block and as $G_{0}$ acts transitively on the set  of blocks incident with $0$. Also $\mathrm{Fix}(C)$ is a $e$-dimensional $GF(2)$-subspace of $V$, with $e \geq 1$. So $(2^{e}-1)\cdot 3=4 \cdot \mu$, which has no integer solutions.

Assume that $G_{0}\cong SL_{2}(13)$ and $V=V_{6}(3)$. Then $r=3\cdot 28$
divides the order of $G_{0}$. Thus $T$ does not act block-semiregularly on $%
\mathcal{D}$ by Lemma \ref{cici}(1), since $f=1$ and $m=3$, hence the block
incident with $0$ are $3$-dimensional subspaces of $V$ by Lemma \ref{inv}.
Moreover $G_{0,B}\cong Z_{26}$, where $B$ is a block of $\mathcal{D}$
incident with $0$, since $r=84$. The group $GL_{6}(3)$ contains a
unique conjugacy class of subgroups isomorphic to $SL_{2}(13)$, and each $%
SL_{2}(13)$ preserves a unique $3$-spread of $GL_{6}(3)$ defining the Hering
translation plane of order $3^{3}$ by \cite[Propositions 2.7-2.8]{BaPo}.
Then we may assume that $G_{0}$ is be the copy of $SL_{2}(13)$ whose generators
are those in \cite[Lemma 2.4]{BaPo}. Hence, a Sylow $13$-subgroup is
generated by $\varphi $ represented by the matrix $\mathrm{Diag(}\alpha
,\alpha ^{4}\mathrm{)}$, where 
\[
\alpha =\left( 
\begin{array}{ccc}
0 & 0 & 1 \\ 
1 & 0 & 1 \\ 
0 & 1 & 0%
\end{array}%
\right) 
\]%
Let $(x_{1},x_{2},x_{3},y_{1},y_{2},y_{3})\in B$, with $%
(x_{1},x_{2},x_{3},y_{1},y_{2},y_{3})\neq (0,0,0,0,0,0)$. Since $B$ is a
subspace of $V_{6}(3)$, and $G_{0,B}\cong Z_{26}$ acts regularly on $B^{\ast
}$, it follows that 
\[
(x_{1},x_{2},x_{3},y_{1},y_{2},y_{3})+(x_{1},x_{2},x_{3},y_{1},y_{2},y_{3})^{\varphi }=\varepsilon (x_{1},x_{2},x_{3},y_{1},y_{2},y_{3})^{\varphi ^{i}}%
\text{,}
\]%
where $\varepsilon =\pm 1$, for some $0\leq i\leq 12$. Note that $\alpha
^{3}=\alpha +1$ by \cite[Lemma 2.2.(i)]{BaPo}, hence $\alpha ^{4}=\alpha
^{2}+\alpha $ and $\alpha ^{5}=\alpha ^{4}+1$. Then 
\[
\left\{ 
\begin{array}{c}
(x_{1},x_{2},x_{3})+(x_{1},x_{2},x_{3})^{\alpha
}=(x_{1},x_{2},x_{3})^{\alpha ^{3}} \\ 
(y_{1},y_{2},y_{3})+(y_{1},y_{2},y_{3})^{\alpha
^{4}}=(y_{1},y_{2},y_{3})^{\alpha ^{5}}%
\end{array}%
\right. 
\]%
and so $(x_{1},x_{2},x_{3})^{\alpha ^{3}}=\varepsilon
(x_{1},x_{2},x_{3})^{\alpha ^{i}}$ and $(y_{1},y_{2},y_{3})^{\alpha
^{5}}=\varepsilon (y_{1},y_{2},y_{3})^{\alpha ^{4i}}$. If $%
(x_{1},x_{2},x_{3})\neq (0,0,0)$ and $(y_{1},y_{2},y_{3})\neq (0,0,0)$, then 
$i\equiv 3\pmod{13}$ and $4i\equiv 5\pmod{13}$, which cannot occur. Thus either $(x_{1},x_{2},x_{3})=(0,0,0)$ or $%
(y_{1},y_{2},y_{3})=(0,0,0)$ and hence $B$ is a component of a $3$-spread
of $GL_{6}(3)$ defining the Hering translation plane of order $3^{3}$ (see
the proof of \cite[Propositions 2.7]{BaPo}), whereas $\lambda \geq 2$.

Assume that $G_{0}\cong PSL_{2}(13)\times Z_{3^{i}}$, $i\leq 1$, acting on $%
V_{6}(4)$. It follows from \cite[p. 502]{Lieb2} that, $G_{0}$ preserves a
set of $1$-dimensional subspaces of $V_{6}(4)$ of size $91$ or $182$.
Then $r/\lambda $ divides either $91\cdot 3$ or $182\cdot 3$, respectively, by Lemma \ref{PP}. However this is impossible since $r/\lambda =4^{3}+1=\allowbreak 65$.

Assume that $G_{0} \cong PSL_{2}(17)$ and $%
V=V_{8}(2)$. Then $r =2^{f}\cdot 17$ with $1\leq f\leq m=4$. If $B$ is any block incident with $0$,
then $G_{0,B}$ contains a
Sylow $3$-subgroup $W$ of $G_{0}$. So $Z_{9}\cong W\trianglelefteq
G_{0,B}\leq D_{18}$. Then $2^{3} \mid \left \vert T_{B} \right \vert$, since $G_{B}/T_{B}$ is isomorphic to a subgroup of $G_{0}$ and contains an isomorphic  copy of $G_{0,B}$ as a subgroup of index $2^{4}/\left \vert T_{B} \right \vert$ by Lemma \ref{cici}(2). Thus $B$ is a $4$-dimensional $GF(2)$-subspace by Corollary \ref{p2}. Since $SL_{4}(2)$ does not contain cyclic subgroups of order $9$, it follows that $B \subseteq \mathrm{Fix}(W_{0})$ where $W_{0}$ is the cyclic subgroup of $W$ of order $3$. Then $W$ preserves $\mathrm{Fix}(W_{0})$ and acts semiregularly on $V \setminus \mathrm{Fix}(W_{0})$, since $W \cong Z_{9}$. However this is impossible since $\dim \mathrm{Fix}(W_{0}) \geq 4$.

Assume that $G_{0} \cong (PSL_{2}(19) \times Z_{3^{i}}):Z_{2^{j}}$, $i,j\leq 1$ and $V_{9}(4)$. Since $r/\lambda = 3^{3}\cdot 19 $ must divide the order of $G_{0}$, it follows that $j=1$, hence $G_{0}$ contains an elementary abelian subgroup $E$ of order $9$. Then there is a non-zero vector $V$, say $x$, such that $Z_{3} \leq E_{x}$ by \cite[Theorems 3.2.3 and 3.3.1]{Go}. So $\left\vert x^{G_{0}}\right\vert$ is not divisible by $3^{3}$, whereas $r/\lambda= 3^{3}\cdot 19$ must divide $\left\vert x^{G_{0}}\right\vert$ by Lemma \ref{PP}. Thus this case is excluded.   

Assume that $G_{0}\cong PXL_{2}(5^{2})$, where $X=S,\Sigma $, and $%
V=V_{12}(2) $. Then $\left\vert G_{0}\right\vert =2^{3+c}\cdot 3\cdot
5^{2} \cdot 13$, where $c=0,1$ according as $X=S$ or $X=\Sigma $ respectively. Hence $r=2^{f}\cdot 5 \cdot
13$ and $\left\vert G_{0,B}\right\vert
=2^{3+c-f}\cdot 3\cdot 5$, with $1\leq f\leq 3+c$, since $r$
divides $\left\vert G_{0}\right\vert $. Thus either $G_{0,B}\leq E_{5^{2}}:Z_{24}$, or $%
G_{0,B}< S_{5}\times Z_{2^{c}}$ by \cite{At}. The former yields $%
E_{5^{2}}\leq G_{0,B}$, as $E_{5^{2}}:Z_{24}$ is a Frobenius group and as $15$
divides $\left\vert G_{0,B}\right\vert $, but this contradicts $5^{2} \nmid%
\left\vert G_{0,B}\right\vert $. Thus $A_{5}\trianglelefteq G_{0,B}<
S_{5}\times Z_{2^{c}}$ and hence $1\leq f\leq 1+c\leq 2$. Therefore $\left\vert T_{B} \right\vert  \geq 2^{4}$ by Lemma \ref{cici}(1). Actually, either $\left\vert T_{B} \right\vert = 2^{4}$ or $B$ is a $6$-dimensional subspace of $V$ by Corollary \ref{p2}.

Let $K$ be subgroup of order $5$ of $G_{0,B}$. If $\left\vert T_{B} \right\vert = 2^{4}$, then $B$ consists of $4$ regular orbits under $T_{B}$, and $K$ fixes exactly $1$ point in each of them by \cite[Proposition 4.2]{Pass}, since $\left\vert \mathrm{Fix}(K)\right\vert=2^{4}$. Thus $K$ fixes exactly $4$ points in $B$. The same conclusion holds when $B$ is a $6$-dimensional subspace of $V$, since $K <SL_{6}(2)$. Thus $\left\vert B \cap \mathrm{Fix}(K)\right\vert=2^{2}$ in each case. On the other hand, since $\left\vert G_{0,B}\right\vert=60\cdot 2^{\theta}$, where $0 \leq \theta \leq 1+c$, the number of blocks incident with $0$ and preserved by $K$ is equal to $$\frac{\left\vert N_{G_{0}}(K)\right\vert \cdot \left\vert K^{G_{0}} \cap G_{0,B} \right\vert}{\left\vert G_{0,B}\right\vert}=2^{1+c - \theta},$$ since $N_{G_{0}}(K) \cong F_{20} \times Z_{2^{c}}$ and $\left\vert K^{G_{0}} \cap G_{0,B} \right\vert=6$. Also $K$ preserves the $2^{f}$ blocks incident with $0$ and with a non-zero vector of $\mathrm{Fix}(K)$, since $f \leq 2$. Hence, by counting the number of pairs $(B^{\prime},x)$, where $B^{\prime}$ is a block incident with $0$ and preserved by $K$ and where $x \in B^{\prime}\cap \mathrm{Fix}(K)$, $x\neq 0$, we obtain $2^{1+c-\theta}\cdot (2^{2}-1)=2^{f}\cdot (2^{4}-1)$, which is clearly impossible. 

Finally, assume that $G_{0}\cong Z_{2}.PSL_{3}(2^{2}).K$, where $K\leq Z_{2}\times
Z_{2}$, and $V=V_{6}(3)$. Thus $\left\vert G_{0}\right\vert =2^{7+c}\cdot
3^{2}\cdot 5\cdot 7$, with $c\leq 2$, hence $r=3^{f}\cdot 28$ with $%
f=1,2 $, since $r$ divides $\left\vert G_{0}\right\vert $. Therefore $T$ does not
act block-semiregularly on $\mathcal{D}$ by Lemma \ref{cici}(1), since $%
f\leq 2$ and $m=3$, hence each block $B$ incident with $0$ is a $3$-dimensional
subspace of $V$ by Lemma \ref{inv}, since $-1\in G_{0}$. Moreover $B\subseteq \mathrm{Fix}(U)$, where $U$ is any subgroup of $G_{0,B}$ of order $5$, since $5$ does not divide the order of $%
GL_{3}(3)$. Then $i \geq 3$, where $i=\dim \mathrm{Fix}(U)$. So $U$ acts semiregularly on $V \setminus B$ and hence $5\mid 3^{6}-3^{i}$, which is a contradiction as $3 \leq i <6$.

\end{proof}

\begin{proposition}
\label{NChar}Let $S$ be a Lie type simple group in characteristic $p$. If $%
H=TH_{0}$, where $H_{0}=(G_{0})^{\infty }$, then $H$ acts flag-transitively
on $\mathcal{D}$ and one of the following holds:

\begin{enumerate}
\item $H_{0}\cong G_{2}(q)$, $q=2^{h}$, and $\mathcal{D}$ is a $2$-$%
(q^{6},q^{3},q^{3})$ design isomorphic to that constructed in Example %
\ref{Ex1}.

\item $H_{0}\cong Sz(q)$, $q=2^{h}$ with $h$ odd and $h>1$, and $\mathcal{D}
$ is a $2$-$(q^{4},q^{2},\lambda)$ design with $\lambda \mid q^{2}$. Moreover, if $B$ is any block of 
$\mathcal{D}$ incident with $0$, either $Z_{q-1}\leq H_{0,B}\leq N_{H_{0}}(W)$,
where $W$ is a Sylow $2$-subgroup of $H_{0}$, or $Z_{q-1}\trianglelefteq
H_{0,B}\leq D_{2(q-1)}$.
\end{enumerate}
\end{proposition}

\begin{proof}
Assume that $S$ and $H$ are as in the Proposition's statement. By \cite[Theorem 3.1]{BP} one of the following holds:

\begin{enumerate}
\item[(i).] $V=V_{6}(q)$, $q=2^{h}$ and $H_{0}\cong G_{2}(q)^{\prime }$;

\item[(ii).] $V=V_{4}(q)$, $q=2^{h}$, $h\equiv 2\pmod{4}$ and $H_{0}\cong
Sz(q^{1/2})$;

\item[(iii).] $V=V_{4}(q)$, $q=2^{h}$, $h\equiv 1\pmod{2}$ and $%
H_{0}\cong Sz(q)$.
\end{enumerate}

Since $G_{0}\leq \left( Z_{2^{h}-1}\times H_{0}\right) .Z_{h/\epsilon }$,
where $\epsilon =2$ in (ii) and $\epsilon =1$ in (i) or (iii), it follows that $%
\frac{\left[ G_{0}:H_{0}\right] }{\left[ G_{0,B}:H_{0,B}\right] }\mid \frac{h%
}{\epsilon }(2^{h}-1)$. Also, since $r=\frac{\left[
G_{0}:H_{0}\right] }{\left[ G_{0,B}:H_{0,B}\right] }\left[ H_{0}:H_{0,B}%
\right] =2^{f}(2^{m}+1)$, we get 
\[
\frac{\left[ G_{0}:H_{0}\right] }{\left[ G_{0,B}:H_{0,B}\right] }\mid \left( 
\frac{h}{\epsilon }(2^{h}-1),2^{f}(2^{m}+1)\right) \text{.} 
\]%
Since $(2^{m}+1,2^{h}-1)=1$, as $h\mid m$ being $n=4,6$, it follows that $\frac{\left[
G_{0}:H_{0}\right] }{\left[ G_{0,B}:H_{0,B}\right] }=2^{u}$ for some $u\leq
f $ in case (i), whereas $r=\left[ G_{0}:G_{0,B}\right] =\left[ H_{0}:H_{0,B}%
\right] $, and hence $H=TH_{0}$ acts flag-transitively on $\mathcal{D}$ in
cases (ii) and (iii).

Assume that Case (i) holds. Then $\left[ H_{0}:H_{0,B}\right]
=2^{f-u}(2^{3h}+1)$. Let $M$ be a maximal subgroup of $H_{0}$ containing $%
H_{0,B}$. Then $M$ is one of the groups listed in \cite[Table 8.30]{BHRD} (or in \cite[Theorems 2.3 and 2.4]{Coop}). If $M$ is isomorphic either to 
$PSL_{2}(13)$ or to $J_{2}$, then $h=2$ and hence $\left[ H_{0}:H_{0,B}%
\right] =2^{f-u}\cdot 5\cdot 13$. The former is ruled out since $\left[
H_{0}:M\right] =2^{10}\cdot 3^{2} \cdot 5^{2}$ does not divide $\left[ H_{0}:H_{0,B}%
\right] $, the latter yields $\left[ M:H_{0,B}\right] \mid 2^{f-u}\cdot
5\cdot 13$ but $J_{2}$ has no such transitive permutation representations
degrees by \cite{At}.

Assume that $M\cong G_{2}(2^{h/x})$ with $x$ a prime. Then $\frac{%
2^{6h}(2^{2h}-1)(2^{6h}-1)}{2^{6\frac{h}{x}}(2^{2\frac{h}{x}}-1)(2^{6\frac{h%
}{x}}-1)}$ must divide $2^{f-u}(2^{3h}+1)$ and so $x=2$ as $6(h-h/x)\leq
f\leq 3h$. However $\left[ M:H_{0,B}\right] =2^{3h}(2^{h}+1)(2^{3h}+1)$ does
not divide $2^{f-u}(2^{3h}+1)$.

Assume that $M$ is isomorphic either to $[2^{5h}]:GL_{2}(2^{h})$ or to $%
SL_{2}(2^{h})\times SL_{2}(2^{h})$, or to $SU_{3}(2^{h}).Z_{2}$. Then $\left[
M:H_{0,B}\right] $ is divisible either by $\frac{2^{6h}-1}{2^{2h}-1}$ or by $%
2^{2h-1}(2^{3h}-1)$, and none of these integers divides $2^{f-u}(2^{3h}+1)$, hence these cases are ruled out.

Finally, assume that $M\cong SL_{3}(2^{h}).Z_{2}$.
Since $\left[ H_{0}:M\right] =2^{3h-1}(2^{3h}+1)$ must divide $\left[ H_{0}:H_{0,B}%
\right] =2^{f-u}(2^{3h}+1)$, it follows that $3h\geq f\geq u+3h-1$ and hence 
$u\leq 1$. Thus $\left[ M:H_{0,B}\right] \leq 2$, as $\left[ M:H_{0,B}\right]
=2^{f-u-3h+1}$, and hence $H_{0,B}\cong SL_{3}(2^{h}).Z_{2^{\varepsilon }}$,
where $\varepsilon =0,1$.

Since $H$ is a normal subgroup of $G$ acting $2$-transitively on $V$ and
since $\left[ H_{0}:H_{0,B}\right] =2^{3h-\varepsilon +1}(2^{3h}+1)$, it
follows that $\mathcal{D}_{\varepsilon -1}=(V_{6}(2^{h}),B^{H})$, where $B$
is any block of $\mathcal{D}$, is a $2$-$(2^{6h},2^{3h},2^{3h-\varepsilon
+1})$ design admitting $H$ as a point-$2$-transitive, flag-transitive
automorphism group.

Note that $M=\left( H_{0,B}\right) ^{\prime }:\left\langle \sigma
\right\rangle $, with $\sigma $ of order $2$, is maximal in $H_{0}$. Then $M$
preserves the decomposition $V=U_{1}\oplus U_{2}$, where $U_{1},U_{2}$ are
two $3$-dimensional totally isotropic subspaces of $V$ with respect to a
symplectic form preserved by $G_{2}(q)$. If $\varphi \in \left( H_{0,B}\right) ^{\prime }$, then $\varphi$ preserves $U_{1},U_{2}$ and $\varphi _{\mid U_{2}}$ is the inverse transpose of $%
\varphi _{\mid U_{1}}$. In \ particular $\left( H_{0,B}\right) ^{\prime }$
acts naturally on $U_{1}$ and on $U_{2}$, whereas the outer automorphism $\sigma $ of $\left( H_{0,B}\right) ^{\prime }$ switches $U_{1}$ and $U_{2}$ (see \cite[Lemma
7.4.5 and proof therein]{BHRD}). Thus $V=U_{1} \oplus U_{1}^{\sigma}$.

The group $H_{B}/T_{B}$ is isomorphic to a subgroup of $H_{0}$ of order $%
2^{3h-t}\left\vert H_{0,B}\right\vert $, where $2^{t}=\left\vert
T_{B}\right\vert $, by Lemma \ref{cici}(2). Then $t \geq 3h+\varepsilon -1$,
since $H_{0,B}\cong SL_{3}(2^{h}).Z_{2^{\varepsilon }}$ and $H_{B}/T_{B}\leq
SL_{3}(2^{h}).Z_{2}$, hence $t=3h$ and $B$ is a $3h$-dimensional $GF(2)$-subspace of $V$ by Lemma \ref{p2}.

Let $Q$ be a Sylow $2$-subgroup of $(H_{0,B})^{\prime}$ normalized by $\sigma$. Then $\mathrm{Fix}(Q)=\left\langle
x\right\rangle \oplus \left\langle x^{\sigma }\right\rangle $ for some non-zero vector $x$ of $U_{1}$. Moreover there is cyclic subgroup $K$ of $N_{H_{0,B}}(Q)$
of order $2^{h}-1$ preserving $\mathrm{Fix}(Q)$, inducing the scalar multiplication on $\left\langle
x\right\rangle$ and on $\left\langle x^{\sigma}\right\rangle $ and permuting regularly the remaining $2^{h}-1$ one dimensional subspaces. Thus either $\mathrm{Fix}(Q) \subseteq B$, or $\varepsilon=1$ and $\mathrm{Fix}(Q) \cap B$ is one of the spaces $\left\langle x\right\rangle$ or $\left\langle x^{\sigma}\right\rangle $, since $B$ is a $GF(2)$-subspace of $V$ and $K \leq H_{0,B}$. Since $H_{0,B}^{\prime}$ acts transitively on $U_{1}^{\ast}$ and on $(U_{1}^{\sigma})^{\ast}$ and since $k=q^{3}$, the former case is ruled out, whereas the latter yields $B=U_{1}$ or $B=U_{1}^{\sigma}$ respectively. In each case $%
B^{H}=B^{G}$ and hence $\mathcal{D}$ is isomorphic to the $2$-design
constructed in Example \ref{Ex1}. This proves (1).

Case (ii) is ruled out since the order of $Sz(2^{h/2})$, which is $%
2^{h}(2^{h}+1)(2^{h/2}-1)$, is not divisible by $r/2^{f}=$ $2^{2h}+1$.

Assume that Case (iii) holds. Recall that $H$ acts flag-transitively
on $\mathcal{D}$. Hence $2^{2h-f}(2^{h}-1)$, with $f\leq 2h$, divides the
order of $H_{0,B}$ since $r=2^{f}(2^{2h}+1)$. Then $2^{2h-f}(2^{h}-1)$
divides the order of any maximal subgroup $R$ of $H_{0}$ containing $H_{0,B}$. The possibilities for $R$ are listed in \cite[Table 8.16]{BHRD}, or in 
\cite{Suz}. Clearly, $R$ is not isomorphic to $Z_{2^{h}\pm
2^{(h+1)/2}+1}:Z_{4}$. If $M\cong Sz(2^{h_{0}})$, where $h/h_{0}$ is an odd
prime and $h_{0}>1$, then $2^{2h-f}(2^{h}-1)$ divides $%
2^{2h_{0}}(2^{2h_{0}}+1)(2^{h_{0}}-1)$, which is a contradiction. Then $%
Z_{2^{h}-1}\leq H_{0,B} \leq R$ and hence either $R=N_{H_{0}}(W)$, where $W$ is a Sylow 
$2$-subgroup of $H_{0}$, or $R\cong D_{2(2^{h}-1)}$. Thus we obtain (2).
\end{proof}

\bigskip

\subsection{The Suzuki group}\label{Sz}

In this subsection we focus on Case (2) of Proposition \ref{NChar}. In order to do so, some useful facts about $\Sigma =Sz(q)$, $q=2^{2e+1}$ and $e\geq 1$, need to be recalled.

Let $\sigma :x\rightarrow x^{2^{e+1}}$ be an automorphism of $%
GF(q) $, and let%
\bigskip
\tiny
\begin{equation}\label{Sugr}
\varphi (l,w)=\left( 
\begin{array}{cccc}
1 & 0 & 0 & 0 \\ 
l & 1 & 0 & 0 \\ 
l^{1+\sigma }+w & l^{\sigma } & 1 & 0 \\ 
l^{2+\sigma }+lw+w^{\sigma } & w & l & 1%
\end{array}%
\right) \text{, }\psi (m)=\left( 
\begin{array}{cccc}
m^{2+\sigma } & 0 & 0 & 0 \\ 
0 & m^{\sigma } & 0 & 0 \\ 
0 & 0 & m^{-\sigma } & 0 \\ 
0 & 0 & 0 & m^{-2-\sigma }%
\end{array}%
\right) \text{, }\phi =\left( 
\begin{array}{cccc}
0 & 0 & 0 & 1 \\ 
0 & 0 & 1 & 0 \\ 
0 & 1 & 0 & 0 \\ 
1 & 0 & 0 & 0%
\end{array}%
\right) \text{ }
\bigskip 
\end{equation}
\normalsize

where $l,w,m\in GF(q)$, $m\neq 0$. By \cite{Suz} (see also \cite[Theorem IV.24.2]{Lu}), the following hold:

\begin{enumerate}
\item[(i).] $U=\left\{ \varphi (l,w):l,w\in GF(q)\right\} $ is a Sylow $2$%
-subgroup of $\Sigma $.\ It has order $q^{2}$, exponent $4$ and $Z(U)=\left\{ \varphi (0,w):w\in GF(q)\right\} $;

\item[(ii).] $N_{\Sigma }(U)=U:K$ is a Frobenius group, where $K=\left\{ \psi (m):m\in GF(q)^{\ast }\right\} $ is a cyclic group of order $q-1$;

\item[(iii).]  $\phi $ is involutory and normalizes $K$. Also $D=K\left\langle \beta
\right\rangle $ is a dihedral group of order $2(q-1)$ and is maximal in $\Sigma$;

\item[(iv).] $K^{\Sigma }$, $D^{\Sigma }$ are the unique conjugacy classes of cyclic subgroups of
order $q-1$ \ and dihedral subgroups of order $2(q-1)$ in $\Sigma $ by \cite%
[Table 8.16]{BHRD}. Also $\Sigma =\left\langle U,K,\phi
\right\rangle $.

\item[(v).] The set $\mathcal{S}=\left\{ L\left( \infty \right) ,L(l,w):%
\text{ }l,w\in GF(q)\right\} $, where
\begin{eqnarray*}
L\left( \infty \right) &=&\left\{ \left( x,y,0,0\right) :x,y\in GF(q)\right\}
\\
L(l,w) &=&\left\{ \left( x\left(l^{\sigma +2}+lw+  w^{\sigma }\right)
+y(l^{\sigma +1}+w),xw+yl^{\sigma },xl+y,x\right) :x,y\in GF(q)\right\},
\end{eqnarray*}%
is a $2$-spread of $V$ defining the L\"{u}neburg translation plane of order $%
q^{2}$, and $\Sigma $ acts $2$-transitively on $\mathcal{S}$ (more details
on the L\"{u}neburg translation plane are provided in \cite[Chapter IV]{Lu}).$\allowbreak $
\end{enumerate}

\bigskip

\emph{Throughout this subsection we will use properties (i)--(v) of $\Sigma $
without recalling them.} The following fact about the group $T\Sigma $ is also needed.

\bigskip

\begin{proposition}
\label{trisesira} Each subgroup of $T\Sigma $ isomorphic to $UK$, and not contained in an element of $\Sigma ^{T}$, has a unique
point-orbit of length $q^{2}$ in the L\"{u}neburg translation plane of order 
$q^{2}$ and that is a line.
\end{proposition}

\begin{proof}
We prove the assertion in four steps.

\begin{enumerate}
\item \textbf{$T\Sigma $ contains $q$ conjugacy classes of
subgroups isomorphic to $\Sigma $ and $q-1$ of these are
fused in $T:(Z \times \Sigma)$, where $Z=Z(GL_{4}(q))$.}
\end{enumerate}

$T\Sigma $ contains $q$ conjugacy classes of subgroups isomorphic to $\Sigma 
$, each of length $q^{4}$, by \cite[Theorem 2]{Be} and \cite[17.7]{Asch2}. Let $\Sigma _{0},\Sigma _{1}...,\Sigma _{q-1}$
be the representatives of such classes, where $\Sigma _{0}=\Sigma $. Then $Z$
preserves $\Sigma _{0}^{T\Sigma }$ and permutes the remaining classes.
Suppose that $Z$ does not act transitively on $\left\{ \Sigma _{1}^{T\Sigma
},...,\Sigma _{q-1}^{T\Sigma }\right\} $. Then there is subgroup $%
\left\langle \xi \right\rangle $ of $Z$ preserving $\Sigma _{i}^{T\Sigma }$
for some $1\leq i\leq q-1$. Since $\left\vert \Sigma _{i}^{T\Sigma
}\right\vert =q^{4}$, $\left\langle \xi \right\rangle $ normalizes an
element of $\Sigma _{i}^{T\Sigma }$. Without loss we may assume that the normalized element is $\Sigma _{i}$.

Let $\Pi (V,\mathcal{S})$ be the projective extension of the L\"{u}neburg
translation plane $\mathcal{A}(V,\mathcal{S})$ of order $q^{2}$ preserved by 
$T\Sigma $, where $\Sigma $ fixes the origin $O$ of $\mathcal{A}(V,\mathcal{S})$%
, and $\ell _{\infty }$ be the line at infinity of $\mathcal{A}(V,\mathcal{S}%
)$. Clearly $Z$ is a collineation group of $\mathcal{A}(V,\mathcal{S})$, and
hence of $\Pi (V,\mathcal{S})$, by \cite[Theorem I.1.10]{Lu}. Since $\Sigma
_{i}\notin \Sigma _{0}^{T\Sigma }$, the group $\Sigma _{i}$ has a unique
line-orbit $\mathcal{O}$ of $\Pi (V,\mathcal{S})$ of length $q^{2}+1$ by 
\cite[Theorem IV.28.11]{Lu}, since $\Sigma _{i}$ preserves $\ell _{\infty }$%
. Moreover, $\mathcal{O}$ is a line-oval and $\ell _{\infty }$ is its
nucleus, and both are preserved by $\xi $. Also $\xi $ fixes each line lying
in $\mathcal{O}$, since $\xi $ fixes $\ell _{\infty }$ pointwise and since
each point of $\ell _{\infty }$ is the intersection point of $\ell _{\infty }$ with a unique line belonging to $\mathcal{O}$. Then each line in $\mathcal{O}$
is incident with $O$, as this one is the unique affine point of $\Pi (V,\mathcal{S})$ fixed by $\xi $%
, but this contradicts the fact that $\mathcal{O}$ is a line-oval. Thus $Z$ acts
transitively on $\left\{ \Sigma _{1}^{T\Sigma },...,\Sigma _{q-1}^{T\Sigma
}\right\} $ and hence (1) holds.

\begin{enumerate}
\item[2.] \textbf{Let $C$ be any subgroup of $T\Sigma $ isomorphic to $UK$, then $N_{T\Sigma }(C)=C$.}
\end{enumerate}

Let $C$ be any subgroup of $T\Sigma $ isomorphic to $UK$. Assume that $C\neq
N_{T\Sigma }(C)$. Then there is $\alpha \in T\Sigma \setminus C$ such that $C^{\alpha
}=C$. Since $C$ has a unique conjugate class of subgroups isomorphic to $K$
by \cite[17.10]{Asch2}, it follow that $R^{\alpha \beta^{-1} }=R$ for some subgroup $R$
of $C$ isomorphic to $K$ and for some $\beta \in C$. Clearly $R$ fixes a point $x$ of $V$ and hence it is a conjugate of $K$ in $T \Sigma$. Thus $\mathrm{Fix}(R)=\left\{ x\right\} $ and hence $\alpha\beta^{-1} $ fixes $x$ with $ \alpha \beta^{-1} \in  N_{T\Sigma }(C) \setminus C$. So, $R\left\langle \alpha \beta^{-1} \right \rangle$ fixes $x$. Then $N_{T\Sigma }(C)/(N_{T\Sigma
}(C)\cap T)$ is isomorphic to a subgroup of $\Sigma $ and contains a subgroup
isomorphic $R\left\langle \alpha \beta^{-1} \right \rangle$. Then $N_{T\Sigma }(C)/(N_{T\Sigma
}(C)\cap T)$ is a dihedral group of order $2(q-1)$ by \cite[Table 8.16]{BHRD}. Thus $R \leq C/(C\cap
T)\leqslant \left\langle R,\alpha \beta^{-1} \right\rangle $ and hence $C$ contains an
elementary abelian subgroup of index at most $2$, but this contradicts \cite%
[Theorem IV.24.2.(b)]{Lu}, as $C \cong UK$ and $q>2$. Therefore $N_{T\Sigma }(C)=C$.

\begin{enumerate}
\item[3.] \textbf{$T\Sigma $ contains $q$ conjugacy classes of
subgroups isomorphic to $UK$ and $q-1$ of these are
fused in $T:(Z \times \Sigma)$.}
\end{enumerate}

Let $C_{i}$ be a subgroup of $\Sigma _{i}$ isomorphic to $UK$ and assume
that $C_{i}\leq \Sigma _{i}^{\alpha }$ for some $\alpha \in T \Sigma$. Then the
maximality of $C_{i}$ in $\Sigma _{i}$ implies $C_{i}=\Sigma _{i}^{\alpha }\cap
\Sigma _{i}$. Then $\alpha =\alpha _{0}\tau $ for some $\alpha _{0}\in
\Sigma _{i}$ and $\tau \in T$, since $T \Sigma = T \Sigma_{i}$, hence $C_{i}=\Sigma _{i}^{\tau }\cap \Sigma
_{i}$. Then $C_{i}^{\tau }=\Sigma _{i}\cap \Sigma _{i}^{\tau }=C_{i}$, as $o(\tau
)=2 $, and hence $\tau \in N_{T\Sigma }(C_{i})$. Then $\tau \in C_{i}$, as $%
N_{T\Sigma }(C_{i})=C_{i}$ by (2), and hence $\Sigma _{i}^{\tau }=\Sigma _{i}$ and $%
\Sigma _{i}^{\alpha }=\Sigma _{i}$. Thus each element of $%
C_{i}^{T\Sigma } $ is contained in a unique element of $\Sigma _{i}^{T\Sigma }$%
, hence the number of conjugacy classes under $T \Sigma$ of subgroups isomorphic to $UK$ is
at most $q$.

It can be deduced from the proof of \cite[Theorem 2]{Be} that $T \Sigma$ contains exactly $q$ conjugacy classes of subgroups isomorphic to $UK$. Thus $C_{i}^{T\Sigma}$, where $C_{i}\leq \Sigma _{i}$ for $%
i=0,...,q-1$, are the $q$ conjugacy classes of subgroups isomorphic to $UK$.
Then $T:(Z \times \Sigma)$ acts transitively on $\left\{ C_{1}^{T\Sigma },...,C_{q-1}^{T\Sigma
}\right\} $ by (1), and we obtain (3).

\begin{enumerate}
\item[4.] \textbf{The proposition's statement holds.}
\end{enumerate}

Let $C$ be any subgroup of $T\Sigma $ isomorphic to $UK$. By (3) and since $C$ is not contained in any element of $\Sigma_{0}^{T}$, we may
assume that $C=C_{i}$ for $1\leq i\leq q-1$. Since $\Sigma _{i}<T\Sigma $, it
follows that $\Sigma _{i}$ preserves the L\"{u}neburg translation plane of
order $\mathcal{A}(V,\mathcal{S})$ of order $q^{2}$. The line-$\Sigma _{i}$%
-orbits on $\mathcal{A}(V,\mathcal{S})$ have length $q^{2}+1$, $%
(q^{2}+1)(q-1)$ and $(q^{2}+1)q(q-1)$ by \cite[Theorem IV.28.11]{Lu}, since $%
\Sigma _{i}$ preserves the line at infinity of $\mathcal{A}(V,\mathcal{S})$.
Thus $C_{i}$ preserves exactly one line $\ell $ of $\mathcal{A}(V,\mathcal{S}%
)$.

Let $P\in \ell $, then $\left\vert C_{i,P}\right\vert \geq q-1\geq 7$ as 
$\left\vert \ell \right\vert =q^{2}$ and $q\geq 8$. On the other hand, $%
\left\vert C_{i,P}\right\vert $ divides $2(q-1)$, since $C_{i,P}\leq \Sigma
_{i,P}$ and since the point-$\Sigma _{i}$-orbits have length $\frac{1}{2}%
q^{2}(q^{2}+1)$ and $\frac{1}{4}q^{2}(q-1)(q\pm \sqrt{2q}+1)$ again by \cite%
[Theorem IV.28.11]{Lu}. Since $C_{i}$ is a Frobenius group with kernel
isomorphic to $U$, $C_{i}$ does not contain subgroups isomorphic to $%
D_{2(q-1)}$ and hence $\left\vert C_{i,P}\right\vert =q-1$. Thus $\ell $ is a $%
C_{i}$-orbit of length $q^{2}$ in $\mathcal{A}(V,\mathcal{S})$. Actually $%
\ell $ is the unique $C_{i}$-orbit of length $q^{2}$ in $\mathcal{A}(V,%
\mathcal{S})$, since any cyclic subgroup of $C_{i}$ fixes exactly one point
in the $\Sigma _{i}$-orbit of length $\frac{1}{2}q^{2}(q^{2}+1)$ and no
points in the complementary set of the orbit. This proves (4).
\end{proof}

\bigskip
Let $B$ be a block of $\mathcal{D}$ incident with $(0,0,0,0)$ and denote by $B_{0}$ and $B_{\infty}$ the subsets $B \cap L(0,0)$ and $B\cap L(\infty)$ respectively. \emph{We may assume that $B$ is such that either $K\leq H_{0,B}\leq N_{H_{0}}(U)=U:K$ , or $H_{0,B}=K\left\langle \phi \right\rangle $ by Proposition \ref{NChar}(2) and by (iv). Moreover, if $H_{0,B}\leq U:K$ then $H_{0,B}$ is isomorphic to one
of the groups $K$, $Z(U):K$ or $U:K$ by (ii).}

\bigskip

Note that $H$ is a primitive rank $3$ group and the $H_{0}$-orbits on $%
V^{\ast }$ have length $(q^{2}+1)(q-1)$ and $(q^{2}+1)q(q-1)$. If $%
\mathcal{O}$ denotes the short length orbit ($\mathcal{O}$ induces the Tits Ovoid on $%
PG_{3}(q)$), then $P=(x,y,z,t)$ lies in $\mathcal{O}$ if, and only if,
\medskip%
\begin{equation} \label{Ovoid}
P= \left\lbrace
\begin{array}{ll}
(c,0,0,0)& \text{ with } c \in GF(q)^{\ast } \\
(0,0,0,d)& \text{ with } d \in GF(q)^{\ast } \\ 
\frac{1}{m^{2+\sigma }}\left( l^{\sigma +2}+lw+w^{\sigma },w,l,1\right)& \text{ with } m,w,l\in
GF(q),m\neq 0,(l,w)\neq (0,0).\\ 
\end{array}
\right.
\end{equation}
\medskip
Since $r\left\vert B\cap \mathcal{O}\right\vert =(q^{2}+1)(q-1)\lambda $ by Lemma \ref{PP} and
since $r/\lambda =q^{2}+1$, it follows that $\left\vert B\cap \mathcal{O}%
\right\vert =q-1$ and hence that $\left\vert B\cap V \setminus \mathcal{O}%
\right\vert =q(q-1)$.

\bigskip
\begin{lemma}
\label{solo2} The following hold:
\begin{enumerate}
\item If $H_{0,B}\neq U:K$.
\item If $H_{0,B}=Z(U):K$, then $\mathcal{D}$ is isomorphic to the $2$-design constructed in Example \ref{Ex2} (Family 1).
\end{enumerate}
\end{lemma}

\begin{proof}
Assume that $H_{0,B}=UK$. Since each $H_{0,B}$-orbit on $V\setminus L(\infty
)$ is of length $q^{2}(q-1)$ and since $k=q^{2}$, it follows that $%
B=L(\infty )$. Thus $\mathcal{D}$ is the L\"{u}neburg translation plane but
this contradicts the assumption $\lambda \geq 2$. Thus $H_{0,B}\neq UK$.

Assume that $H_{0,B}=Z(U)K$. Then $B_{\infty }=B\cap \mathcal{O}%
=\left\langle (1,0,0,0)\right\rangle _{GF(q)}^{\ast }$ since $\left\vert B\cap \mathcal{O}%
\right\vert =q-1$, since the $H_{0,B}$%
-orbits on $V\setminus L(\infty )$ are of length $q(q-1)$ and those in $%
L(\infty )^{\ast }$ are of length $q-1$, and since $L(\infty)\cap \mathcal{O}%
=\left\langle (1,0,0,0)\right\rangle _{GF(q)}^{\ast }$.

Since $H_{B}/T_{B}$ is isomorphic to a subgroup of $H_{0}$ and contains a copy
of $H_{0,B}$ as a subgroup of index $q^{2}/\left\vert T_{B}\right\vert $ by
Lemma \ref{cici}(2), since $UK$ is the unique maximal subgroup of $H_{0}$
containing $Z(U)K$ by \cite[Table 8.16]{BHRD}, and since $U/Z(U):K$ is a
Frobenius group, it results $\left\vert T_{B}\right\vert =q$ or $q^{2}$.

Assume that $\left\vert T_{B}\right\vert =q$. Then $B_{\infty}=(0,0,0,0)^{T_{B}}$. Since $H_{B}$ acts transitively on $B$, since $T_{B} \trianglelefteq H_{B}$ and since $K$ acts semiregularly on $B \setminus B_{\infty}$, it follows that
\begin{equation}
B \setminus B_{\infty} = \bigcup_{m \in GF(q)^{\ast}} B_{\infty}+(x_{0},y_{0},z_{0},t_{0})^{\psi(m)}.
\end{equation}
Then $Z(U)$ preserves each of the $q-1$ orbits under $T_{B}$ partitioning $B \setminus B_{\infty}$, namely the sets $B_{\infty}+(x_{0},y_{0},z_{0},t_{0})^{\psi(m)}$ with $m \in GF(q)^{\ast}$, since $Z(U)$ preserves $B_{\infty}$ and normalizes $T_{B}$. Thus $T_{B}Z(U)$ is the kernel of the action of $H_{B}$ on the $q$ orbits under $T_{B}$ partitioning $B$, since $H_{B}$ acts $2$-transitively on such a partition. Therefore, $T_{B}Z(U) \triangleleft H_{B}$. Actually $T_{B}Z(U) \triangleleft H_{B} \leq T: UK$, since $H_{B}$ permutes $B_{\infty}+(x_{0},y_{0},z_{0},t_{0})^{\phi(m)}$ and hence the parallel lines $L(\infty)+(x_{0},y_{0},z_{0},t_{0})^{\phi(m)}$  of the L\"{u}neburg translation plane.

Let $Q$ be the Sylow $2$-subgroup $H_{B}$ and let $\alpha \in Q\setminus T_{B}Z(U)$. There is $\tau=\tau _{(v_{1},v_{2},v_{3},v_{4})}$ in $T$ and $\alpha_{0}=\varphi (l,w)$ in $U$ such that $\alpha =\alpha _{0} \tau$. Since $\alpha$ normalizes $T_{B}$ and $T_{B}Z(U)$ and since for each $\varphi (0,e)$ in $Z(U)$ we have that 
\[
\varphi (0,e)^{\alpha}=\varphi (0,e)^{\tau
_{(v_{1},v_{2},v_{3},v_{4})}}=\varphi (0,e)\tau _{(ev_{3}+e^{\sigma
}v_{4},ev_{4},0,0)},
\]
it follows that $v_{4}=0$.

The fact that $H_{B}/T_{B}Z(U)$ is Frobenius group of order $q(q-1)$, and its kernel is an elementary abelian $2$-group, implies $\alpha ^{2}\in Z(U)T_{B}$. 
On the other
hand $\alpha ^{2}=(\alpha _{0}\tau)^{2}=\alpha _{0}^{2}[\alpha
_{0},\tau _{0}]$ with $\alpha _{0}^{2}\in Z(U)$ and $[\alpha
_{0},\tau _{0}]\in T$. Hence $[\alpha
_{0},\tau _{0}] \in T_{B}$ as $[\alpha
_{0},\tau _{0}]=\alpha _{0}^{-2}\alpha^{2} \in H_{B}$. Since
\begin{equation}\label{SlaUkr}
(x,y,z,t)^{[\alpha_{0},\tau _{0}]}=\left( x+lv_{2}+(w+l^{\sigma
+1})v_{3},y+l^{\sigma }v_{3},z,t\right), 
\end{equation}
$[\tau _{0},\alpha _{0}]\in T_{B}$ if, and only if, $l^{\sigma }v_{3}=0$. If $l=0$, then $\alpha _{0}\in Z(U)\leq H_{0,B}$ and hence $\tau _{0}\in
H_{B}$ as $\alpha \in H_{B}$. Then $\tau _{0}\in T_{B}$ and so $\alpha \in
T_{B}Z(U)$, which is a contradiction. Thus $l\neq 0$ and hence $v_{3}=0
$. So $\tau \in T_{B}$, $\alpha _{0}\in H_{0,B}$ and
hence $\alpha _{0}\in Z(U)$ and $\alpha \in T_{B}Z(U)$, and we again reach a
contradiction.
 
Assume that $\left\vert T_{B}\right\vert =q^{2}$. Then $B$ is a $GF(2)$%
-subspace of $V$. Since $B^{\ast }$ consists of $B_{\infty }$ and one $%
H_{0,B}$-orbit of length $q(q-1)$ contained in $V\setminus L(\infty )$,
there is $(\bar{x},\bar{y},\bar{z},\bar{t})$ in $V\setminus L(\infty )$ such
that%
\small
\[
B \setminus B_{\infty} = \left\{
(m^{\sigma +2}\left( \bar{x}+w\bar{z}+w^{\sigma }\bar{t}\right) ,m^{\sigma
}\left( \bar{y}+w\bar{t}\right) ,m^{-\sigma }\bar{z},m^{-2-\sigma }\bar{t}%
):w,m\in GF(q),m\neq 0\right\} \text{.}
\]%
\normalsize
Since $B$ is a $GF(2)$-subspace of $V$, for each $c,w_{1},m_{1}\in GF(q)$,
with $c,m_{1}\neq 0$, there must be $w_{2},m_{2}\in GF(q),m_{2}\neq 0$, such that
\begin{equation*}
\left\lbrace
\begin{array}{rcl}
c+m_{1}^{\sigma +2}\left( \bar{x}+w_{1}\bar{z}+w_{1}^{\sigma }\bar{t}\right) &=& m_{2}^{\sigma
+2}\left( \bar{x}+w_{2}\bar{z}+w_{2}^{\sigma }\bar{t}\right) \\
m_{1}^{\sigma }\left( \bar{y}+w_{1}\bar{t}\right) & =& m_{2}^{\sigma }\left( \bar{y}+w_{2}\bar{t}\right) \\
m_{1}^{-\sigma }\bar{z}&=& m_{2}^{-\sigma }\bar{z}\\ 
m_{1}^{-2-\sigma }\bar{t} &=& m_{2}^{-2-\sigma }\bar{t}\text{.}\\
\end{array}
\right.
\end{equation*}
Note that $(\bar{z},\bar{t})\neq (0,0)$, since $(\bar{x},\bar{y},\bar{z},%
\bar{t})\in V\setminus L(\infty )$. If $\bar{t} \neq 0$ then $m_{1}=m_{2}$, $%
w_{1}=w_{2}$ and hence $c=0$, which is not the case. Therefore $\bar{t} = 0$ and $\bar{z} \neq 0$.

Again since $B$ is a $GF(2)$%
-subspace of $V$, for $m_{1},m_{2},w_{1},w_{2} \in GF(q)$, with $m_{1},m_{2}\neq 0$ and $m_{1} \neq m_{2}$, there are $m_{3},w_{3} \in GF(q)$, with $m_{3} \neq 0$, such that

\begin{equation*}
\left\lbrace
\begin{array}{rcl}
(m_{1}+m_{2})^{\sigma +2}\left( \bar{x}+w_{1}\bar{z}\right) &=& m_{3}^{\sigma
+2}\left( \bar{x}+w_{2}\bar{z}\right) \\
(m_{1}+m_{2})^{\sigma }\bar{y} & =& m_{3}^{\sigma }\bar{y} \\
(m_{1}+m_{2})^{-\sigma }\bar{z}&=& m_{3}^{-\sigma }\bar{z}\text{,}\\
\end{array}
\right.
\end{equation*}   
since both $\sigma$ and $2+\sigma$ are automorphisms of $GF(q)$ (see \cite[Lemma IV.2.1]{Lu}). If $\bar{y} \neq 0$, then $m_{1}+m_{2}=m_{3}$ and $m_{1}^{-1}+m_{2}^{-1}=m_{3}^{-1}$, as $\bar{z}\neq 0$, hence the order of $m_{1}m_{2}^{-1}$ is $3$, whereas $q \equiv 2 \pmod{3}$.  
Therefore $\bar{y}=\bar{t}=0$ and hence $B=\left\langle(1,0,0,0),(0,0,1,0) \right\rangle_{GF(q)}$, which induces a line of $PG_{3}(q)$ tangent to the Tits Ovoid. Thus $\mathcal{D}=(V,B^{H})=(V,B^{G})$ is isomorphic to the $2$-design constructed in Example \ref{Ex2} (Family 1).
\end{proof}
\bigskip
On the basis of Lemma \ref{solo2}, in the sequel \emph{we may assume that $H_{0,B}=K$ or $H_{0,B}=K\left\langle \phi \right\rangle$}. Hence either $\lambda=q^{2}$ or $\lambda=q^{2}/2$ respectively, since $r=[H_{0}:H_{0,B}]$ and $r=\lambda (q^{2}+1)$.
\bigskip

\begin{lemma}
\label{BigBig}If $\left\vert T_{B}\right\vert =q$, then either $B_{0}=(0,0,0,0)^{T_{B}}$ or $B_{\infty }=(0,0,0,0)^{T_{B}}$.
\end{lemma}

\begin{proof}
Recall that $B$ is a block of $\mathcal{D}$ incident with $(0,0,0,0)$ and preserved by $K$. Then $K$ preserves $%
(0,0,0,0)^{T_{B}}$, as $T_{B}\trianglelefteq H_{B}$. Also $(0,0,0,0)^{T_{B}}=\left\{
(0,0,0,0)\right\} \cup (x,y,z,t)^{K}$, where $(x,y,z,t)$ is a non-zero vector of $V$
contained in $(0,0,0,0)^{T_{B}}$, since $K$ acts semiregularly on $V^{\ast }$ and $\left\vert T_{B}\right\vert =q$. Since $(0,0,0,0)^{T_{B}}$ is a $GF(2)$-subspace of $V$, there are $\psi (m_{1}),\psi (m_{2})\in K$ such that $(x,y,z,t)^{\psi
(m_{1})}+(x,y,z,t)^{\psi (m_{2})}=(x,y,z,t)$. Hence
\bigskip
\tiny 
\begin{eqnarray*}
(x,y,z,t)^{\psi (m_{1})}+(x,y,z,t)^{\psi (m_{2})} &=&\left( \left(
m_{1}^{2+\sigma }+m_{2}^{2+\sigma }\right) x,(m_{1}^{\sigma }+m_{2}^{\sigma
})y,(m_{1}^{-\sigma
}+m_{2}^{-\sigma })z,(m_{1}^{-2-\sigma }+m_{2}^{-2-\sigma })t\right) \\
&=&\left( \left( m_{1}+m_{2}\right) ^{2+\sigma }x,(m_{1}+m_{2})^{\sigma
}y,(m_{1}^{-1}+m_{2}^{-1})^{\sigma
}z,(m_{1}^{-1}+m_{2}^{-1})^{2+\sigma }t\right),
\bigskip
\bigskip
\end{eqnarray*}%
\normalsize
since $2+\sigma $ is an automorphism of $GF(q)$.

Assume that $(z,t)=(0,0)$. Then $(0,0,0,0)^{T_{B}}\subseteq L(\infty )$ and hence $%
T_{B}$ preserves $B_{\infty}$. Hence $B_{\infty}$ is a union of $%
T_{B}$-orbits each of length $q$, as $T_{B}$ acts point-semiregularly on $%
\mathcal{D}$. Since $T_{B}:K\leq H_{B}$, with $T_{B}\trianglelefteq H_{B}$,
and since $K$ acts semiregularly on $L(\infty )^{\ast }$, it follows that $K$
fixes $(0,0,0,0)^{T_{B}}$ and permutes semiregularly the other possible $T_{B}$-orbits
contained in $B_{\infty}$. Thus either $B_{\infty}=0^{T_{B}}$ or $B_{\infty}$ consists of $0^{T_{B}}$
and other $q-1$ orbits under $T_{B}$. The latter yields $B=L(\infty )$,
hence $\mathcal{D}$ is the L\"{u}neburg translation plane of order $q^{2}$%
, but this contradicts the assumption $\lambda \geq 2$. Therefore $B_{\infty}=(0,0,0,0)^{T_{B}}$.

Assume that $(z,t)\neq (0,0)$. If $x,t\neq 0$. Then $\left(
m_{1}+m_{2}\right) ^{2+\sigma }=(m_{1}^{-1}+m_{2}^{-1})^{2+\sigma }=1$ and
hence $m_{1}+m_{2}=1$ and $m_{1}^{-1}+m_{2}^{-1}=1$. Then the order of $%
m_{1} $ is $3$, whereas $q\equiv 2\pmod{3}$. If $x,z\neq 0$, then $\left(
m_{1}+m_{2}\right) ^{2+\sigma }=(m_{1}^{-1}+m_{2}^{-1})^{\sigma }=1$, hence $m_{1}+m_{2}=1$ and $m_{1}^{-1}+m_{2}^{-1}=1$, and we again reach a
contradiction. Therefore $x=0$, as $(z,t)\neq (0,0)$. Similarly, $y,z\neq 0$
and $y,t\neq 0$ yield a contradiction, hence $y=0$. Thus $%
(0,0,0,0)^{T_{B}}\subseteq L(0,0)$ since $K$ preserves $L(0,0)$. Now, we may argue
as above, with $L(0,0)$ in the role of $L(\infty )$, and we obtain $
B_{0}=(0,0,0,0)^{T_{B}}$ in this case.
\end{proof}

\begin{lemma}
\label{oh}$B$ is a $2h$%
-dimensional $GF(2)$-subspace of $V$.
\end{lemma}

\begin{proof}
Recall that $H_{0,B}$ is one of
the groups $K$ or $K\left\langle \phi \right\rangle $ by Lemma \ref{solo2}.

Assume that $H_{0,B}=K\left\langle \phi \right\rangle $. The group $%
H_{B}/T_{B}$ is isomorphic to a subgroup of $H_{0}$ and contains an
isomorphic copy of $H_{0,B}$ as subgroup of index $2^{2h-e}$, where $2^{e}=\left\vert
T_{B}\right\vert $, by Lemma \ref{cici}(2). Then $e=2h$, since $H_{0,B}$ is
maximal in $H_{0}$, hence $B$ is a $GF(2)$-vector subspace of $V$ in this
case.

Assume that $H_{0,B}=K$. If $T$ acts block-semiregularly on $\mathcal{D}$, then $H_{B}$ is
isomorphic to a subgroup $H_{0}$ containing a copy of $H_{0,B}$ as a
subgroup of index $q^{2}$. Then $H_{B}\cong UK$ by \cite[Table 8.16]{BHRD}, hence $B$ is a line of the L\"{u}neburg translation plane by Proposition %
\ref{trisesira}, since $H_{0,B}=K$. So $T_{B}$ is a group of order $q^{2}$, but this
contradicts our assumption. Thus $T$ does not act block-semiregularly on $%
\mathcal{D}$ and hence $\left\vert W\right\vert >1$, where $W=(0,0,0,0)^{T_{B}}$.
Let $x$ be a non-zero vector of $W$. Since $K$ acts semiregularly on $V^{\ast }$, it
follows that $2^{h}-1\mid 2^{e}-1$ where $2^{h}=q$ and $\left\vert
W\right\vert =\left\vert T_{B}\right\vert =2^{e}$. Then either $e=h$ and $%
\left\vert W\right\vert =q$, or $e=2h$ and $\left\vert W\right\vert =q^{2}$,
since $h\mid e$ and $e\leq 2h$.

Assume that $\left\vert W\right\vert =q$. Then either $W=B_{\infty}$ or $W=B_{0}$. by Lemma \ref{BigBig}. Since $\phi $
switches $L(0,0)$ and $L(\infty )$, normalizes $K$ and since $B^{\phi}$ is a block of $\mathcal{D}$ preserved by $K$, we may assume that $W=B_{\infty}$.

Since $H_{B}/T_{B}$ is isomorphic to a subgroup of 
$H_{0}$ containing $H_{0,B}$ as a subgroup of index $q$, being $\left\vert
T_{B}\right\vert =q$, it follows from \cite[Table 8.16]{BHRD} that $%
H_{B}/T_{B}\cong Z(U):K$. Then $H_{B}$ contains a unique Sylow $2$-subgroup,
and we denote this one by $X$. Therefore $T_{B}\cap Z(X)\neq 1$, as $%
T_{B}\trianglelefteq X$. The group $K$ normalizes both $Z(X)$ and $T_{B}$ and hence $T_{B}\cap Z(X)$.
Moreover, $K$ acts regularly on $T_{B}^{\ast }$ by \cite[Proposition
4.2]{Pass}, as $K$ acts semiregularly on $V^{\ast }$ and hence on $W^{\ast }$. Thus 
$T_{B}\leq Z(X)$ since $\left\vert T_{B}\right\vert =q$, hence either $T_{B}=Z(X)$ or $X$ is abelian, since $%
H_{B}/T_{B}\cong Z(U):K$ is a Frobenius group.

Let $\alpha \in X \setminus T_{B}$, then $\alpha =\tau \alpha _{0}$ for some $\tau \in
T$ and $\alpha _{0}\in H_{0}$. Both $\alpha $ and $\alpha _{0}$ centralize $%
T_{B}$, since $T_{B}\leq Z(X)$. Thus $\alpha _{0}$ fixes $(0,0,0,0)^{T_{B}}$
pointwise by \cite[Proposition 4.2]{Pass}, hence $\alpha _{0}\in Z(U)$.
If $\alpha _{0}=1$ then $\alpha =\tau \in T_{B}$, since $\alpha \in X\leq
H_{B}$, which is not the case. Thus $o(\alpha _{0})=2$ and hence $\alpha
_{0}=\varphi (0,w_{0})$ for some $w_{0}\in GF(q)$ and $w_{0}\neq 0$. Let $\tau
:(x,y,z,t)\rightarrow (x,y,z,t)+(x_{0},y_{0},z_{0},t_{0})$, where $%
x_{0},y_{0},z_{0},t_{0},w_{0}\in GF(q)$, then 
\[
\alpha :(x,y,z,t)\rightarrow (x+x_{0}+w_{0}\left( z+z_{0}\right)
+w_{0}^{\sigma }\left( t+t_{0}\right) ,y+y_{0}+w_{0}\left( t+t_{0}\right)
,z+z_{0},t+t_{0})\text{.} 
\]

Assume that $T_{B}=Z(X)$. Since $\alpha ^{2}:(x,y,z,t)\rightarrow
(x+w_{0}z_{0}+w_{0}^{\sigma }t_{0},y+w_{0}t_{0},z,t)$, with $w_{0}\neq 0$, it follows that $t_{0}=z_{0}=0$ since $o(\alpha _{0})=2$. Then $\alpha $
preserves $L(\infty )$, $B$ and hence $W$, as $W=B_{\infty}$. So $T_{B}\left\langle \alpha
\right\rangle \cap H_{0,B}\neq 1$, whereas $H_{0,B}=K$, and we reach a
contradiction.

Assume that $X$ is abelian. Since $X\trianglelefteq H_{B}$, it follows that $%
\alpha ^{\psi }\in X$ for each $\psi \in K$. Moreover, $\alpha \alpha ^{\psi
}=\alpha ^{\psi }\alpha $. Let $\psi =\psi (m)$, where $%
m$ is a primitive element of $GF(q)$. It follows that $\alpha ^{\psi }$ maps $%
(x,y,z,t)$ onto $(x^{\prime },y^{\prime },z^{\prime },t^{\prime })$, where%
\[
\left\{ 
\begin{array}{l}
x^{\prime }=x+m^{2+\sigma }x_{0}+m^{\sigma +2}w_{0}(z_{0}+m^{\sigma
}z)+m^{2+\sigma }w_{0}^{\sigma }(t_{0}+m^{2+\sigma }t) \\ 
y^{\prime }=y+m^{\sigma }y_{0}+m^{\sigma }w_{0}(t_{0}+m^{2+\sigma }t) \\ 
z^{\prime }=z+m^{-\sigma }z_{0} \\ 
t^{\prime }=t+m^{-2-\sigma }t_{0}.%
\end{array}%
\right. 
\]%
Then $\alpha \alpha ^{\psi }=\alpha ^{\psi }\alpha $ occurs if, and only if, 
\begin{equation}
\left\{ 
\begin{array}{r}
m^{2}w_{0}z_{0}(m^{3\sigma+2}-1)+w_{0}^{\sigma }t_{0}(m^{3\sigma+6}-1)=0 \\ 
w_{0}t_{0}\left( m^{3\sigma +4}-1\right) =0%
\end{array}%
\right.  \label{GlUk}
\end{equation}%
If $m^{3\sigma +4}=1$ then $2^{2e+1}-1 \mid 3\cdot2^{e+1}+4$, since $m$ is a primitive element of $GF(q)$, where $q=2^{2e+1}$ and $e \geq 1$, and since $m^{\sigma}=m^{2^{e+1}}$. However this is impossible. This forces $t_{0}=0$. Similarly, $m^{3\sigma +2} \neq 1$ and hence $z_{0}=0$. Thus $\left\vert W\right\vert =q^{2}$, which is the assertion.
\end{proof}

\bigskip
An immediate consequence of Lemma \ref{oh} is that both $B_{0}$ and $B_{\infty}$ are $GF(2)$-subspaces of $V$.  
\bigskip

\begin{corollary}\label{OneTwoBoth}
 At least one among $B_{0}$ and $B_{\infty}$ is of dimension $h$ over $GF(2)$.
\end{corollary}

\begin{proof}
Suppose that $\dim_{GF(2)}B_{0}=\dim_{GF(2)}B_{\infty}=0$. The group $%
K $ preserves the hyperplane $\mathcal{H}$ of $V$ of equation $y=0$ and
hence the $GF(2)$-space $B\cap \mathcal{H}$. Then either $\dim_{GF(2)} B\cap 
\mathcal{H}=h$ or $B\subset \mathcal{H}$ since $K$ acts semiregularly on $%
V^{\ast }$. If $\dim_{GF(2)} B\cap \mathcal{H}=h$, then 
\[
B\cap \mathcal{H}=\left\{ (0,0,0,0),(x_{0},0,z_{0},t_{0})^{\psi (m)}:m\in
GF(q)^{\ast }\right\} 
\]%
for some $(x_{0},0,z_{0},t_{0})\in V$, with $x_{0}\neq 0$ \ and $%
(z_{0},t_{0})\neq (0,0)$, since  $B_{0}=B_{\infty}=\left\{ (0,0,0,0) \right\} $. Then, for each $m_{1},m_{2}\in GF(q)^{\ast }$, $m_{1}\neq m_{2}$
there is a unique $m_{3}\in GF(q)^{\ast }$ such that%
$
(x_{0},0,z_{0},t_{0})^{\psi (m_{1})}+(x_{0},0,z_{0},t_{0})^{\psi
(m_{2})}=(x_{0},0,z_{0},t_{0})^{\psi (m_{3})}\text{.} 
$
A similar argument to that of Lemma \ref{BigBig} leads to a contradiction in this case, since $x_{0}\neq 0$ and $%
(z_{0},t_{0})\neq (0,0)$. Thus $B\subset \mathcal{H}$ and hence $B_{0}\neq \left\{
(0,0,0,0) \right\} $ since $L(0,0)\subset \mathcal{H}$, since $\dim_{GF(2)}B=\dim_{GF(2)} L(0,0)=2h$ by Lemma \ref{oh}, and since $\dim_{GF(2)} 
\mathcal{H}=3h$. However this contradicts our assumptions. Thus, at least one among $B_{0}$ and $B_{\infty}$ is of $GF(2)$-dimension greater that $0$.

Suppose that $\dim_{GF(2)} B_{0}>0$. Clearly $\dim_{GF(2)} B_{0}<2h$ as $\lambda \geq 2$, hence $\dim_{GF(2)} B_{0}=h$ since $K$ preserves $L(0,0)$ and acts
semiregularly on $V^{\ast }$. A similar conclusion holds for $\dim_{GF(2)} B_{\infty}>0$. 
\end{proof}

\begin{lemma}\label{block}
$$B=\left\lbrace  
(m_{1}^{\sigma+2}x_{0},m_{1}^{\sigma }y_{0},m_{2}^{-\sigma)}z_{0},m_{2}^{-\sigma-2
}t_{0}) : \; m_{1},m_{2}\in GF(q) \right\rbrace$$ for some $x_{0},y_{0},z_{0},t_{0} \in GF(q)$ such that $(x_{0},y_{0})\neq (0,0)$ and $(z_{0},t_{0})\neq (0,0)$.
\end{lemma}

\begin{proof}
It follows from Corollary \ref{OneTwoBoth} that at least one among $B_{0}$ and $B_{\infty}$ is of dimension $h$ over $GF(2)$. Possibly by substituting $B$ with $B^{\phi}$, we may actually assume that $\dim_{GF(2)} B_{\infty}=h$. Indeed, this is possible since $B$ is $K$-invariant and since $\phi$ normalizes $K$ and switches $L(0,0)$ and $L(\infty )$. Then there are $x_{0},y_{0}\in GF(q)$, with $%
(x_{0},y_{0})\neq (0,0)$ such that 
\[
B_{\infty }=\left\{ (0,0,0,0),(x_{0},y_{0},0,0)^{\psi (m)}:m\in
GF(q)^{\ast }\right\} \text{.}
\]%
Since $k=q^{2}$, there is $(x_{1},y_{1},z_{1},t_{1})\in B \setminus B_{\infty}$.
Then $B_{\infty}+(x_{1},y_{1},z_{1},t_{1})\subseteq B$ since $B$ is a $GF(2)$%
-subspace of $V$. If $B_{\infty}+(x_{1},y_{1},z_{1},t_{1})^{\psi
(m_{0})}=B_{\infty}+(x_{1},y_{1},z_{1},t_{1})$ for some $m_{0}\in GF(q)^{\ast }$,
then 
\[
L(\infty )+(x_{1},y_{1},z_{1},t_{1})^{\psi (m_{0})}=L(\infty
)+(x_{1},y_{1},z_{1},t_{1})
\]%
and hence $\psi (m_{0})$ preserves the line of $L(\infty
)+(x_{1},y_{1},z_{1},t_{1})$ of the L\"{u}neburg plane. Then $%
(x_{1},y_{1},z_{1},t_{1})=(0,0,0,0)$, since the unique affine lines preserved by $K$
are $L(0,0)$, $L(\infty )$, and we reach a contradiction. Thus 
\[
B\setminus B_{\infty}=\bigcup_{m\in GF(q)^{\ast
}}B_{0}+(x_{1},y_{1},z_{1},t_{1})^{\psi (m)}
\]%
since $K$ preserves $B,B_{\infty}$ and since $k=q^{2}$. As $B/B_{\infty}$ is a vector space
over $GF(2)$, for $m_{1},m_{2}\in GF(q)^{\ast }$, $m_{1}\neq
m_{2}$, there is $m_{3}\in GF(q)^{\ast }$ such that 
\[
(x_{1},y_{1},z_{1},t_{1})^{\psi (m)}+(x_{1},y_{1},z_{1},t_{1})^{\psi
(m_{2})}+(x_{1},y_{1},z_{1},t_{1})^{\psi (m_{3})}\in B_{\infty}\text{.}
\]%
and this is equivalent to say that%
\begin{equation}\label{Sys}
\left\{ 
\begin{array}{rcl}
x_{1}\left( m_{1}^{\sigma +2}+m_{2}^{\sigma +2}+m_{3}^{\sigma +2}\right) &=& x_{0}m^{\sigma +2} \\ 
y_{1}\left( m_{1}^{\sigma }+m_{2}^{\sigma }+m_{3}^{\sigma }\right) &=& y_{0}m^{\sigma } \\  
z_{1}\left( \left( m_{1}^{-1}\right) ^{\sigma }+\left( m_{2}^{-1}\right)^{\sigma }+\left( m_{3}^{-1}\right) ^{\sigma }\right) &=& 0 \\
t_{1}\left( \left( m_{1}^{-1}\right) ^{\sigma +2}+\left( m_{2}^{-1}\right)^{\sigma +2}+\left( m_{3}^{-1}\right) ^{\sigma +2}\right) &=& 0 \\
\end{array}%
\right.   
\end{equation}%
for some $m\in GF(q)^{\ast }$. Since $(x_{1},y_{1},z_{1},t_{1})\in
B\backslash B_{0}$, it follows that $(z_{1},t_{1})\neq (0,0)$. Then $%
m_{3}^{-1}=m_{1}^{-1}+m_{2}^{-1}$, since $\sigma,\sigma +2$ are automorphisms of $%
GF(q)$. If $m_{3}=m_{1}+m_{2}$, then $%
m_{1}^{2}+m_{1}m_{2}+m_{2}^{2}=0$ and hence $m_{2}^{-1}m_{1}$ is an element
of $GF(q)^{\ast }$ of order $3$, but this is impossible since $q\equiv 2\pmod{3}$. Then $m_{3}\not=m_{1}+m_{2}$ and from (\ref{Sys}) we obtain%
\[
x_{1}=x_{0}\left( \frac{m}{m_{1}+m_{2}+m_{3}}\right) ^{\sigma +2}\text{ and }%
y_{1}=y_{0}\left( \frac{m}{m_{1}+m_{2}+m_{3}}\right) ^{\sigma }\text{.}
\]%
Thus $(x_{1},y_{1},0,0)\in B_{\infty}$ and hence $%
B_{\infty}+(x_{1},y_{1},z_{1},t_{1})=B_{\infty}+(0,0,z_{1},t_{1})$. Therefore $(0,0,z_{1},t_{1})\in B_{0}$. Then $B_{0}\neq
\left\{ (0,0,0,0)\right\} $ and so $\left\vert B_{0}\right\vert =q$
since $B_{0}$ is a $GF(2)$-subspace of $V$ and since $K$ preserves $%
L(0,0)$ and acts semiregularly on $V^{\ast }$. Thus%
\begin{equation*}
B=\left\lbrace  
(m_{1}^{\sigma+2}x_{0},m_{1}^{\sigma }y_{0},m_{2}^{-\sigma}z_{1},m_{2}^{-\sigma-2
}t_{1}) : \; m_{1},m_{2}\in GF(q) \right\rbrace,
\end{equation*} 
since $B$ is a $GF(2)$-subspace of $V$ by Lemma \ref{oh}, which is the assertion with $z_{1}$ and $t_{1}$ simply denoted by $z_{0}$ and $t_{0}$ respectively.
\end{proof}

\begin{proposition}
\label{IntOrb}$\left\vert B\cap \mathcal{O}\right\vert =q-1$ if, and only if, $x_{0},y_{0},z_{0},t_{0}\neq 0$ and one of the following holds:
\begin{enumerate}
\item  $x_{0}z_{0}^{\sigma +1}=y_{0}^{\sigma +1}t_{0}$;

\item $x_{0}z_{0}^{\sigma +1}\neq y_{0}^{\sigma +1}t_{0}$ and the map $\zeta
:X\rightarrow \frac{(y_{0}/t_{0})^{\sigma }X^{\sigma }+(z_{0}/t_{0})^{\sigma
+2}}{(x_{0}/t_{0})X^{\sigma }+(z_{0}/t_{0})(y_{0}/t_{0})}$ of $P\Gamma
L_{2}(q)$ fixes exactly one point in $PG_{1}(q)\setminus \{0,\infty\}$.
\end{enumerate}
\end{proposition}

\begin{proof}
Since $K$ acts semiregularly on $B^{\ast }$, where $B$ is as in Lemma \ref{block}, it follows that 
\[
\mathcal{R}=\left\{ \left( x_{0},y_{0},0,0\right) ,\left( m_{1}^{\sigma
+2}x_{0},m_{1}^{\sigma }y_{0},z_{0},t_{0}\right) :m_{1}\in GF(q)\right\} 
\]%
is a system of distinct representatives of the $K$-orbits on $B^{\ast }$.
Hence $\left\vert B\cap \mathcal{O}\right\vert =q-1$ is equivalent to $\left\vert B\cap \mathcal{R}\right\vert =1$, since $\mathcal{O}$ is a $H_{0}$-orbit.

Assume that $t_{0}=0$. Then $\left( m_{1}^{\sigma +2}x_{0},m_{1}^{\sigma
}y_{0},z_{0},0\right) \in \mathcal{O}$ if and only if $m_{1}^{\sigma
+2}x_{0}=1$ and $y_{0}=z_{0}=0$. So $\left\vert B\right\vert =q
$, which is not the case. 

Assume that $t_{0}\neq 0$. The previous argument rules out the case $x_{0}=y_{0}=z_{0}=0$. Therefore $(x_{0},y_{0},z_{0}) \neq (0,0,0)$ and hence $\left( m_{1}^{\sigma +2}x_{0},m_{1}^{\sigma
}y_{0},z_{0},t_{0}\right) \in \mathcal{O}$ if, and only if, 
\[
\left( m_{1}^{\sigma +2}x_{0},m_{1}^{\sigma }y_{0},z_{0},t_{0}\right) =\frac{%
1}{m^{2+\sigma }}\left( w^{\sigma }+l^{\sigma +2}+lw,w,l,1\right) \text{.}
\]%
Suppose that $z_{0}=0$. Then $l=0$, $m_{1}^{\sigma+2}x_{0}=w^{\sigma}t_{0}$ and $m_{1}^{\sigma}y_{0}=wt_{0}$. Also $(0,0,0,1)\in B \cap \mathcal{R}$ as $t_{0} \neq 0$. Then $\left\vert B\cap \mathcal{R}\right\vert =1$ implies $y_{0}\neq 0$, since $(x_{0},y_{0}) \neq (0,0)$. If $x_{0}=0$, then $B= \left\langle (1,0,0,0),(0,0,1,0) \right\rangle^{\phi}_{GF(q)}$ but in this case $H_{0,B} \nleq D_{2(q-1)}$ (see (\ref{Ex2}) (Family 1)). Thus $x_{0},y_{0},t_{0} \neq 0$ and hence $m_{1}^{\sigma}=t_{0}^{1-\sigma}y_{0}^{\sigma}/x_{0}$ and $w=y_{0}^{1+\sigma}/x_{0}t_{0}^{\sigma}$ provide a further point of $B\cap \mathcal{R}$. So this case is excluded.

Suppose that $z_{0}\neq 0$. Then we may argue as above with $B^{\phi}$ in the role of $B$ thus obtaining $x_{0},y_{0} \neq 0$. Therefore we have proven that $\left\vert B\cap \mathcal{R}\right\vert =1$ (and $H_{0,B} \leq D_{2(q-1)}$) implies $x_{0},y_{0},z_{0},t_{0} \neq 0$. Thus $B_{0}\cap \mathcal{R} = \varnothing$ and hence $m_{1},w \neq 0$. So, $\left\vert B\cap \mathcal{R}\right\vert =1$ if, and only if, there are unique $l,m,w\in GF(q)$,
with $m\neq 0$, such that $m^{2+\sigma }=1/t_{0}$ and $l=z_{0}/t_{0}$
and the ratios $\frac{l^{\sigma +2}+lw+w^{\sigma }}{m_{1}^{\sigma +2}}$ and $%
w/m_{1}^{\sigma }$ are $x_{0}/t_{0}$ and $y_{0}/t_{0}$ respectively. Some computations leads to
\begin{equation}
\left( (x_{0}/t_{0})m_{1}^{2}+(z_{0}/t_{0})(y_{0}/t_{0})\right)
m_{1}^{\sigma }=(y_{0}/t_{0})^{\sigma }m_{1}^{2}+(z_{0}/t_{0})^{\sigma +2}%
\text{.}  \label{ug1}
\end{equation}  
\begin{enumerate}
\item[(I).] If $(x_{0}/t_{0})m_{1}^{2}+(z_{0}/t_{0})(y_{0}/t_{0})=0$, then $m_{1}^{2}=\frac{y_{0}z_{0}}{t_{0}x_{0}}$ as $%
x_{0},t_{0}\neq 0$. On the other hand $%
m_{1}^{2}=\frac{(z_{0}/t_{0})^{\sigma +2}}{(y_{0}/t_{0})^{\sigma }}=\frac{%
z_{0}^{\sigma +2}}{t_{0}^{2}y_{0}^{\sigma }}$ by (\ref{ug1}), hence $\frac{y_{0}z_{0}}{%
t_{0}x_{0}}=m_{1}=\frac{z_{0}^{\sigma +2}}{t_{0}^{2}y_{0}^{\sigma }}$. So, in this case
$\left\vert B\cap \mathcal{R}\right\vert =1$ occurs if and only if
$x_{0}z_{0}^{\sigma +1}=y_{0}^{\sigma +1}t_{0}$, with $x_{0},y_{0},z_{0},t_{0} \neq 0$, which is (1).

\item[(II).] If $(x_{0}/t_{0})m_{1}^{2}+(z_{0}/t_{0})(y_{0}/t_{0})\neq 0$,
then $$m_{1}^{\sigma }=\frac{(y_{0}/t_{0})^{\sigma }(m_{1}^{\sigma })^{\sigma
}+(z_{0}/t_{0})^{\sigma +2}}{(x_{0}/t_{0})(m_{1}^{\sigma })^{\sigma
}+(z_{0}/t_{0})(y_{0}/t_{0})}.$$ 
\begin{enumerate}
\item[(a).] If $(y_{0}/t_{0})^{\sigma
+1}+(x_{0}/t_{0})(z_{0}/t_{0})^{\sigma +1}=0$, then $t_{0}y_{0}^{\sigma
+1}=z_{0}^{\sigma +1}x_{0}$ and hence $x_{0}\neq 0$ and $m_{1}^{\sigma }=%
\frac{z_{0}}{t_{0}^{\sigma +1}y_{0}}$, which is still (1).
\item[(b).] If $(y_{0}/t_{0})^{\sigma +1}+(x_{0}/t_{0})(z_{0}/t_{0})^{\sigma +1}\neq 0$,
then $m_{1}^{\sigma }$ is a point of $PG_{1}(q)$ fixed by 
\[
\zeta :X\rightarrow \frac{(y_{0}/t_{0})^{\sigma }X^{\sigma
}+(z_{0}/t_{0})^{\sigma +2}}{(x_{0}/t_{0})X^{\sigma
}+(z_{0}/t_{0})(y_{0}/t_{0})},
\]%
with $\zeta \in P\Gamma L_{2}(q)$. Clearly $\left\vert B\cap \mathcal{R}%
\right\vert =\left\vert \mathrm{Fix}(\zeta )\right\vert $. It is worth nothing that $x_{0},z_{0} \neq 0$ rule out the possibility for $\zeta$ to fix $\infty,0$ respectively. Hence $\left\vert
B\cap \mathcal{R}\right\vert =1$ if, and only if, $\zeta $ fixes exactly one
point in $PG_{1}(q) \setminus \{0,\infty\}$, and we get (2).
\end{enumerate}
\end{enumerate}
\end{proof}

\bigskip 

\begin{corollary}
\label{DiheD}$H_{0,B}=K\left\langle \phi \right\rangle $ if, and only if, 
\begin{equation}\label{mantua}
B=\left\{ (m_{1}^{\sigma +2}x_{0},m_{1}^{\sigma }y_{0},m_{2}^{-\sigma
}y_{0},m_{2}^{-\sigma -2}x_{0}):m_{1},m_{2}\in GF(q)\right\}.
\end{equation}
\end{corollary}

\begin{proof}
Clearly $K\left\langle \phi \right\rangle $ stabilizes $B$ when this one is as in (\ref{mantua}). Actually $%
H_{0,B}=K\left\langle \phi \right\rangle $, since $K\left\langle \phi
\right\rangle $ is maximal in $H_{0}$.

Now assume that $H_{0,B}=K\left\langle \phi \right\rangle $. Then $B$ is as in Lemma \ref{block}. Also $(x_{0},y_{0},z_{0},t_{0})$ fulfills one of the constraints of Proposition \ref{IntOrb}. Note that $%
(0,0,y_{0},x_{0})\in B$, as it is the image of $%
(x_{0},y_{0},0,0)$ under $\phi $. Therefore,
$$B=\left\langle
(x_{0},y_{0},0,0)^{K}\right\rangle _{GF(2)}\oplus \left\langle
(0,0,y_{0},x_{0})^{K}\right\rangle _{GF(2)}$$
and hence $B$ is as in (\ref{mantua}).
\end{proof}

\begin{theorem}
\label{LemSuz}Let $\mathcal{D}$ be a $2$-$(q^{4},q^{2},\lambda )$ design,
with $\lambda \mid q^{2}$, admitting $G=TG_{0}$, where $Sz(q)\trianglelefteq
G_{0}$, as a flag-transitive automorphism group. Let $B$ a block of $\mathcal{D}$ incident with $0$ such that either $H_{0,B} \leq U:K$, then one of the following holds:
\begin{enumerate}
\item $H_{0,B}=Z(U):K$ and $\mathcal{D}$ is a $2$-$(q^{4},q^{2},q )$ design isomorphic to that constructed in Example (\ref{Ex2}) (Family 1);
\item $H_{0,B}=K\left\langle \phi \right\rangle$ and $\mathcal{D}$ is a $2$-$(q^{4},q^{2},q^{2}/2 )$ design isomorphic to that constructed in Example (\ref{Ex2}) (Family 2);
\item $H_{0,B}=K$ and $\mathcal{D}$ is a $2$-$(q^{4},q^{2},q^{2} )$ design isomorphic to that constructed in Example (\ref{Ex2}) (Families 3 or 4);
\end{enumerate}
\end{theorem}

\begin{proof} Let $B$ a block of $\mathcal{D}$ incident with $0$ such that either $H_{0,B} \leq U:K$ or $H_{0,B}=K\left\langle \phi \right\rangle$. Actually $H_{0,B} \neq U:K$ by Lemma \ref{solo2}(1). Therefore $H_{0,B}$ is isomorphic to one of the group $Z(U):K$, $K$ or $K\left\langle \phi \right\rangle$. In the former case $\mathcal{D}$ is a $2$-$(q^{4},q^{2},q )$ design isomorphic to that constructed in Example (\ref{Ex2}) (Family 1) by Lemma \ref{solo2}(2).

Assume that $H_{0,B}=K\left\langle \phi \right\rangle$. Then $B$ is as in (\ref{mantua}) of Corollary \ref{DiheD}, hence $\mathcal{D}$ is a $2$-$(q^{4},q^{2},q^{2}/2 )$ design isomorphic to that constructed in Example (\ref{Ex2}) (Family 2).

Finally, assume that $H_{0,B}=K$. Then $B$ is as in Lemma \ref{block} with $(x_{0},y_{0},z_{0},t_{0})$ fulfilling one of the constraints of Proposition \ref{IntOrb} and such that $(z_{0},t_{0}) \neq (y_{0},x_{0})$ by Corollary \ref{DiheD}. Hence, $\mathcal{D}$ is a $2$-$(q^{4},q^{2},q^{2} )$ design isomorphic to that constructed in Example (\ref{Ex2}) (Families 3 or 4).       
\end{proof}

\bigskip

\begin{proof}[Proof of Theorem \protect\ref{ClassS}]
Assume that $G_{0}$ is a nearly simple group. Hence, the socle $S$ of $G_{0}/\left( G_{0}\cap Z\right)$, where $Z$ is the center of $GL(V)$, is a non-abelian simple group. Then $S$ is not sporadic by Lemma \ref{NotSpor}. If $S$ is alternating or a Lie type simple group in characteristic $p^{\prime }$, then $\mathcal{D}$ is isomorphic to the $2$-design constructed in Example \ref{Ex0} for $q=2$ by Lemmas \ref{NotFDPM} and \ref{Cross} respectively. Finally, if $S$ is a Lie type simple group in characteristic $p$, then $\mathcal{D}$ is isomorphic to
one of the $2$-designs constructed in Examples \ref{Ex1} or \ref{Ex2} by Proposition \ref{NChar} and Theorem \ref{LemSuz}. This completes the proof.
\end{proof}

\bigskip

\section{The case where $G_{0}\leq (D_{8}\circ Q_{8}).S_{5}$ and $(q,n)=(3,4)$}\label{S3}

This small section is devoted to the analysis of the $2$-$(3^{4},3^{2},3)$
designs admitting a flag-transitive automorphism group $G=TG_{0}$, where $%
G_{0}\leq (R\circ SL_{2}(5)):Z_{2}$ and $R\cong D_{8}\circ Q_{8}$. In order
to do so, we need the information contained in the following lemma.

\begin{lemma}
\label{uniCc}The following hold:

\begin{enumerate}
\item $GL_{4}(3)$ contains a unique conjugacy class of subgroups isomorphic
to $SL_{2}(5)$, $N_{GL_{4}(3)}(SL_{2}(5))\cong (Z_{8}\circ SL_{2}(5)):Z_{2}$
and $SL_{2}(5)<Sp_{4}(3)$.

\item $SL_{2}(5)$ has precisely two orbits of length $30$ on the set of $2$%
-dimensional subspaces of $V_{4}(3)$ and these are permuted transitively by $%
N_{GL_{4}(3)}(SL_{2}(5))$.
\end{enumerate}
\end{lemma}

\begin{proof}
Let $X$ be a subgroup of $GL_{4}(3)$ isomorphic to $SL_{2}(5)$. It follows
from \cite[Tables 8.8 and 8.9]{BHRD} that either $X$ lies in $\mathcal{C}%
_{3}$-member of $GL_{4}(3)$, and hence $X$ preserves a Desarguesian $2$%
-spread of $V$ by \cite[Theorem 1]{Dye}, or $X<Sp_{4}(3)$. Assume that the
latter occurs. The set $V^{\ast }$ is partitioned into two $X$-orbits each of
length $40$, since each Sylow $2$-subgroup of $SL_{2}(5)$ is isomorphic to $%
Q_{8}$ and its unique involution is $-1$, and since each $5$-subgroup of $X$
acts irreducibly on $V$. Thus, if $K$ is any cyclic subgroup of $X$ of order 
$3$, $K$ fixes precisely four points on each $X$-orbit of length $40$ and
hence $\mathrm{Fix(}K\mathrm{)}$ is a $2$-dimensional subspace of $V$.
Therefore $\mathcal{S}=\mathrm{Fix(}K\mathrm{)}^{X}$ is a $2$-spread of $V$,
since $N_{X}(K)\cong \left\langle -1\right\rangle .S_{3}$ and since $%
X_{x}\cong Z_{3}$ for each non-zero vector of $V$. Also $\mathcal{S}$ is
Desarguesian by \cite[Proposition 5.3 and Corollary 5.5]{Fou2}, since $X$
acts transitively on $\mathcal{S}$. The Desarguesian $2$-spreads of $V$ lie
in a unique $GL_{4}(3)$-orbit by \cite[Theorem I.11.1]{Lu}. Moreover $%
GL_{4}(3)_{\mathcal{S}}\cong \Gamma L_{2}(9)$ and it can be deduced from 
\cite{At}, that $%
GL_{4}(3)_{\mathcal{S}}$ contains a unique conjugate class of subgroups isomorphic to 
$X$. Thus $GL_{4}(3)$ contains a unique conjugacy class of subgroups
isomorphic to $X$.

Let $N=N_{\Gamma L_{2}(9)}(X)$, hence $N\cong (Z_{8}\circ SL_{2}(5)):Z_{2}$
by \cite[Lemma 11.2]{Fou}. Then $N=N_{GL_{4}(3)}(X)$ by \cite%
[Tables 8.8, 8.12 and 8.13]{BHRD}. This proves (1).

Let $X$ be the copy of $SL_{2}(5)$ provided in Example \ref{Ex3}, and denote
by $Y$ the copy $Sp_{4}(3)$ containing $X$. Then $Y$ partitions the set of $%
2 $-dimensional subspaces of $V$ into two orbits of length $40$ and $90$,
say $\mathcal{I}$ and $\mathcal{N}$. Then $\mathcal{I}$ and $\mathcal{N}$
consist of the totally isotropic $2$-dimensional subspaces and then
non-isotropic ones, respectively, with regards to the symplectic form
preserved by $Y$. Moreover, $\mathcal{I}$ is partitioned into two $X$-orbits
of length $30$ and $10$, and the one of length $10$ is a Desarguesian $2$-spread of $%
V$ (see Example \ref{Ex3} and proof therein) that we still denote
by $\mathcal{S}$ . By (1) we know that $N\cong (C\circ X):Z_{2}$, where $%
N=N_{GL_{4}(3)}(X_{0})$ and $C\cong Z_{8}$. If $N$ preserves $%
\mathcal{I \setminus S}$, then there is a subgroup $C_{1}$ of $C$, such that $\left[
C:C_{1}\right] \leq 2$, preserving at least one $2$-subspace contained in $%
\mathcal{I \setminus S}$. Then $C_{1}$ preserves each subspace in $\mathcal{I \setminus S}$,
since $C_{1}$ centralizes $X$ and $X$ acts transitively on $\mathcal{I \setminus S}$.
Let $U_{1},U_{2}\in \mathcal{I \setminus S}$ such that $\dim (U_{1}\cap U_{2})=1$.
Such subspaces do exist since $\left\vert \mathcal{I \setminus S}\right\vert =30$.
Then $C_{1}$ preserves $U_{1}\cap U_{2}$ and hence the involution $-1$,
which lies in $C_{1}$, fixes pointwise $U_{1}\cap U_{2}$, a contradiction.
Then $\left( \mathcal{I \setminus S}\right) ^{N}$ is a union of $s$ orbits under $X$
each length $30$, where $s\geq 2$ and $s$ divides $\left[ N:X\right] =8$.
Furthermore $s-1$ of these orbits are contained in $\mathcal{N}$. Let $%
\mathcal{X}$ be the union of these $s-1$ orbits. Then $\mathcal{X}\subset 
\mathcal{N}$ and $\left\vert \mathcal{X}\right\vert =30(s-1)$.

The group $X$ preserves a Hall $2$-spread $\mathcal{H}$ of $V$ and
partitions $\mathcal{H}$ into two orbits each of length $5$ by \cite[Proposition 5.3 and Corollary 5.5]{Fou2}. Clearly, $\mathcal{H\subset N}$.
Moreover, $GL_{4}(3)_{\mathcal{H}}\cong (D_{8}\circ Q_{8}).S_{5}$ by \cite%
[Theorem I.8.3]{Lu}. It is easy to see that $N_{GL_{4}(3)_{\mathcal{H}%
}}(X)\cong Z_{2}.S_{5}^{-}$ again by \cite[Proposition 5.3 and
Corollary 5.5]{Fou2}. Then $X$ preserves a further Hall $2$-spread $\mathcal{H}%
^{\prime }$ of $V$ and partitions $\mathcal{H}^{\prime }$ into two orbits of
length $5$, since $N_{GL_{4}(3)}\left( Z_{2}.S_{5}^{-}\right) \cong \left(
Z_{2}.S_{5}^{-}\right) :Z_{2}$. Also, $\mathcal{H}^{\prime }\mathcal{\subset
N}$ and hence $\mathcal{X}\subset \mathcal{N}-(\mathcal{H\cup H}^{\prime })$%
. Then $\left\vert \mathcal{N}\right\vert =90$, $\left\vert \mathcal{X}%
\right\vert =30(s-1)$, with $s\geq 2$ and $s\mid 8$, and $\left\vert 
\mathcal{H}^{\prime }\right\vert =\left\vert \mathcal{H}\right\vert =10$,
imply $s=2$ and $\left\vert \mathcal{N}-(\mathcal{H}\cup \mathcal{H}^{\prime
}\cup \mathcal{X})\right\vert =40$.

Suppose that there is a further $X$-orbit $\mathcal{O}$ of length $30$
contained in $\mathcal{N}-(\mathcal{H}\cup \mathcal{H}^{\prime }\cup 
\mathcal{X})$ and such that $\mathcal{O}\notin \left( \mathcal{I \setminus S}\right)
^{N}$. Thus $\mathcal{O}$ is also a $N$-orbit, since $\left\vert \mathcal{N}%
-(\mathcal{H}\cup \mathcal{H}^{\prime }\cup \mathcal{X})\right\vert =40$.
Then we may apply the above argument with $\mathcal{O}$ in the place of $%
\mathcal{I \setminus S}$ thus obtaining that a central cyclic subgroup of $N$ of order 
$4$, containing $-1$, preserves each $2$-dimensional subspace of $V$ contained in $\mathcal{O}$
and derive a contradiction from this. Thus the unique $X$-orbits of length $%
30$ on the set of $2$-dimensional subspaces of $V$ are those lying in $%
\left( \mathcal{I \setminus S}\right) ^{N}$, and we have seen that these are $2$. This
proves (2).
\end{proof}

\begin{theorem}
\label{normIrre}Let $\mathcal{D}$ be a $2$-$(3^{4},3^{2},3)$ design
admitting a flag-transitive automorphism group $G=TG_{0}$, where $G_{0}\leq
(R\circ SL_{2}(5)):Z_{2}$ and $R\cong D_{8}\circ Q_{8}$, then $\mathcal{D}$
is isomorphic to the $2$-design constructed in Example \ref{Ex3}.
\end{theorem}

\begin{proof}
Clearly, $SL_{2}(5)\leq G_{0}$ as $r=30$. Since the action of any Sylow $5$%
-subgroup of $G_{0}$ on $R/Z(R)$ is irreducible, as $5$ is a primitive prime
divisor of $2^{4}-1$, it follows that either $R\leq G_{0}$ or $R\cap
G_{0} = \left\langle -1\right\rangle $. Therefore, either $R\circ
SL_{2}(5)\trianglelefteq G_{0}\leq (R\circ SL_{2}(5)):Z_{2}$, or $%
SL_{2}(5)\trianglelefteq G_{0}\leq \left\langle -1\right\rangle .S_{5}^{-}$,
respectively, by \cite{AtMod}, since $SL_{2}(5)\leq G_{0}$. Also $t\geq
3-f\geq 1$ by Lemma \ref{cici}(1), and hence $t=2$ by Lemma \ref{inv}, since 
$-1\in G_{0}$. Hence, the blocks incident with $0$ are $2$-dimensional
subspaces of $V$.

Suppose that $R\leq G_{0}$. The group $R$ splits $V^{\ast }$ into $5$ orbits
each of length $16$ and $G_{0}$ induces $A_{5}\trianglelefteq G_{0}/R\leq
S_{5}$ in its natural $2$-transitive permutation representation of degree $5$
on the set of these $R$-orbits. Each of the $10$ non-central involutions
fixes exactly $8$ points of $V^{\ast }$, which lie in the same $R$-orbit.
Hence each $R$-orbit consists of the $8$ points fixed by an involution, say $%
\sigma $, and its complementary set fixed by $\epsilon \sigma $, where $%
\epsilon $ denotes the involution $-1$. Since $R\trianglelefteq G_{0}$ and since 
$G_{0}$ acts transitively on the set of $30$ blocks incident with $0$, it
follows that $\left\vert R_{B}\right\vert =16$ where $B$ is any block of $%
\mathcal{D}$ incident with $0$. Then $B^{R}=\left\{ B,B^{\prime }\right\} $
and hence $\left( B\cup B^{\prime }\right) \setminus \left\{ 0\right\} $ is a point-$%
R $-orbit, since it is an $R$-invariant set of cardinality $16$. Let $\gamma$ be non-central involution in $R$ such that $\mathrm{Fix}(\gamma) \subset \left( B\cup B^{\prime }\right) \setminus \left\{ 0\right\}$. Actually, either $\mathrm{Fix}(\gamma)=B$ or $\mathrm{Fix}(\gamma)=B^{\prime}$ since $B,B^{\prime}$ and $\mathrm{Fix}(\gamma)$ are $2$-dimensional subspaces of $V$. This forces $r=10$ since $V^{\ast}$ is partitioned into five $R$-orbits each of length $2$, whereas $r=30$. So, this case is ruled out.

Suppose that $R\cap G_{0}= \left\langle -1\right\rangle $. Then $%
X_{0}\trianglelefteq G_{0}$ with $X_{0}\cong SL_{2}(5)$, since $%
SL_{2}(5)\leq G_{0}$. The group $GL_{4}(3)$ contains a unique conjugacy
class of subgroups isomorphic to $SL_{2}(5)$ by Lemma \ref{uniCc}(1).
Therefore, we may assume that $X_{0}$ is the copy of $SL_{2}(5)$ provided in
Example \ref{Ex3}. Hence $X_{0}$ has exactly two orbits of length $30$ on
the set of $2$-dimensional subspaces by Lemma \ref{uniCc}(2). Let $%
B_{1},B_{2}$ any two representatives of these orbits respectively. At least
one of the incidence structures $(V,B_{i}^{X_{0}})$, $i=1,2$, is the $2$-$%
(3^{4},3^{2},3)$ design defined in Example \ref{Ex3}, say $(V,B_{1}^{X_{0}})$%
. Actually, $(V,B_{2}^{X_{0}})$ is a $2$-$(3^{4},3^{2},3)$ design isomorphic
to $(V_{4}(3),B_{1}^{X_{0}})$, since $N_{GL_{4}(3)}(X_{0})$ switches $%
B_{1}^{X_{0}}$ and $B_{2}^{X_{0}}$ by Lemma \ref{uniCc}(2). Thus $\mathcal{D}
$ is any of these two isomorphic $2$-designs, and we obtain (2).
\end{proof}

\bigskip

\section{The case where $G_{0}$ is a classical group}\label{S4}

In this section we assume that $X_{0}\trianglelefteq G_{0}$, where $%
X_{0}$ denotes one of the groups $SL_{n}(q)$, $Sp_{n}(q)$, $SU_{n}(q^{1/2})$
with $n$ odd, or $\Omega _{n}^{-}(q)$ and we complete the analysis of the
affine type case. A preliminary reduction is obtained by comparing the minimal
primitive permutation representation of suitable subgroups of $X_{0,x}$, where 
$x$ is a non-zero vector of $V_{n}(q)$, with $\lambda \leq q^{n/2}$. This allows us to
reduce the possibilities for $X_{0}$ to be one of the groups $SL_{2}(q)$, $Sp_{4}(q)$, $SU_{3}(q^{1/2})$ or $\Omega_{4}^{-}(q)$. Then we carry out a separate
investigation of these remaining cases. More precisely, we prove the following theorem.

\begin{theorem}
\label{classical}Let $\mathcal{D}$ be a $2$-$(q^{n},q^{n/2},\lambda )$, with 
$\lambda \mid q^{n/2}$, admitting a flag-transitive automorphism group $%
G=TG_{0}$. If $X_{0}\trianglelefteq G_{0}$ and $X_{0}$ is one of the groups
defined above, then one of the following holds:
\begin{enumerate}
\item $X_{0} \cong SL_{2}(q)$ and $\mathcal{D}$ is isomorphic to the $2$-design constructed in Example \ref{Ex4}.
\item $X_{0} \cong SU_{3}(q^{1/2})$, where $q$ is a square, and $\mathcal{D}$ is isomorphic to the $2$-design constructed in Example \ref{Ex5}.
\item $X_{0}\cong Sp_{4}(q)$ and $\mathcal{D}$ is isomorphic to the $2$%
-design constructed in Example \ref{Ex0}.
\end{enumerate}
\end{theorem}

\bigskip

\begin{lemma}
\label{SL}If $X_{0}\cong SL_{n}(q)$, then $n=2$.
\end{lemma}

\begin{proof}
Let $x$ be any non-zero vector of $V_{n}(q)$. Then $C:Y\leq X_{0,x}$, where $C$ is an
elementary abelian $p$-group and $Y$ is a isomorphic to $SL_{n-1}(q)$, by 
\cite[Lemma 4.1.13]{KL}. Let $B$ be any block of $\mathcal{D}$ incident with 
$0,x$. Then $P(Y)\leq \left\vert B^{Y}\right\vert \leq q^{n/2}$, where $P(Y)$
denotes the minimal primitive permutation representation of $Y$, since $\lambda \leq q^{n/2}$. If $n\geq 4$
and $(n,q)\neq (4,2)$, then $P(Y)=\frac{q^{n-1}-1}{q-1}$ by \cite[Proposition 5.2.1 and Table 5.2.A]{KL}, which is greater than $q^{n/2}$. So this
case is ruled out. Also, $(n,q)=(4,2)$ is ruled out by Lemma \ref{NotFDPM}, since 
$X_{0}\cong SL_{4}(2)\cong A_{8}$. Thus $Y$ fixes each block of $\mathcal{D}$
incident with $0,x$, and hence it lies in the kernel of the action of $C:Y$ on
the set of blocks of $\mathcal{D}$ incident with $0,x$. This forces $C:Y$ to
fix each block of $\mathcal{D}$ incident with $0,x$. So $\frac{q^{n}-1}{q-1}%
\mid r$, since $X_{0,x}\leq G_{0,\left\langle x\right\rangle }$, which is
impossible as $r=p^{f}(q^{n/2}+1)$ and $n \geq 4$.

Assume that $n=3$. Then $q$ is a square and $[X_{0}:X_{0,B}]\mid
q^{3/2}(q^{3/2}+1)$, which is impossible by \cite[Tables 8.3--8.4]{BHRD}.
Thus $n=2$ since $n>1$ by our assumption.
\end{proof}

\begin{proposition}
\label{SLSp}If $SL_{n}(q)\trianglelefteq X_{0}$, then $n=2$ and $\mathcal{D}$
is isomorphic to the $2$-design of Example \ref{Ex4}.
\end{proposition}

\begin{proof}
We prove the assertion in a series of steps:

\begin{enumerate}
\item[(i).] \textbf{$\mathcal{D}_{X}\cong (V,B^{X})$, where $V=V_{2}(q)$, $X=T:X_{0}$ and $X_{0}\cong SL_{2}(q)$, is a point-$2$-transitive, flag-transitive $2$-$(q^{2},q,\lambda _{0})$ design, with $\lambda _{0}\mid \lambda $ and $\lambda _{0}\leq \lambda \leq q$.}
\end{enumerate}

Let $X=T:X_{0}$, where $X_{0}\cong SL_{2}(q)$, and let $B$ be any block
incident with $0$. Then $(0,B)^{G}$ is split into $X$-orbits of equal
length, since $X\trianglelefteq G$. Then $\mathcal{D}_{X}\cong (V,B^{X})$ is
a flag-transitive tactical configuration with parameters $(v,k,r_{0},b_{0})$
with $r_{0}\mid r$ and $b_{0}\mid b$. Actually, $\mathcal{D}_{X}$ is a $2$-$%
(q^{2},q,\lambda _{0})$ design, since $X$ acts $2$-transitively on $V$.
Also, as $r_{0}=(q+1)\lambda _{0}$ and $r_{0}$ divides $r$, it follows that $%
\lambda _{0}\mid \lambda $ and hence $\lambda _{0}\leq \lambda \leq q$.

\begin{enumerate}
\item[(ii).] \textbf{$T$ does not act block-semiregularly on $\mathcal{D}%
_{X}$.}
\end{enumerate}

Suppose the contrary. We may use Lemma \ref{cici}(1), with $\mathcal{D}%
_{X}$ and $X$ in the role of $\mathcal{D}$ and $G$ respectively, in order
to get $\lambda _{0}=q$. Thus $\lambda _{0}=\lambda$ and hence $%
\mathcal{D}_{X}=\mathcal{D}$. Since $r=q(q+1)$, it follows that $\left\vert
X_{0,B}\right\vert =q-1$. Then $\left\vert X_{B}\right\vert =q(q-1)$ and
hence $X_{B}$ is a Frobenius group with elementary abelian kernel $N$ and a cyclic complement $X_{0,B}$ by
Lemma \ref{cici}(2). Since $X_{0}$ contains a
unique conjugate class of cyclic subgroups of order $q-1$, we may assume
that $X_{0,B}=\left\langle \gamma \right\rangle $, where $\gamma $ is
represented by $\mathrm{Diag(}\omega ,\omega ^{-1}\mathrm{)}$ and $\omega $
is a primitive element of $GF(q)^{\ast }$.

If $x=(x_{1},x_{2})$ is any non-zero vector of $V$ lying in $B$, then $B=(x_{1},x_{2})^{\left\langle \gamma \right\rangle}\cup
\left\{ (0,0)\right\}$. If $x_{1}=0$, then $B=\left\langle (0,1)\right\rangle $ and hence $T_{B}\neq
1$, whereas $T$ acts block-semiregularly on $\mathcal{D}$ by our assumption.
Thus $x_{1}\neq 0$. Also $x_{2}\neq 0$ by a similar argument.

Let $\varphi \in N$, $\varphi \neq 1$, then there are $\alpha \in X_{0}$ and $\tau \in T$ such that $\varphi=\alpha\tau$. Therefore $N=\left\{ 1,\alpha
^{\gamma ^{j}}\tau ^{\gamma ^{j}}:i=0,...,q-1\right\} $, since $N$ and $X_{0,B}$ are the kernel and a complement of the Frobenius group $X_{B}$ respectively. Thus, $\left( \alpha
\tau \right) \left( \alpha ^{\gamma }\tau ^{\gamma }\right) =\left( \alpha
^{\gamma }\tau ^{\gamma }\right) \left( \alpha \tau \right) $, since $N$ is abelian. It follows that $%
\alpha \alpha ^{\gamma }T=\alpha ^{\gamma }\alpha T$, since $T$ is a normal subgroup of $X$, hence $\left(
\alpha \alpha ^{\gamma }\right) \left( \alpha ^{\gamma }\alpha \right)
^{-1}\in T\cap X_{0}$. Thus $\alpha \alpha ^{\gamma }=\alpha ^{\gamma
}\alpha $, since $T\cap X_{0}=1$. Note that $o(\alpha)=p$, since $o(\varphi)=p$ and since $T$ is a normal subgroup of $X$ such that $T\cap X_{0}=1$. Therefore $\alpha ^{\gamma }$ lies in the same Sylow $p$%
-subgroup of $X_{0}$ containing $\alpha $. Let $S$ be such a Sylow $p$%
-subgroup of $X_{0}$. Then $\gamma \in N_{X_{0}}(S)$, since the intersection
of any two distinct Sylow $2$-subgroups of $X_{0}$ is trivial by \cite[Satz II.8.2.(c)]{Hup}. Therefore, $\gamma $ preserves $\left\langle
(x_{1},x_{2})\right\rangle $, since $\mathrm{Fix}(S)=\mathrm{Fix}(\alpha )=\left\langle
(x_{1},x_{2})\right\rangle $, and hence either $x_{1}=0$ or $x_{2}=0$, a
contradiction.

\begin{enumerate}
\item[(iii).] \textbf{The blocks of $\mathcal{D}_{X}$ are affine $%
GF(p)$-subspaces of $V$.}
\end{enumerate}

The assertion is true for $p$ odd by (ii) and by Lemma \ref{inv}, since $-1\in
X_{0}$. Hence assume that $p=2$. Since $r_{0}=\lambda_{0} (q+1)$ and $\lambda_{0}
=p^{f_{0}}\leq q$, it follows that $\left\vert X_{0,B}\right\vert $ is divisible
by $q-1$. Then $X_{B}/T_{B}$ is isomorphic to a subgroup of $X_{0}$ of
order $\frac{q}{\left\vert T_{B}\right\vert }\left\vert X_{0,B}\right\vert $
by Lemma \ref{cici}(2). Also $q/\left\vert T_{B}\right\vert <q$ by (ii).
Then $X_{B}/T_{B}\cong D_{2(q-1)}$ and $\left\vert T_{B}\right\vert =q/2$ 
by \cite[Tables 8.1--8.2]{BHRD}. However, this is impossible by Corollary \ref{p2}. Thus (iii) holds for $q$ even as well.

\begin{enumerate}
\item[(iv).] \textbf{$\mathcal{D}_{X}$ is isomorphic to the $2$-design
constructed in Example \ref{Ex4}.}
\end{enumerate}

If $B$ is contained in a $1$-dimensional $GF(q)$-subspace of $V$, then $B$
is a $1$-dimensional $GF(q)$-subspace of $V$, since $k=q$. Then $\mathcal{D}%
_{X}\cong AG_{2}(q)$, whereas $\lambda \geq 2$ by our assumption. Thus $B$
is not contained in any $1$-dimensional subspace of $V$ and hence $q \neq p$. Therefore $B$ is a $2$-dimensional $GF(q^{1/2})$-subspace
of $V$ when $q=p^{2}$ by (iii), hence $SL_{2}(q^{1/2})\trianglelefteq
X_{0,B}\leq SL_{2}(q^{1/2}).Z_{(2,q^{1/2}-1)}$ in this case.

Assume that $q\neq p,p^{2}$. Then either $Z_{q-1}\leq X_{0,B}$ or $q$ is
odd and $X_{0,B}\cong Q_{q-1}$, or $q$ is a square and $SL_{2}(q^{1/2})\trianglelefteq
X_{0,B}\leq SL_{2}(q^{1/2}).Z_{(2,q^{1/2}-1)}$ by \cite[Tables 8.1--8.2]{BHRD}, since $q-1$ divides the order of $X_{0,B}$.

Suppose that the former occurs. Since $SL_{2}(q)$ has one conjugacy class of
subgroups isomorphic to $Z_{q-1}$, we may assume that the cyclic group
contained in $X_{0,B}$ is that generated by the element represented by $%
\mathrm{Diag(}\omega ,\omega ^{-1}\mathrm{)}$, where $\omega $ is a
primitive element of $GF(q)^{\ast }$. Let $(v_{1},v_{2})$ be any non-zero
vector of $B$. Then 
\[
B=\left\{ (\omega ^{i}v_{1},\omega ^{-i}v_{2}):i=1,...,q-1\right\} \cup
\left\{ (0,0)\right\} \text{.} 
\]%
Then $(\omega v_{1},\omega ^{-1}v_{2})+(v_{1},v_{2})=(\omega
^{s}v_{1},\omega ^{-s}v_{2})$ for some $s\in \left\{ 1,...,q-1\right\} $, since $B$ is a $GF(p)$-subspace of $V$ by (iii).
Then $\omega +1=\omega ^{s}$ and $\omega ^{-1}+1=\omega ^{-s}$ and hence $%
\omega ^{2}+\omega ^{-2}+1=0$. So $\omega ^{3}=1$ and $q=4$, whereas $q\neq
p^{2}$ by our assumption.

Assume that $X_{0,B}\cong Q_{q-1}$, $q$ odd. Since $SL_{2}(q)$ has two
conjugacy classes of subgroups isomorphic to $Q_{q-1}$, we may assume that%
\[
X_{0,B}=\left\langle \left( 
\begin{array}{cc}
\omega ^{2} & 0 \\ 
0 & \omega ^{-2}%
\end{array}%
\right) ,\left( 
\begin{array}{cc}
0 & -\varepsilon \\ 
\varepsilon ^{-1} & 0%
\end{array}%
\right) \right\rangle \text{,} 
\]%
where $\varepsilon $ is either $1$ or $\omega $. Then 
\[
B=\left\{ (\omega ^{2i}v_{1},\omega ^{-2i}v_{2}),(-\omega ^{2i}\varepsilon
v_{2},\left( \omega ^{2i}\varepsilon \right)
^{-1}v_{1}):i=1,...,(q-1)/2\right\} \cup \left\{ (0,0)\right\} \text{,} 
\]%
for some vector $(v_{1},v_{2})$ with $v_{1},v_{2}\neq 0$, since $X_{0,B}$ acts regularly on $B^{\ast}$. If 
\[
(\omega ^{2}v_{1},\omega ^{-2}v_{2})+(v_{1},v_{2})=(\omega ^{2s}v_{1},\omega
^{-2s}v_{2}) 
\]%
for some $s=1,...,(q-1)/2$, then the above argument implies $\omega ^{6}=1$
and hence $q-1\mid 6$, but this contradicts $q\neq p$, $p$ odd. Then%
\[
(\omega ^{2}v_{1},\omega ^{-2}v_{2})+(v_{1},v_{2})=\left( -\omega
^{2i_{0}}\varepsilon v_{2},\left( \omega ^{2i_{0}}\varepsilon \right)
^{-1}v_{1}\right) 
\]%
for some $i_{0}=1,...,(q-1)/2$, hence $\omega ^{4}+\omega ^{2}+1=0$ as $%
v_{1},v_{2}\neq 0$. So $\omega ^{6}=1$ and we again reach a contradiction.

Finally, assume that $q$ is a square and that $SL_{2}(q^{1/2})\trianglelefteq X_{0,B}\leq
SL_{2}(q^{1/2}).Z_{(2,q^{1/2}-1)}$. Note that this case also encloses $q=p^{2}$ with $p$ a prime. Suppose that $q$ is odd and that $X_{0,B}\cong
SL_{2}(q^{1/2}).Z_{q^{1/2}-1}$. Then $r_{0}=\left[ X_{0}:X_{0,B}\right]
=q^{1/2}\frac{q+1}{q^{1/2}-1}$, hence $r=\lambda \frac{q+1}{q^{1/2}-1}$ since $r/r_{0} \mid q$. However this is impossible since $r=\left[ G_{0}:G_{0,B}\right]=\lambda(q+1)$ with $\lambda \mid q$. Thus $X_{0,B}\cong SL_{2}(q^{1/2})$.  

The group $X_{0,B}$ preserves two $(q^{1/2}+1)$-sets $\mathcal{R},\mathcal{R}^{\prime}$ each of which consisting of $2$-dimensional $GF(q^{1/2})$-subspaces of $V$. Also $\mathcal{R}$ and $\mathcal{R}^{\prime}$ induce a regulus and its opposite in $PG_{3}(q^{1/2})$ respectively, which are the two systems of generators of an hyperbolic quadric (e.g. see \cite[Section 15.3.III]{Hir} or \cite[Section II.13]{Lu}). Hence, $\mathcal{R}$ and $\mathcal{R}^{\prime}$ cover the same subset of points of $V$ and $X_{0,B}$ acts naturally on each member of $\mathcal{R}^{\prime}$. Since the number conjugacy of $SL_{2}(q)$-classes of subgroups isomorphic to $X_{0,B}$ is $(2,q^{1/2}-1)$  by \cite[Table 8.1]{BHRD}, and these are fused in $\Gamma L_{2}(q)$, we may assume that $X_{0,B}$ is the copy of $SL_{2}(q^{1/2})$ preserving  
$$
\mathcal{R}=\left\lbrace \left\langle (1,c) \right\rangle_{GF(q)}: c \in GF(q), \; c^{q^{1/2}+1}=1  \right\rbrace \text{ and } \mathcal{R}^{\prime}=W^Z,   
$$
where $$W=\left\lbrace (x,x^{q^{1/2}}c): x,c \in GF(q), \; c^{q^{1/2}+1}=1\right\rbrace$$ and $Z=Z(GL_{2}(q))$ (see \cite[Theorem II.13.5]{Lu}). Since any Sylow $p$-subgroup of $X_{0,B}$ acts semiregularly on $V \setminus \bigcup \mathcal{R}^{\prime}$, if $B\cap \left( \bigcup \mathcal{R}^{\prime} \right) = \left\{ (0,0)\right\} $ then $q^{1/2}\mid \left\vert B^{\ast
}\right\vert $, whereas $B^{\ast }$ is coprime to $p$ as $k=q$. Therefore $B\cap \left( \bigcup \mathcal{R}^{\prime} \right) \neq \left\{ (0,0)\right\}$ and hence there is $z_{0}$ in $Z$ such that $B\cap W^{z_{0}} \neq \left\{ (0,0)\right\}$. Actually, $B = W^{z_{0}}$ since $k=q$ and since $%
X_{0,B}$ acts transitively on $(W^{z_{0}})^{\ast }$.

\begin{enumerate}
\item[(v).] \textbf{$\mathcal{D}$ is isomorphic to the $2$-design
constructed in Example \ref{Ex4}.}
\end{enumerate}

Since $X_{0,B}\cong SL_{2}(q^{1/2})$ by (iii), and the number of conjugacy $SL_{2}(q)$-classes of subgroups isomorphic to $X_{0,B}$ is $(2,q^{1/2}-1)$, it follows that $[G_{0}:X_{0}]_{p}=[G_{0,B}:X_{0,B}]_{p}$ and hence that $r_{p}=[G_{0}:G_{0,B}]_{p}=[X_{0}:X_{0,B}]_{p}=(r_{0})_{p}$. Therefore $r=r_{0}$, since $r=\frac{\lambda}{\lambda_{0}}r_{0}$ with $\lambda/\lambda_{0} \mid q$, hence $\mathcal{D}=\mathcal{D}_{X}$ and the assertion follows from (iii).
\end{proof}

\begin{lemma}
\label{sympl}If $X_{0}\cong Sp_{n}(q)$, then one of the following holds:

\begin{enumerate}
\item $n=2$ and $\mathcal{D}$ is isomorphic to the $2$-design of Example \ref%
{Ex4}.

\item $n=4$ and $\mathcal{D}$ is isomorphic to the $2$-design of Example \ref%
{Ex0}.
\end{enumerate}
\end{lemma}

\begin{proof}
Let $x$ be any non-zero vector. Then $C:Y\leq X_{0,x}$, where $C$ is a $p$%
-group and $Y$ is a isomorphic to $Sp_{n-2}(q)$, by \cite[Lemma 4.1.12]{KL}.
Let $B$ be any block of $\mathcal{D}$ incident with $0,x$. Then $P(Y)\leq
\left\vert B^{Y}\right\vert \leq q^{n/2}$, where $P(Y)$ denotes the minimal
primitive permutation representation of $Y$. Assume that $n\geq 6$. Then $%
P(Y)=\frac{q^{n-2}-1}{q-1}$, or $2^{n/2-2}(2^{n/2-1}-1)$ according to
whether $q>2$ or $q=2$, respectively by \cite[Proposition 5.2.1 and
Table 5.2.A]{KL}. Then $q=2$ then $n=6$, since $\left\vert B^{Y}\right\vert \leq
2^{n/2}$ . However, this is impossible, since $\Phi _{6}^{\ast }(2)=1$
contradicts our assumption (see remark before Lemma \ref{alternative}). Thus 
$Y$ fixes each block of $\mathcal{D}$ incident with $0,x$, and hence it lies
in the kernel of the action of $C:Y$ on the set of blocks of $\mathcal{D}$
incident with $0,x$. Therefore $C:Y$ fixes each block of $\mathcal{D}$
incident with $0,x$. So $\frac{q^{n}-1}{q-1}\mid r$, which is impossible
since $r=p^{f}(q^{n/2}+1)$ and $n>2$. Thus $n=2,4$. Case $n=2$ leads to $2$%
-design of Example \ref{Ex4}, since $Sp_{2}(q)\cong SL_{2}(q)$, which is the
assertion (1).

Assume that $n=4$. Suppose that $T$ acts block-semiregularly on $\mathcal{D}$%
. Then $q>2$ by Lemma \ref{NotFDPM}, since $Sp_{4}(2)^{\prime }\cong A_{6}$.
Also $G_{B}$ is isomorphic to a subgroup $J$ of $G_{0}$ of index $q^{2}+1$
by Lemma \ref{cici}(2). Then $\left[ X_{0}:J\cap X_{0}\right] \mid q^{2}+1$, hence $X_{0}\trianglelefteq J$ by \cite[Proposition 5.2.1 and Table
5.2.A]{KL}. Therefore $\left[ G_{0}:J\right] \mid (q-1)\cdot (2,q-1)\cdot \log
_{p}q$, whereas $\left[ G_{0}:J\right] =q^{2}+1$. Thus $T$ does not act
block semiregularly on $\mathcal{D}$, hence $B$ is a $GF(p)$-subspace of $V$
by Lemma \ref{inv}. Moreover, either $X_{0,B}\leq Sp_{2}(q)\wr Z_{2}$, or $X_{0,B}\cong SO^{+}_{4}(q)$ with $q$ even, by \cite[Tables 8.12--8.14]{BHRD}%
, since $\left[ X_{0}:X_{0,B}\right] \mid p^{f}(q^{2}+1)$.

Assume that $X_{0,B}\cong SO^{+}_{4}(q)$ with $q$ even. The group $X_{0,B}$ preserves two $(q+1)$-sets $\mathcal{R},\mathcal{R}^{\prime}$ each of which consists of $2$-dimensional $GF(q)$-subspaces of $V$. Also $\mathcal{R}$ and $\mathcal{R}^{\prime}$ induce a regulus and its opposite in $PG_{3}(q)$ respectively, which are the two systems of generators of an hyperbolic quadric. Hence, $\mathcal{R}$ and $\mathcal{R}^{\prime}$ cover the same subset of points of $V$ and $X_{0,B}$ induces a copy of $PSL_{2}(q)$ in its $2$-transitive action of degree $q+1$ on $\mathcal{R}$ and another copy with the same action on $\mathcal{R}^{\prime}$. Since any Sylow $p$-subgroup of $X_{0,B}$ does not fix points in $V \setminus \bigcup \mathcal{R}$, but it fixes at least a point $B^{\ast}$, and since $k=q^{2}$, it follows that $B \subset \bigcup \mathcal{R}$. Therefore there is $W$ in $\mathcal{R}$ such that $B\cap W \neq \left\{ (0,0,0,0)\right\}$. Actually $B = W$, since $k=q^{2}$ and since $%
X_{0,B,W}$ induces a transitive group on $W^{\ast }$. So $%
X_{0,B}$ preserves $W$, whereas $X_{0,B}$ induces a transitive group on the $q+1$ members of $\mathcal{R}$. Hence, this case is ruled out.

Assume that $K\trianglelefteq X_{0,B}$, where $K\cong Sp_{2}(q)\times Sp_{2}(q)$,
and $\lambda =q^{2}$. Hence $K$ preserves the sum decomposition $%
V=V_{1}\oplus V_{1}^{\perp}$ where $V_{1}$ is a non-degenerate with respect to
the symplectic form preserved by $X_{0}$. It follows that $K\leq GL_{m}(p)$,
since $B$ is a $GF(p)$-subspace of $V$ and since $q^{2}=p^{m}$. Assume that $%
p^{m}-1$ has primitive prime divisor $w\,$. Then $w\mid \left\vert
K(B)\right\vert $, since $(q+1)^{2}\mid \left\vert K\right\vert $ and since $%
K\leq GL_{m}(p)$. Then either $Sp_{2}(q)\times 1\leq K(B)$ or $1\times
Sp_{2}(q)\leq K(B)$, since $K(B)\trianglelefteq K$. Therefore either $%
B=V_{1} $ or $B=V_{1}^{\perp}$ respectively, and hence $\mathcal{D}$ is isomorphic
to the $2$-design of Example \ref{Ex0}, which is the assertion (2).

Assume that $p^{m}-1$ does not have primitive prime divisors. Then $m=2$ and $q=p$ is a Mersenne prime by \cite[Theorem
II.6.2]{Lu}, since $q^{2}=p^{m}$. Then $K^{B}\leq GL_{2}(p)$ and hence either $Sp_{2}(p)\times 1\leq
K(B)$ or $1\times Sp_{2}(p)\leq K(B)$. Thus either $B=V_{1}$ or $B=V_{1}^{\perp}$
respectively, and we again obtain the assertion (2).
\end{proof}

\begin{lemma}
If $X_{0}\cong SU_{n}(q^{1/2})$, $q$ square, then $n=3$.
\end{lemma}

\begin{proof}
Assume that $X_{0}\cong SU_{n}(q^{1/2})$, $q$ square. If $x$ is a non-zero
isotropic vector of $V$, then $C:Y\leq X_{0,x}$, where $C$ is a $p$-group and $Y$
is a isomorphic to $SU_{n-2}(q^{1/2})$, by \cite[Lemma 4.1.12]{KL}. Let $B$
be any block of $\mathcal{D}$ incident with $0,x$. Then $P(Y)\leq \left\vert
B^{Y}\right\vert \leq q^{n/2}$, where $P(Y)$ denotes the minimal primitive
permutation representation of $Y$. Assume that $n>7$. Then%
\[
q^{n/2}\geq \lambda \geq \left\vert B^{Y}\right\vert >\frac{\left(
q^{(n-2)/2}+1\right) \left( q^{(n-2)/2-1}-1\right) }{q-1}
\]%
by \cite[Proposition 52.1 and Table 5.2.A]{KL}, since $n$ is odd, but the previous inequality has no solutions. Thus $Y$ fixes each block of $\mathcal{D}$ incident
with $0,x$, and hence it lies in the kernel of the action of $C:Y$ on the
set of blocks of $\mathcal{D}$ incident with $0,x$. Therefore, $C:Y$ fixes
each block of $\mathcal{D}$ incident with $0,x$. So $C:Y\leq X_{0,B}$ and
hence $\frac{\left( q^{n/2}+1\right) \left( q^{n/2-1}-1\right) }{q-1}\mid r$%
, which is impossible since $r=p^{f}(q^{n/2}+1)$ and $n>7$.

Assume that $n=5,7$. Then $X_{0,B}\leq GU_{n-1}(q^{1/2})<X_{0}$ by \cite%
[Tables 8.20--8.21 and 8.37--8.38]{BHRD}, since $\left[ X_{0}:X_{0,B}\right]
\mid p^{f}(q^{n/2}+1)$. Moreover, $p^{f}\geq q^{(n-1)/2}$ and $%
SU_{n-1}(q^{1/2})\trianglelefteq X_{0,B}$. Since the non-zero
isotropic vectors of $V$ lie in a unique $G_{0}$-orbit of length $%
(q^{n/2}+1)(q^{n/2-1}-1)$, it follows that $\left\vert B\cap
x^{G_{0}}\right\vert =q^{n/2-1}-1$ by Lemma \ref{PP}. Hence the copy of $%
SU_{n-1}(q^{1/2})$ inside $X_{0,B}$ fixes $B\cap x^{G_{0}}$ pointwise, since
the minimal primitive permutation representation of $SU_{n-1}(q^{1/2})$ is
greater than $q^{n/2-1}-1$ again by \cite[Proposition 52.1 and Table
5.2.A]{KL}. So, $SU_{n-1}(q^{1/2})\leq X_{0,x}$ with $x$ a non-zero isotropic
vector of $V$, but this contradicts \cite[Tables 8.20 and 8.37]{BHRD}. Thus $%
n=3$, which is the assertion.
\end{proof}

\begin{lemma}
\label{SU3}If $X_{0}\cong SU_{n}(q^{1/2})$, then $n=3$ and $\mathcal{D}$ is
a $2$-$(q^{3},q^{3/2},q)$ design isomorphic to that constructed in Example %
\ref{Ex4}.
\end{lemma}

\begin{proof}
Set $q=s^{2}$, then $X_{0}\cong SU_{3}(s)$, $V=V_{3}(s^{2})$, $r= \lambda (s^{3}+1)$ and $\lambda \mid s^{3}$. The set of blocks incident with $0$ is
split into $X_{0}$-orbits of the same length as $X_{0}\trianglelefteq G_{0}$, hence $\left[ X_{0}:X_{0,B}\right]
\theta =r$ where $B$ is any block incident with $0$ and $\theta $ is a positive divisor of $r$. Thus $\theta \mid
\left( s+1\right) \cdot 2\log _{p}(s)$, since $\theta \mid \frac{s^{2}-1}{%
(3,s+1)} \cdot \left\vert \mathrm{Out(}PSU_{3}(s)\mathrm{)}\right\vert $ and $%
\left\vert \mathrm{Out(}PSU_{3}(s)\mathrm{)}\right\vert =(3,s+1)\cdot 2\log
_{p}(s)$. Thus%
\begin{equation}
\left\vert X_{0,B}\right\vert =\theta \frac{s^{3}}{\lambda }(s^{2}-1)\text{
with }\theta \mid \left( s+1\right) \cdot 2\log _{p}(s).  \label{ICSB}
\end{equation}%
and by \cite%
[Tables 8.5--8.6]{BHRD} one of the following holds:
\begin{enumerate}
\item  $X_{0,B}\leq X_{0,\left\langle x\right\rangle }$ for some
isotropic $1$-dimensional $GF(s^{2})$-subspace $\left\langle x\right\rangle$ of $V$;
\item $X_{0,B}\leq (Z_{s+1} \times Z_{s+1}).S_{3}$, with $s=5,7$, and $\lambda=s^3$;
\item $SU_{2}(s)\trianglelefteq X_{0,B}\leq GU_{2}(s)$;
\item $SO_{3}(s)\trianglelefteq X_{0,B}\leq Z_{(3,s+1)}\times SO_{3}(s)$ with $s$ odd;
\item $X_{0,B}\leq 3_{+}^{1+2}:Q_{8}$, with $s=5$, and $\lambda=5^3$;
\item $X_{0,B} \cong PSL_{2}(7)$ for $s=13$, $\theta =1$ and $\lambda = 13^3$.
\end{enumerate}

Now, we are going to show that (1)--(3) and (5)--(6) are ruled out and that (4) leads to a $2$-design isomorphic to that constructed in Example \ref{Ex4}.

It is well known that $G$ is a rank $3$ group and that the non-zero isotropic, and
non-isotropic vectors of $V$ are the $G_{0}$-orbits partitioning $V^{\ast}$ and have have length $(s^{3}+1)(s^{2}-1)$ and $s^{2}(s^{3}+1)(s-1)$ respectively. Let $x_{1},x_{2}$ be representatives of such orbits, then $\left \vert B \cap x_{1}^{G_{0}} \right \vert =s^{2}-1$ and $\left \vert B \cap x_{2}^{G_{0}} \right \vert =s^{2}(s-1)$ by Lemma \ref{PP}.

\begin{enumerate}
\item[(i).] \textbf{Case (1) cannot occur.} 
\end{enumerate}
Assume that $X_{0,B}\leq X_{0,\left\langle x\right\rangle }$ for some
isotropic $1$-dimensional $GF(s^{2})$-subspace $\left\langle x\right\rangle $
of $V$. Then $X_{0,B}\cong Z_{s^{2}-1}$, or $X_{0,B}\cong Z(U):Z_{s^{2}-1}$
or $X_{0,B}\cong U:Z_{s^{2}-1}$, where $U$ is a Sylow $p$-subgroup of $X_{0}$, since $%
Z(U):Z_{s-1}$ and $U/Z(U):Z_{s^{2}-1}$ are Frobenius groups. In each case $X_{0}$ contains a unique conjugate class of subgroups
isomorphic to $X_{0,B}$. Thus $\left[ G_{0}:X_{0}\right] =\left[
G_{0,B}:X_{0,B}\right]$ and hence $r=\left[ G_{0}:G_{0,B}\right] =\left[
X_{0}:X_{0,B}\right] $. Therefore $X=T:X_{0}$ acts flag-transitively on $%
\mathcal{D}$.

Assume that $X_{0,B}\cong U:Z_{s^{2}-1}$. Since $X_{0,B}$ is maximal in $X_{0}$, it follows from Lemma \ref{cici}(2) that $B$ is a $GF(p)$-subspace of $V$. Since $U$ preserves an isotropic $1$-dimensional subspace of $V$ and acts semiregularly on the set of the remaining non-zero isotropic vectors and since $\left \vert B \cap x_{1}^{G_{0}} \right \vert =s^{2}-1$, it results that $B \cap x_{1}^{G_{0}}=\left\langle z\right\rangle^{\ast}_{GF(s^2)}$, where $\left\langle z\right\rangle_{GF(s^2)}$ is the isotropic $1$-dimensional subspace of $V$ preserved by $U$.

Since $\left \vert B \cap x_{2}^{G_{0}} \right \vert =s^{2}(s-1)$ and $\left \vert U \right \vert =s^{3}$, the stabilizer in $U$ of a non-isotropic vector $u$ in $B$ is non-trivial. Then $\left\langle u\right\rangle_{GF(s^2)}$ lies in the line of $PG_{2}(s^{2})$ tangent in $\left\langle z\right\rangle_{GF(s^2)}$ to the Hermitian unital preserved by $X_{0}$. Therefore $\left \vert u^{X_{0,B}}\right \vert =s^{2}(s+1)$ by \cite[Satz II.10.12]{Hup}, whereas $u^{X_{0,B}} \subseteq B \cap x_{2}^{G_{0}}$ with $\left \vert B \cap x_{2}^{G_{0}} \right \vert =s^{2}(s-1)$. So, this case is ruled out.    

Assume that $X_{0,B}\cong Z(U):Z_{s^{2}-1}$. Therefore $r=s^{2}(s^{3}+1)$%
. Thus $\left\vert T_{B}\right\vert \geq s$ by Lemma \ref{cici}(1), since $%
p^{m}=q^{3/2}=s^{3}$. Now $X_{B}/T_{B}$ is isomorphic to a subgroup of $X_{0}$ and contains a copy of $X_{0,B}$ as subgroup of index $s^{3}/\left\vert T_{B}\right\vert 
$ by Lemma \ref{cici}(2). Thus either $\left\vert T_{B}\right\vert =s$ or $\left\vert T_{B}\right\vert =s^{3}$, since $X_{0,B}\cong Z(U):Z_{s^{2}-1}$, since $U:Z_{s^{2}-1}$ is the unique maximal subgroup of $X_{0}$ containing $X_{0,B}$ and since $U/Z(U):Z_{s^{2}-1}$ is a
Frobenius group. In both cases there is a non-trivial subgroup of order a divisor of $\frac{s+1}{(2,s+1)}$ fixing at least a non-zero vector in $0^{T_{B}}$, but this is impossible for $s \neq 3$ (e.g. see \cite[Satz II.10.12]{Hup}). So $s=3$ and $X_{0,B}\cong Z_{3}:Z_{8}$, but there is a cyclic groups of order at least $4$ fixing a point in $0^{T_{B}}$ and this case is excluded.

Assume that $X_{0,B}\cong Z_{s^{2}-1}$. We may argue as in the previous case and we get that $\left\vert
T_{B}\right\vert$ is either $1$ or $s$ or $s^{3}$. In the two latter cases is a non-trivial subgroup of order a divisor of $\frac{s+1}{(2,s+1)}$ fixing at least a non-zero vector in $0^{T_{B}}$ and we again reach a contradiction as above. Thus $\left\vert T_{B}\right\vert =1$ and hence $%
X_{B}\cong X_{0,\left\langle x\right\rangle }$ by Lemma \ref{cici}(2) and by \cite%
[Tables 8.5--8.6]{BHRD}. It follows from \cite[Theorem 2.14.(i)]{KaLib3} that $H^{1}(X_{0},V)=0$, also $\left\vert
H^{1}(X_{0,\left\langle x\right\rangle },V)\right\vert \leq \left\vert
H^{1}(X_{0},V)\right\vert $ by \cite[2.3.(g)]{CPS}, as $\left\langle x\right\rangle $ is isotropic. Then $X_{B}=X_{0,\left\langle x\right\rangle }^{\tau
}$ for some $\tau \in T$ by \cite[17.10]{Asch2}, hence $X_{0,\left\langle
x\right\rangle }$ acts transitively on $B^{\tau ^{-1}}$. Then all the vectors contained in $B^{\tau ^{-1}}
$ are either non-zero isotropic or non-isotropic. However, this is impossible since $X$ acts flag-transitively on $\mathcal{D}$ and each block incident with $0$ must contain (non-zero) isotropic vectors as well as non-isotropic ones.

\begin{enumerate}
\item[(ii).] \textbf{Cases (2), (5) and (6) cannot occur.} 
\end{enumerate}
Assume that (2) occurs. If $s=5$, from (\ref{ICSB}) we derive that either $X_{0,B} \cong S_{4}$ or $3_{+}^{1+2} \leq X_{0,B}$. In the latter case $X_{0,B}$ contains the center of $SU_{3}(5)$, which fixes at least two points on $B$ since $k=125$ is equivalent to $2$ modulo $3$, and we reach a contradiction. For the same reason $B$ cannot be a $GF(5)$-subspace of $V$ when $X_{0,B} \cong S_{4}$. On the other hand, $X_{B}/T_{B}$ is isomorphic to a subgroup of $X_{0}$ containing an isomorphic copy of $X_{0,B}$ as a subgroup of index $5^{3}/\left\vert T_{B} \right \vert$ by Lemma \ref{cici}(2). It follows from \cite{At} that $\left\vert T_{B} \right \vert = 5^{2}$, since $B$ cannot be a $GF(5)$-subspace of $V$. Then $X_{0,B}$ acts on $W$, where $W=0^{T_{B}}$, which is not a $1$-dimensional $GF(25)$-subspace of $V$, since $X_{0,B}$ does not fix points of $PG_{2}(25)$. Hence $W$ intersects six points of $PG_{2}(25)$ each of them in a non-zero vector of $V$. Also, since there are no subgroups of $GL_{2}(5)$ isomorphic to $S_{4}$, it follows that $X_{0,B}(W) \neq 1$. Hence $Z_{2} \times Z_{2} \leq X_{0,B}(W)$, since $X_{0,B} \cong S_{4}$. Then $Z_{2} \times Z_{2}$ fixes six points of $PG_{2}(25)$, namely those intersecting $W$ in a non-zero vector of $V$, but this is not the case since $Z_{2} \times Z_{2}$ acts faithfully on $PG_{2}(25)$ as a groups of homologies in a triangular configuration. Thus $s=5$ is ruled out.

Suppose that $s=7$. Hence $(Z_{4} \times Z_{4}).S_{3} \leq X_{0,B}$, since $48$ divides the order of $X_{0,B}$ by (\ref{ICSB}). By Lemma \ref{cici}(2) $X_{B}/T_{B}$ is isomorphic to a subgroup of $X_{0}$ containing an isomorphic copy of $X_{0,B}$ as a subgroup of index $7^{3}/\left\vert T_{B} \right \vert$. It follows from \cite{At} that $\left\vert T_{B} \right \vert \geq 7^{2}$, since $(Z_{4} \times Z_{4}).S_{3} \leq X_{0,B}$. If $\left\vert T_{B} \right \vert = 7^{2}$, the admissible candidates of maximal subgroups of $X_{0}$ containing an isomorphic copy of $X_{B}/T_{B}$ are either $SL_{2}(7):Z_{8}$ or $PGL_{2}(7)$ by \cite{At}, however none of these contains a subgroup isomorphic to $(Z_{4} \times Z_{4}).S_{3}$. Therefore $\left\vert T_{B} \right \vert = 7^{3}$ and hence $B$ is a $3$-dimensional $GF(7)$-subspace of $V$. So $X_{0,B} \leq PGL_{2}(7)$ by \cite{At}, and we again reach a contradiction as $PGL_{2}(7)$ does not contains subgroups isomorphic to $(Z_{4} \times Z_{4}).S_{3}$, whereas $X_{0,B}$ does it. This completes the exclusion of (2).   

Assume that (5) occurs. From (\ref{ICSB}) we derive that either $X_{0,B} \cong Z_{3} \times Q_{8}$ or that $X_{0,B} \cong 3_{+}^{1+2}:Q_{8}$. In each case $X_{0,B}$ contains the center of $SU_{3}(5)$ and we reach a contradiction as in case (2).

Assume that (6) occurs. By Lemma \ref{cici}(2) $X_{B}/T_{B}$ is isomorphic to a subgroup of $X_{0}$ containing an isomorphic copy of $X_{0,B}$ as a subgroup of index $13^{3}/\left\vert T_{B} \right \vert$, it follows that $\left\vert T_{B} \right \vert = 13^{3}$, hence $B$ is a $3$-dimensional $GF(13)$-subspace of $V$, since $X_{0,B} \cong PSL_{2}(7)$. Then $X_{0,B}$ lies in a $\mathcal{C}_{5}$-member of $X_{0}$, but this is impossible by \cite[Table 8.5]{BHRD}. 

\begin{enumerate}
\item[(iii).] \textbf{Case (3) cannot occur.} 
\end{enumerate}
Assume that $SU_{2}(s)\trianglelefteq X_{0,B}\leq GU_{2}(s)$. Then $r=\left[
G_{0}:G_{0,B}\right] =\left[ X_{0}:X_{0,B}\right] $ by \cite[Table
8.5]{BHRD}, hence $X=T:X_{0}$ acts flag-transitively on $\mathcal{D}$. Moreover, $r=s^{2}(s^{3}+1)$ and therefore $X_{0,B}\cong SU_{2}(s)$. By Lemma \ref{cici}(2) $X_{B}/T_{B}$ is isomorphic to a subgroup of $X_{0}$ containing an isomorphic copy of $X_{0,B}$ as a subgroup of index $s^{3}/\left\vert T_{B} \right \vert$, it follows that $\left\vert T_{B} \right \vert = s^{3}$ and hence $B$ is a $GF(p)$-subspace of $V$.

Each Sylow $p$-subgroup of $X_{0,B}$ fixes a unique isotropic $1$-dimensional $GF(s^{2})$-subspace of $V$ and distinct Sylow $p$-subgroups of $X_{0,B}$ fix distinct isotropic $1$-dimensional $GF(s^{2})$-subspaces of $V$. Hence $B$ intersects at least $s+1$ isotropic $1$-dimensional $GF(s^{2})$-subspaces of $V$, say $\left\langle y_{i} \right\rangle_{GF(s^{2})}$, with $i=1,...,s+1$, each of them in at least a non-zero vector, that we may assume to be $y_{i}$. Moreover, $X_{0,B}$ permutes $2$-transitively the set $\{\left\langle y_{i} \right\rangle _{GF(s^{2})} : \; i=1,...,s+1 \}$. Since $X_{0,B,\left\langle y_{i} \right\rangle_{GF(s^{2})}}$ partitions the non-zero vectors of $\left\langle y_{i} \right\rangle _{GF(s^{2})}$ into $s+1$ orbits each of length $s-1$, since $\left \vert B \cap x_{1}^{X_{0}} \right \vert =s^{2}-1$ and since $B \cap \left\langle y_{i} \right\rangle _{GF(s^{2})}$ is a $GF(p)$-subspace of $V$, it results that $B \cap \left\langle y_{i} \right\rangle_{GF(s^{2})}=\left\langle y_{i} \right\rangle_{GF(s)}$ for each $i=1,...,s+1$ and that $B \cap x_{1}^{X_{0}} = \bigcup ^{s+1}_{i=1} \left\langle y_{i} \right\rangle ^{\ast}_{GF(s)}$.

If $B \cap x_{1}^{X_{0}}  \not \subset \left\langle y_{1},y_{2} \right\rangle^{\ast}_{GF(s^{2})}$, then there is $3 \leq i_{0} \leq s+1$ such that $B=\left\langle y_{1},y_{2},y_{i_{0}} \right\rangle_{GF(s)}$, as $k=s^{3}$. Then $X_{0,B}$ lies in a $\mathcal{C}_{5}$-member of $X_{0}$, but this is impossible by \cite[Table 8.5]{BHRD}. Thus $B \cap x_{1}^{X_{0}} \subset \left\langle y_{1},y_{2} \right\rangle^{\ast}_{GF(s^{2})}$ and hence $B \cap x_{1}^{X_{0}}= \left\langle y_{1},y_{2} \right\rangle^{\ast}_{GF(s)}$.

It is not difficult to see that $X_{0,B}$ preserves the decomposition $V=\left\langle y_{1},y_{2} \right\rangle_{GF(s^{2})} \oplus\left\langle y_{1},y_{2} \right\rangle^{\perp}_{GF(s^{2})} $ and that each $X_{0,B}$-orbit on $V \setminus (\left\langle y_{1},y_{2} \right\rangle_{GF(s^{2})} \cup\left\langle y_{1},y_{2} \right\rangle^{\perp}_{GF(s^{2})})$ is of length $s(s^{2}-1)$. Hence $B \subset (\left\langle y_{1},y_{2} \right\rangle_{GF(s^{2})} \cup\left\langle y_{1},y_{2} \right\rangle^{\perp}_{GF(s^{2})})$, since $k=s^3$ and since $ \left \vert B \cap \left\langle y_{1},y_{2} \right\rangle_{GF(s^{2})} \right \vert \geq s^{2}$. Actually, $B \subset \left\langle y_{1},y_{2} \right\rangle_{GF(s^{2})}$ since $B$ is a $GF(p)$-subspace of $V$ sharing at least $s^{2}$ vectors with $\left\langle y_{1},y_{2} \right\rangle_{GF(s^{2})}$. The $X_{0,B}$-orbits in $\left\langle y_{1},y_{2} \right\rangle_{GF(s^{2})} \setminus \left\langle y_{1},y_{2} \right\rangle_{GF(s)}$ have length $s(s^{2}-1)$ and each of these must be disjoint from $B$, so we reach a contradiction since $k=s^{3}$ and since $\left\langle y_{1},y_{2} \right\rangle_{GF(s)} \subset B$. Thus, this case is excluded.

\begin{enumerate}
\item[(iv).] \textbf{Lemma's statement holds.} 
\end{enumerate}
Assume that $SO_{3}(s)\trianglelefteq X_{0,B}\leq
Z_{(3,s+1)}\times SO_{3}(s)$ with $s$ odd. Suppose that $s \equiv 2 \pmod{3}$ and that $X_{0,B}\cong Z\times SO_{3}(s)$, where $Z \cong Z_{3}$. Then $Z$ fixes at least a non-zero vector of $B \cap x_{2}^{G_{0}}$, as $\left\vert B \cap x_{2}^{G_{0}}\right\vert = s^{2}(s-1)$, and we reach a contradiction as $Z$ is the center of $SU_{3}(s)$. Thus $X_{0,B} \cong SO_{3}(s)$. 

Since $\left\vert G_{0,B}\right\vert
_{p}=s\cdot \frac{\left( \log _{p}(s)\right) _{p}}{a}$ for some positive divisor $a$ of $\left( \log _{p}(s)\right) _{p}$ by (\ref{ICSB}), hence $\lambda =s^{2} \cdot a$. Then $T_{B}\neq 1$ by Lemma \ref{cici}(1), since $p^{m}=q^{3/2}=s^{3}$ and since $a<s$. Then $%
\Omega _{3}(s)$ acts on $0^{T_{B}}$. Clearly $\Omega _{3}(s)$ does not fix a non-zero vectors in $%
0^{T_{B}}$ by \cite[Table 8.5]{BHRD}. Thus $\left\vert 0^{T_{B}}\right\vert \geq s^{3}$ by \cite[Proposition
5.4.11]{KL}. Therefore $B=0^{T_{B}}$, as $k=s^{3}$ and $0^{T_{B}}\subseteq B$, hence $B$ is a $GF(p)$-subspace of $V$.

$X_{0,B}$ acts faithfully on $PG_{2}(s^{2})$ and one $X_{0,B}$-orbit on $PG_{2}(s^{2})$ is $\mathcal{H}\cap \mathcal{C}$, where $\mathcal{H}$ is the Hermitian unital of order $s$ preserved by $G_{0}$ and where $\mathcal{C}$ is a conic preserved by $X_{0,B}$. Each Sylow $p$-subgroup of $\Omega _{3}(s)$ fixes each non-zero vector of
a point $\left\langle y_{0}\right\rangle _{GF(s^{2})}$ of $\mathcal{H}\cap 
\mathcal{C}$ and acts regularly on $(\mathcal{H}\cap \mathcal{C})\backslash
\{\left\langle y_{0}\right\rangle _{GF(s^{2})}\}$. Moreover, the $X_{0,B}$-orbits on $\mathcal{H}\setminus 
\mathcal{C}$ have length $s(s^{2}-1)/2$. Therefore either $B\cap
x^{G_{0}}=\left\langle y_{0}\right\rangle _{GF(s^{2})}^{\ast }$, or $%
\left\vert B\cap \left\langle y\right\rangle _{GF(s^{2})}\right\vert >0$ for
each $\left\langle y\right\rangle _{GF(s^{2})}\in \mathcal{H}\cap \mathcal{C}
$, since $\left\vert B\cap x_{1}^{G_{0}}\right\vert =s^{2}-1$. If $B\cap
x_{1}^{G_{0}}=\left\langle y_{0}\right\rangle _{GF(s^{2})}^{\ast }$ then $%
SO_{3}(s) \leq X_{0,\left\langle y_{0}\right\rangle
_{GF(s^{2})}}$, which is impossible.
Therefore, $\left\vert B\cap \left\langle y\right\rangle
_{GF(s^{2})}\right\vert >0$ for each $\left\langle y\right\rangle
_{GF(s^{2})}\in \mathcal{H}\cap \mathcal{C}$. Moreover, since $\Omega
_{3}(s)_{\left\langle y\right\rangle _{GF(s^{2})}}$ is a Frobenius group with cyclic complement of order $(s-1)/2$ acting semiregularly on $%
\left\langle y\right\rangle _{GF(s^{2})}^{\ast }$, and since $B$ is a $GF(p)$%
-subspace of $V$, it follows that $B\cap \left\langle y\right\rangle
_{GF(s^{2})}$ is a $GF(s)$-subspace of $\left\langle y\right\rangle
_{GF(s^{2})}$. Thus $B$ is a $3$-dimensional $GF(s)$-subspace of $V$, since $\mathcal{C}$ is a conic, and
hence $B$ induces a Baer subplane of $PG_{2}(s^{2})$. Moreover, $B$ is the unique $3$-dimensional $GF(s)$-subspace of $V$ preserved by $X_{0,B}$. Thus, we assume that $B$ is as in Example \ref{Ex5}, since the $(3,s+1)$ conjugacy $X_{0}$-classes of subgroup isomorphic to $X_{0,B}$ are fused in $\mathrm{Aut}(X_{0}) \cong \Gamma U_{3}(s)$ (see \cite[Table 8.5]{BHRD}). Therefore, $\mathcal{D}$ is
a $2$-$(q^{3},q^{3/2},q)$ design isomorphic to that constructed in Example %
\ref{Ex4}. This completes the proof.
\end{proof}

\begin{lemma}
\label{Omegan}If $X_{0}\cong \Omega _{n}^{-}(q)$, then $n=4$.
\end{lemma}

\begin{proof}
Assume that $X_{0}\cong \Omega _{n}^{-}(q)$. Then $C:Y\leq X_{0,x}$, where $%
C $ is a $p$-group and $Y$ is a isomorphic to $\Omega _{n-2}^{-}(q)$, by 
\cite[Lemma 4.1.12]{KL}. Let $B$ be any block of $\mathcal{D}$ incident with 
$0,x$. Then $P(Y)\leq \left\vert B^{Y}\right\vert$, where $P(Y)$
denotes the minimal primitive permutation representation of $Y$. Hence
\begin{equation}
q^{n/2}\geq \lambda \geq \left\vert B^{Y}\right\vert >\frac{\left(
q^{(n-2)/2}+1\right) \left( q^{(n-2)/2-1}-1\right) }{q-1}  \label{ohm}
\end{equation}%
by \cite[Proposition 52.1 and Table 5.2.A]{KL}, and we reach a contradiction for $n>6$. Thus $Y$ fixes each block of $\mathcal{D}$ incident with $0,x$ for $n>6$,
and hence it lies in the kernel of the action of $C:Y$ on the set of blocks of $%
\mathcal{D}$ incident with $0,x$. Therefore, $C:Y$ fixes each block of $%
\mathcal{D}$ incident with $0,x$. So $\frac{\left( q^{n/2}+1\right) \left(
q^{n/2-1}-1\right) }{q-1}\mid r$, which cannot occur since $%
r=p^{f}(q^{n/2}+1)$ and $n>6$. Thus $n\leq 6$.

Assume that $n=6$. Then $\left[ X_{0}:X_{0,B}\right] \mid p^{f}(q^{3}+1)$
and hence $p^{f}\mid q^{3}$. Then either $X_{0,B}\leq M$, where $M$ is
either $\Omega _{5}(q).Z_{2}$ or $Sp_{4}(q)$ according to whether $q$ is odd
or even, respectively, by \cite[Tables 8.33--8.34]{BHRD}. Since $\left[
X_{0}:M\right] =q^{2}\frac{q^{3}+1}{(2,q-1)}$, it follows that $\left[
M:X_{0,B}\right] \mid \frac{p^{f}}{q^{2}}(2,q-1)$ and hence $M^{\prime} \leq X_{0,B} \leq M^{\prime}.Z_{(2,q-1)}$ by \cite[Proposition 52.1 and Table 5.2.A]{KL}, since $r=p^{f}(q^{3}+1)$. Thus $\lambda
=q^{2}$.

The group $G_{B}/T_{B}$ is isomorphic to a subgroup $J$ of $G_{0}$ such
that $[G_{0}:J]=\frac{p^{t}}{q}(q^{3}+1)$ and $J$ contains an isomorphic
copy of $G_{0,B}$ by Lemma \ref{cici}(2). Then $\left[ X_{0}:J\cap X_{0}%
\right] \mid \frac{p^{t}}{q}(q^{3}+1)$ and arguing as above with $J\cap X_{0}$ in the role of $X_{0,B}$
we deduce that $k=p^{t}=q^{3}$. Therefore, $B$ is a $GF(p)$-subspace of $V$ and hence $%
(X_{0,B})^{B}\leq GL_{m}(p)$. Since the order of $X_{0,B}$ is divisible by a
primitive prime divisor of $p^{4m/3}-1$, being $q^{6}=p^{2m}$, whereas the order of $GL_{m}(p)$ is
not, it follows that $\left( X_{0,B}\right) ^{\prime }$ fixes $B$ pointwise.
So $B\subseteq \left\langle z\right\rangle _{GF(q)}$ for some non-singular
vector $z$ of $V$, and we reach a contradiction since $k=q^{3}$.

Finally, $n=2$ cannot occur. Indeed, in this case $X_{0}\cong \Omega
_{2}^{-}(q)\cong D_{2\frac{q+1}{(2,q-1)}}\leq \Gamma L_{1}(q^{2})$ and it
contradicts the definition of $q$.
\end{proof}

\begin{lemma}
\label{Omega4no2}$X_{0}$ is not isomorphic to $\Omega _{n}^{-}(q)$, $n\geq 2$%
.
\end{lemma}

\begin{proof}
$X_{0}\cong \Omega _{4}^{-}(q)$ by Lemma \ref{Omegan}. Since $r=\lambda
(q^{2}+1)$, where $\lambda \mid q^{2}$, and since $X_{0}\trianglelefteq
G_{0} $, it follows that $\left[ X_{0}:X_{0,B}\right] \theta =\lambda
(q^{2}+1)$ for some positive integer $\theta$. On the other hand $\theta \mid \lbrack G_{0}:X_{0}]$ and hence $%
\theta \mid (q-1)\cdot (2,q-1)\cdot m$, as $q^{2}=p^{m}$. Thus $\theta \mid
(2m,\lambda (q^{2}+1))$ and hence 
\begin{equation}
\left\vert X_{0,B}\right\vert =\frac{q^{2}}{\lambda }\frac{q^{2}-1}{(2,q-1)}%
\theta \text{.}  \label{ord}
\end{equation}%
Therefore, by \cite[Table 8.17]{BHRD} one of the following holds:
\begin{enumerate}
\item $X_{0,B}\leq X_{0,\left\langle x\right\rangle _{GF(q)}}$, where $x$ is a non-zero singular vector of $V$;
\item $PSL_{2}(q)\trianglelefteq X_{0,B}\leq $ $PGL_{2}(q)$;
\item $X_{0,B}\cong D_{2\frac{(q^{2}-1)}{(2,q-1)}}$;
\item $X_{0,B}\leq A_{5}$ and $q=3$.
\end{enumerate}
 We are going to prove that none of the cases (1)--(4) occurs.
\begin{enumerate}
\item[(i).] \textbf{Case (1) and (4) cannot occur.}
\end{enumerate}
Assume that $X_{0,B}\leq X_{0,\left\langle x\right\rangle _{GF(q)}}$. Then either $X_{0,B}= X_{0,\left\langle x\right\rangle _{GF(q)}}$ or $X_{0,B}$ is cyclic of order $\frac{q^{2}-1}{(2,q-1)}$ by (\ref{ord}), since $%
X_{0,\left\langle x\right\rangle _{GF(q)}}$ is a Frobenius group of order $%
q^{2}\frac{q^{2}-1}{(2,q-1)}$ with a cyclic complement of order $\frac{%
q^{2}-1}{(2,q-1)}$. In both cases $X_{0}$ contains a unique conjugacy class of subgroups isomorphic to $X_{0,B}$. Therefore $X=TX_{0}$ acts flag-transitively on $\mathcal{D}$. So $X_{0,B} \neq X_{0,\left\langle x\right\rangle _{GF(q)}}$ since $\lambda \geq 2$. Thus $X_{0,B}$ is cyclic of order $\frac{q^{2}-1}{(2,q-1)}$ and hence $r=q^{2}(q^{2}+1)$ and $\lambda =q^{2}$.

Since $X_{B}/T_{B}$ is isomorphic to a subgroup of $X_{0}$ containing $X_{0,B}$ as subgroup of
index $q^{2}/\left\vert T_{B}\right\vert $ by Lemma \ref{cici}(2), it follows that $X_{B}/T_{B}$ is a Frobenius group with cyclic complement of order $%
\frac{q^{2}-1}{(2,q+1)}$ by \cite[Table 8.17]{BHRD}. Thus, either $T_{B}=1$ or $%
\left\vert T_{B}\right\vert =q^{2}$. 

The $X_{0,B}$-orbits on the set of non-zero singular vectors of $V$ have length $\frac{q^{2}-1}{(2,q+1)}$ except for two ones of have length $q-1$, which are $\left\langle y_{1}\right\rangle
_{GF(q)}^{\ast }$ and $\left\langle y_{2}\right\rangle _{GF(q)}^{\ast }$ for some suitable non-zero singular vectors  $y_{1},y_{2}$ of $V$. On the other hand $\left\vert B\cap x^{G_{0}}\right\vert =q-1$, where $x$ is a non-zero singular vector of $V$, by Lemma \ref{PP}. Thus $ \left \vert B \cap \left\langle y_{i}\right\rangle _{GF(q)} \right \vert >1$ for $i=1$ or $2$. We may assume that $i=1$. Hence $B\cap x^{G_{0}}= \left\langle y_{1}\right\rangle _{GF(q)} ^{\ast }$, since $\left\langle y_{1}\right\rangle_{GF(q)}^{\ast }$ is a $X_{0,B}$-orbit. 

Assume that $\left\vert T_{B}\right\vert =q^{2}$. Then $B$ is a $GF(p)$-subspace of $V$. The cyclic subgroup $Y$ of $X_{0,B}$ of order $\frac{q+1}{(2,q+1)}$ fixes $B\cap x^{G_{0}}$
pointwise. Thus $\left\vert B\cap \mathrm{Fix}(Y)\right\vert \geq q$%
, where $\mathrm{Fix}(Y)$ is a $2$-dimensional $GF(q)$-subspace of $%
V$. Suppose $\left\vert B\cap \mathrm{Fix}(Y)\right\vert = q$. Then $q>3$, since $X_{0}\cong \Omega
_{4}^{-}(3)\cong PSL_{2}(9)\cong A_{6}$ and this case is ruled out by
Lemma \ref{NotFDPM}. Then there is a non-trivial subgroup $Y_{1}$ of $Y$ fixing a vector in $B \setminus  \mathrm{Fix}(Y)$, since $\left\vert B\setminus \mathrm{Fix}(Y)\right\vert = q(q-1)$. However, this is impossible as $\mathrm{Fix}(Y_{1})=\mathrm{Fix}(Y)$. Therefore $\left\vert B\cap \mathrm{Fix}(Y)\right\vert > q$. Clearly, $X_{0,B}$ acts on $\mathrm{Fix}(Y)$. In particular $X_{0,B}$ induces a cyclic group of order $q-1$ acting regularly on $\mathrm{Fix}(Y)$, and $\left\langle y_{1}\right\rangle _{GF(q)} ^{\ast }, \left\langle y_{2}\right\rangle _{GF(q)} ^{\ast }$ are the unique $X_{0,B}$-orbits consisting of non-zero singular vectors lying in $\mathrm{Fix}(Y)$. Hence $\left\vert B\cap \mathrm{Fix}(Y)\right\vert = q+a(q-1)$ for some positive integer $a$ such that $a \leq q$. On the other hand $\left\vert B\cap \mathrm{Fix}(Y)\right\vert=p^{e}$, with $q<p^e \leq q^{2}$, as both $B$ and $\mathrm{Fix}(Y)$ are $GF(p)$-subspaces of $V$. So $a=q$ and hence $B=\mathrm{Fix}(Y)$, since $k=q^{2}$. However this is impossible since $\left\vert B\cap x^{G_{0}}\right\vert =q-1$, whereas $\left\vert \mathrm{Fix}(Y)%
\cap x^{G_{0}}\right\vert =2(q-1)$.

Assume that $T_{B}=1$ and that $q$ is odd. As above, $q>3$ by
Lemma \ref{NotFDPM} (this fact is independent of the assumptions on the order of $T_{B}$).  Then $H^{1}(X_{0},V)=0$ by 
\cite[Theorem 2.14.(i)]{KaLib3}, also $\left\vert H^{1}(X_{0,\left\langle
x\right\rangle },V)\right\vert \leq \left\vert H^{1}(X_{0},V)\right\vert $
by \cite[2.3.(g)]{CPS}, since $\left\langle
x\right\rangle$ is singular. Then $X_{B}=X_{0,\left\langle x\right\rangle }^{\tau
}$ for some $\tau \in T$ by \cite[17.10]{Asch2}, hence $X_{0,\left\langle
x\right\rangle }=X_{B^{\tau ^{-1}}}$. Therefore the Sylow $p$-subgroup of $%
X_{0,\left\langle x\right\rangle }$ acts regularly on $B^{\tau ^{-1}}$, hence the actions of any cyclic subgroup of $X_{0,\left\langle
x\right\rangle }$ on $B^{\tau ^{-1}}$ and on the Sylow $p$-subgroup of $%
X_{0,\left\langle x\right\rangle }$ are equivalent by \cite[Proposition 4.2]{Pass}. Thus, each cyclic subgroup of $X_{0,\left\langle
x\right\rangle }$ fixes a point in $B^{\tau ^{-1}}$ and partitions the remaining points of $B^{\tau ^{-1}}$ into $2$ orbits each of length $\frac{q^{2}-1}{2}$. A similar conclusion holds for $B$, since $X$ acts block-transitively on $\mathcal{D}$. However, this is impossible since $\left\vert B\cap x^{G_{0}}\right\vert =q-1$.

Assume that $T_{B}=1$ and that $q$ is even. Let $K$ be the Sylow $2$-subgroup of $X_{0,\left\langle y_{1}\right\rangle_{GF(q)}}$. Therefore $K$ fixes $\left\langle y_{1}\right\rangle _{GF(q)}$ pointwise, hence $B^{K}$ is the set of the $\lambda=q^{2}$ blocks of $\mathcal{D}$ incident with $0$ and $y_{1}$. Furthermore, each of them contains $\left\langle y_{1}\right\rangle_{GF(q)}$. Thus $K$ centralizes $\tau$, where $\tau \in T_{\left\langle y_{1}\right\rangle _{GF(q)}}$, and hence $\tau$ permutes the blocks in $B^{K}$. Then there is $\gamma \in K$ such that $\gamma\tau \in X_{B}$. Let $\eta$ in $Y$, where $Y$ is defined as above, then $(\gamma\tau)^{\eta} \in X_{B}$. Also, $(\gamma\tau)^{\eta}=\gamma^{\eta}\tau$, with $\gamma^{\eta} \in K$ since $\tau \in T_{\left\langle y_{1}\right\rangle _{GF(q)}}$, $Y$ fixes $\left\langle y_{1}\right\rangle _{GF(q)}$ pointwise and $Y$ normalizes $K$. Then $\gamma\gamma^{\eta}=\gamma\tau\gamma^{\eta}\tau \in X_{B}$, since $\tau$ centralizes $\gamma^{\eta}$ and since $o(\tau)=2$. Consequently $\gamma\gamma^{\eta} \in K \cap X_{0,B}$, so $\gamma^{\eta}=\gamma$ since $K$ is an elementary abelian $2$-group and since $X_{0,B}$ is cyclic of (odd) order $q^{2}-1$. Therefore $\gamma=1$ by \cite[Proposition 17.2]{Pass}, since $K:Y$ is Frobenius group. So $\tau =\gamma \tau \in X_{B}$, but this contradicts $T_{B}=1$. Thus (1) and (4) are ruled out.

\begin{enumerate}
\item[(ii).] \textbf{Case (2) cannot occur.}
\end{enumerate}  
Assume that\ $PSL_{2}(q)\trianglelefteq X_{0,B}\leq $ $PGL_{2}(q)$. Since $%
\left\vert B\cap x^{G_{0}}\right\vert =q-1$ and since the minimal
primitive permutation representation of $PSL_{2}(q)$ is at least $q$ for $%
q\neq 9$ and $6$ for $q=9$ by \cite[Theorem 5.2.2]{KL}, it follows that $%
PSL_{2}(q)$ fixes at least a point in $B\cap x^{G_{0}}$ and hence it lies in the
stabilizer of a singular $1$-dimensional subspace of $V$, a contradiction.

\begin{enumerate}
\item[(iii).] \textbf{Case (3) cannot occur.}
\end{enumerate}
Assume that $X_{0,B}\cong D_{2\frac{(q^{2}-1)}{(2,q-1)}}$. Since $X_{0}$ has a unique conjugacy class of subgroups isomorphic to $X_{0,B}$ by \cite[Table 8.17]{BHRD}, it follows that $B^{G_{0}}=B^{X_{0}}$. Hence $X=TX_{0}$ acts flag-transitively on $\mathcal{D}$. The group $X_{B}/T_{B}$ is isomorphic to
subgroup of $X_{0}$ containing $X_{0,B}$ as subgroup of index $%
q^{2}/\left\vert T_{B}\right\vert $ by Lemma \ref{cici}(2). Since $X_{0,B}$
is maximal in $X_{0}$, it follows that $\left\vert T_{B}\right\vert =q^{2}$.
Thus $B$ is a $GF(p)$-subspace of $V$.

The $X_{0,B}$-orbits
on the set of non-zero singular vectors $GF(q)$-subspaces $V$ are either $\left\langle y_{1}\right\rangle _{GF(q)}^{\ast }\cup
\left\langle y_{2}\right\rangle _{GF(q)}^{\ast }$, where $y_{1},y_{2}$ have the same meaning as in (i), or they have length $q^{2}-1$.
Since $\left\vert B\cap x^{G_{0}}\right\vert =q-1$ and since $X_{0,B}$ contains involutions
switching $\left\langle y_{1}\right\rangle _{GF(q)}$ and $\left\langle
y_{2}\right\rangle _{GF(q)}$, it follows that $%
\left\vert B\cap x^{G_{0}}\cap \left\langle y_{i}\right\rangle
_{GF(q)}^{\ast }\right\vert =(q-1)/2$ for each $i=1,2$. Thus $\left\vert B\cap \left\langle
y_{i}\right\rangle _{GF(q)}\right\vert =(q+1)/2$, whereas $B\cap
\left\langle y_{i}\right\rangle _{GF(q)}$ is a $GF(p)$-subspace of $V$. So $X_{0}\cong \Omega _{4}^{-}(q)$ is excluded and the assertion follows.
\end{proof}

\bigskip
Now, we are in position to prove Theorem \ref{classical}.
\bigskip
\bigskip
\begin{proof}[Proof of Theorem \protect\ref{classical}.]
Suppose that $X_{0}$ is one of the groups $SL_{n}(q)$, $SU_{n}(q^{1/2})$
with $n$ odd, $Sp_{n}(q)$ or $\Omega _{n}^{-}(q)$. The Latter is ruled out
in Lemma \ref{Omega4no2}. In the remaining cases assertion (1), (2), and
(1) and (3) follows from Proposition \ref{SLSp}, Lemma \ref{SU3}
and Lemma \ref{sympl} respectively.
\end{proof}

\section{Completion of the proof of Theorem \ref{main}} \label{t5}
This small final section is devoted to the completion of the the proof of Theorem \ref{main}.

\begin{proof}[Proof of Theorem \protect\ref{main}.]
Let $\mathcal{D}$ be a $2$-$(p^{2m},p^{m},\lambda )$, with $\lambda \mid
p^{m}$, admitting a flag-transitive automorphism group $G=TG_{0}$. If $\lambda =1$, then assertions (1)--(5) follows from Theorem \ref{lambada}. Hence, assume that $\lambda \geq 2$. If $m>1$ and $(p,2m)\neq (2,6)$, then assertion (1), (6)--(8), (8) for $q=2$ and (9)--(10), and (11) follow from Lemma \ref{jedan} and Theorems \ref{classical}, \ref{ClassS} and \ref{normIrre} respectively. Thus, in order to complete the proof, we need to tackle the remaining cases $m=1$ or $(p,2m)=(2,6)$. The former implies that either $G_{0}\leq  \Gamma L_{1}(p^{2})$ or $SL_{2}(p) \trianglelefteq G_{0}$, since $G_{0}\leq GL_{2}(p)$ and $r=p(p+1)$ divides $\left\vert G_{0}\right\vert $. The former yields assertion (1) by Lemma \ref{jedan}, whereas we may use the arguments
in the proof of Proposition \ref{SLSp} to rule out the latter.

Finally, assume that $(p,2m)=(2,6)$. Then $G_{0}\leq SL_{6}(2)$ and $r=2^{f}\cdot 3^{2}$ with $f=1,2,3$. We may use
the same argument of Lemma \ref{SL} to rule out the case $G_{0}=SL_{6}(2)$.
Then $G_{0}\leq M$, where $M$ is a maximal subgroup of $SL_{6}(2)$. Then $M$
lies in a member $\mathcal{C}_{i}(SL_{6}(2))$, where $i=1,2,3$ or $8$,
by \cite[Tables 8.24--8.25]{BHRD}. Since $G_{0}$ is maximal in $G$ (recall that G acts point-primitively on $\mathcal{D}$), and $%
G=TG_{0}$, it follows that $G_{0}$ acts irreducibly on $V$. Hence $i\neq 1
$.

Suppose that $i=2$. Then $G_{0}$ preserves the sum decomposition $%
V=V_{1}\oplus V_{2}$, where $\dim V_{1}=\dim V_{2}=3$ and $M\cong
SL_{3}(2)\wr Z_{2}$ again by \cite[Table 8.24]{BHRD}. One $M$-orbit is $V_{1}^{\ast }\cup V_{2}^{\ast }$, which is of length $14$, is a union of $G_{0}$-orbits. However,
this is impossible since each $G_{0}$-orbit is of length divisible $3^{2}$
by Lemma \ref{PP}, being $r=2^{f}\cdot 3^{2}$.

Suppose that $i=3$. Then either $M\cong \Gamma L_{2}(8)$ or $M\cong \Gamma
L_{3}(4)$. Suppose that the former occurs. Since $18$ divides the order of $%
G_{0}$, and since the unique subgroup of order $3^{2}$ which do not fix
points are cyclic by \cite{At}, it follows that $G_{0}$ contains a cyclic
group $C$ of order $3^{2}$.

If $G_{0}\cap Z(GL_{2}(8))\neq 1$ then $%
Z_{7}\cong Z(GL_{2}(8))\leq G_{0}$, hence $Z(GL_{2}(8))$ preserves each
block incident with $0$, since $r$ is coprime to $7$ and since $G_{0}$ acts
transitively on the set of blocks incident with $0$. Then $B\cap B^{\prime }$
is preserved by $Z_{7}$, where $B^{\prime }$ is any block of $\mathcal{D}$
distinct from $B$ such that $0,x\in B\cap B^{\prime }$, with $x\neq 0$, as $%
\lambda \geq 2$. However, this is impossible since $k=2^{3}$. Thus $%
G_{0}\cap Z(GL_{2}(8))=1$.

The group $\Gamma L_{2}(8)$ contains two conjugacy classes of cyclic
subgroups of order $9$, say $\mathcal{K}_{1}$ and $\mathcal{K}_{2}$. If $%
C\in \mathcal{K}_{1}$ then $C$ is a subgroup of a Singer cycle and hence $%
N_{M}(C)\cong Z_{63}.Z_{2}.Z_{3}$, whereas if $C\in \mathcal{K}_{2}$ then $%
N_{M}(C)\cong Z_{9}.Z_{3}$. If $C$ is not normal in $G_{0}$ then $%
SL_{2}(8)\trianglelefteq G_{0}$ regardless $C\in \mathcal{K}_{1}$ or $C\in 
\mathcal{K}_{2}$, since $G_{0}\cap Z(GL_{2}(8))=1$. However, this case cannot occur by Proposition \ref{SLSp}. Therefore, $C\trianglelefteq G_{0}$
and since $18$ divides the order of $G_{0}$, it follows that $C\in \mathcal{K%
}_{1}$ and hence $G_{0}\leq N_{M}(C)$. Moreover, since $G_{0}\cap
Z(GL_{2}(8))=1$, it follows that $D_{18}\trianglelefteq G_{0}\leq
D_{18}.Z_{3}$. Therefore, $\lambda =2$ and hence $B$ is a $3$-dimensional $%
GF(2)$-subspace of $V$ by Corollary \ref{p2}. Then $D_{18}\trianglelefteq G_{0}\leq
D_{18}.Z_{3}$ leads to a $2$-$(2^{6},2^{3},2)$
design isomorphic to that of Line 1 in Table \ref{t0} by using \textsf{GAP} \cite{GAP}. This corresponds to case (1) for $q=2^{3}$ and $\lambda =2$.

Assume that $M\cong \Gamma L_{3}(4)$. Then $SL_{3}(4)\nleq G_{0}$ by
Proposition \ref{SLSp}. Since $G_{0}$ acts irreducibly on $V$ and its order
is divisible by $18$, it follows that either $G_{0}$ preserves $V_{3}(2)$,
or $G_{0}\leq \Gamma U_{3}(2)$ or $G_{0}\leq Z_{3}.S_{6}$ by \cite[Tables 8.3--8.4]{BHRD}. The first one is ruled out by Lemma \ref{PP}, whereas if $%
G_{0}\leq \Gamma U_{3}(2)$ then $\mathcal{D}$ is isomorphic to one of the $2$%
-designs listed in Table \ref{t0} by using \textsf{GAP} \cite{GAP}, and these correspond to cases (12)--(13).

Assume that $G_{0}\leq Z_{3}.S_{6}$. Let $S$ be a Sylow $3$-subgroup of $G_{0}$. Since no proper subgroups of $S$ have all
orbits on $V^{\ast }$ of lengths divisible by $3^{2}$, it follows that $S\leq G_{0}$
by Lemma \ref{PP}. Therefore either $G_{0}\leq N_{M}(S)=S:D_{8}$ or $%
Z_{3}.A_{6}\trianglelefteq G_{0}\leq Z_{3}.S_{6}$. Assume that the former
occurs. Note that $N_{M}(S)$ splits $V^{\ast }$ in one orbit of length $%
27$ and two orbits of length $18$. The first one is also a $G_{0}$-orbit as $%
S\leq G_{0}$, each of the remaining two $N_{M}(S)$-orbits is a union of $G_{0}$%
-orbits. Let $B$ be any block incident with $0$. Since $r=9\lambda $ with $%
\lambda =2,4,8$, and since the intersection of $B$ with any of $G_{0}$-orbit
of length different from $27$ is at most $2$ by Lemma \ref{PP}, it follows that
there is $\sigma \in G_{0}$, with $o(\sigma )=3$, such that $\left\vert 
\mathrm{Fix(}\sigma \mathrm{)}\cap B\right\vert =4$. Actually, $\mathrm{Fix(}%
\sigma \mathrm{)}\subset B$. Since $S\trianglelefteq G_{0}$, $r=9\lambda $
with $\lambda =2,4,8$, and the order of $S$ is $27$, it follows that $\left\langle \sigma \right\rangle
\trianglelefteq G_{0,B}$, and hence $G_{0,B}$ preserves $\mathrm{Fix(}\sigma 
\mathrm{)}$. Therefore, $B$ is a $3$-dimensional subspace of $V$ by
Corollary \ref{p2}. However this is impossible by \cite[Theorem 3.3.1]{Go}.
Thus $Z_{3}.A_{6}\trianglelefteq G_{0}\leq Z_{3}.S_{6}$, hence $G_{0}$
splits $V^{\ast }$ into two orbits of length $27$ and $36$ by \cite%
{AtMod}, and if $x_{1}$ and $x_{2}$ are representatives of these orbits and $%
B$ is any block incident with $0$, it follows that $\left\vert B\cap
x_{1}^{G_{0}}\right\vert =3$ and $\left\vert B\cap x_{2}^{G_{0}}\right\vert =4$ by
Lemma \ref{PP}. Thus $\left( G_{0,B}\right) ^{\prime }$ fixes $B$ pointwise,
since $A_{5}\trianglelefteq G_{0,B}\leq S_{5}$ by \cite{AtMod}, hence $%
\left( G_{0,B}\right) ^{\prime }\leq G_{0,y_{i}}$ for $y_{i}\in B\cap
x_{i}^{G}$, $i=1,2$, which is not the case.

Suppose that $i=8$. Then $M\cong Sp_{6}(2)$. The argument used in the proof
of Lemma \ref{sympl} rules out the case $G_{0}=M$. Then $G_{0}$ lies in a
maximal subgroup $K$ of $Sp_{6}(2)$, hence $K$ is one of the groups listed in 
\cite[Tables 8.28--8.29]{BHRD}. Actually, either $K\cong \Gamma L_{2}(8)$ or 
$K\cong SO_{6}^{\pm }(2)$ or $K\cong G_{2}(2)$ since $G_{0}$ acts
irreducibly on $V$.

If $K\cong \Gamma L_{2}(8)$, then the above argument leads to $D_{18}\trianglelefteq G_{0}\leq
D_{18}.Z_{3}$ and to a $2$-$(2^{6},2^{3},2)$ design isomorphic to that of Line 1 in Table \ref{t0}. Hence (1) holds in this case.
 
Assume that $K\cong SO_{6}^{+}(2)$. Then the $K$-orbits on $V^{\ast }$ are
union of $G_{0}$-orbits. Since each $G_{0}$-orbit is of length divisible $%
3^{2}$ by Lemma \ref{PP}, it follows that each $K$-orbit on $V^{\ast }$ is
divisible by $3^{2}$. However this is impossible, since the $K$-orbits have
length $35$ and $28$ by \cite{At}. So, this case is ruled out.

Assume that $K\cong SO_{6}^{-}(2)$. Since $G_{0}$ is irreducible on $V$, it
follows that $\Omega _{6}^{-}(2)\trianglelefteq G_{0}$ by \cite{BHRD},
Tables 8.33--8.34. However, the argument used in the proof of Lemma \ref%
{Omega4no2} excludes this case as well.

Assume that $K\cong G_{2}(2)$. Then either $G_{0}\leq N_{K}(U)=S:SD_{16}$,
where $U$ is a Sylow $3$-subgroup of $K$, or $G_{0}=K^{\prime } \cong PSU_{3}(3)$ or $G_{0}=K$
by \cite{At}. In the first two cases $\mathcal{D}$ is isomorphic to one of
the $2$-designs listed in Table \ref{t0} again by using \textsf{GAP} \cite%
{GAP}, which correspond to case (1), (12) and (13). If $G_{0}=K$ we may proceed as in Proposition %
\ref{NChar} thus obtaining that $\mathcal{D}$
isomorphic to the $2$-$(2^{6},2^{3},2^{3})$ design in Line 5 of Table \ref{t0}, which corresponds to (10) for $q=2$. This completes the proof
\end{proof}

\end{document}